\documentclass[%
    oneside,
    10pt,
]{article}

\usepackage[%
left=1.25in,
right=1.25in,
bottom=1in,
top=1in,
footskip=0.5in,
]{geometry}

\usepackage[usenames,dvipsnames]{color}
\usepackage{amsmath}
\usepackage{amsthm}
\usepackage{amssymb,bbm}
\usepackage[mathcal]{euscript}

\usepackage{cite}
\usepackage{hyperref}
\usepackage{enumitem}
\usepackage{tikz}
\usepackage{mathtools}
\usepackage[raggedright]{titlesec}
\titleformat{\section}{\normalfont\large\bfseries}{\thesection}{1em}{}
\titleformat{\subsection}{\normalfont\bfseries}{\thesubsection}{1em}{}
\titleformat{\subsubsection}{\normalfont\itshape}{\thesubsubsection}{1em}{}

\definecolor{darkgreen}{rgb}{0,0.5,0}

\usepackage{hyperref} 
\makeatletter
\hypersetup{%
	colorlinks	=true,
	linkcolor	=blue,%
	citecolor	=darkgreen,%
	urlcolor	=blue
}

\renewcommand{\theequation}{\thesection.\arabic{equation}}
\newtheorem{theorem}{Theorem}[section]
\newtheorem{remark}[theorem]{Remark}

\newtheorem{corollary}[theorem]{Corollary}
\newtheorem{definition}[theorem]{Definition}
\newtheorem{proposition}[theorem]{Proposition}
\newtheorem{lemma}[theorem]{Lemma}

\numberwithin{equation}{section}

\let\non\nonumber

\def\Lip{Lip\-schitz}
\def\Holder{H\"older}

\def\lhs{left-hand side}
\def\rhs{right-hand side}

\def\multibold #1{\def\arg{#1}%
  \ifx\arg\pto \let\next\relax
  \else
  \def\next{\expandafter
    \def\csname #1#1#1\endcsname{{\bf #1}}%
    \multibold}%
  \fi \next}

\def\pto{.}

\def\multical #1{\def\arg{#1}%
  \ifx\arg\pto \let\next\relax
  \else
  \def\next{\expandafter
    \def\csname cal#1\endcsname{{\cal #1}}%
    \multical}%
  \fi \next}

\def\multimathop #1 {\def\arg{#1}%
  \ifx\arg\pto \let\next\relax
  \else
  \def\next{\expandafter
    \def\csname #1\endcsname{\mathop{\rm #1}\nolimits}%
    \multimathop}%
  \fi \next}

\multibold
qwertyuiopasdfghjklzxcvbnmQWERTYUIOPASDFGHJKLZXCVBNM.

\multical
QWERTYUIOPASDFGHJKLZXCVBNM.

\multimathop
diag dist div dom mean meas sign supp .

\def\R{{\mathbb{R}}}
\def\N{{\mathbb{N}}}

\def\<#1>{\mathopen\langle #1\mathclose\rangle}
\def\norma #1{\left\| #1\right\|}
\newcommand{\vertiii}[1]{{\left\vert\kern-0.25ex\left\vert\kern-0.25ex\left\vert #1 \right\vert\kern-0.25ex\right\vert\kern-0.25ex\right\vert}}

\def\I2 #1{\int_{Q_t}|{#1}|^2}
\def\IT2 #1{\int_{Q_t^T}|{#1}|^2}
\def\IO2 #1{\norma{{#1(t)}}^2}
\def\IOT2 #1{\norma{{#1(T)}}^2}
\def\ov #1{{\overline{#1}}}

\def\ioT {\int_0^T}

\def\iO{\int_\Omega}
\def\iG{\int_\Gamma}

\def\delt{\partial_t}
\def\deln{\partial_{\bf n}}
\def\checkmmode #1{\relax\ifmmode\hbox{#1}\else{#1}\fi}

\def\aeO{\checkmmode{a.e.\ in~$\Omega$}}
\def\aeQ{\checkmmode{a.e.\ in~$Q$}}
\def\aeS{\checkmmode{a.e.\ on~$\Sigma$}}
\def\aeG{\checkmmode{a.e.\ on~$\Gamma$}}

\def\genspazio #1#2#3#4#5{#1^{#2}(#5,#4;#3)}
\def\spazio #1#2#3{\genspazio {#1}{#2}{#3}T0}
\def\spaziok #1#2#3{\genspazio {#1}{#2}{#3}{T_k}0}
\def\L {\spazio L}
\def\H {\spazio H}
\def\C #1#2{C^{#1}([0,T];#2)}

\def\Lk {\spaziok L}

\def\Lx #1{L^{#1}(\Omega)}
\def\LGx #1{L^{#1}(\Gamma)}
\def\Hx #1{H^{#1}(\Omega)}

\let\eps\varepsilon

\def\b{\beta}	
\def\d{\delta}

\def\th{\theta}

\def\ph{\varphi}

\def\ps{\psi}

\def\cd{C_{\d}}

\def\Kn{{K_n}}

\newcommand{\D}{\mathbf{D}}

\def\bv{{\boldsymbol{v}}}

\def\VV{{\bf{V}}}
\def\bv{{\boldsymbol{v}}}

\def\HH{{\bf{H}}}
\def\VV{{\bf{V}}}
\def\LL{{\bf{L}}}

\def\Vp{{V^*}}

\def\HG{{H}_{\Gamma}}
\def\VG{{V}_{\Gamma}}

\def\HHG{\mathbf {H}_{\Gamma}}

\def\BH{{\mathbb{H}}}

\def\Vp{{V^*}}

\newcommand{\CH}{{\cal H}}

\newcommand{\CV}{{\cal V}}
\newcommand{\Hs}{{\bf H}_{\sigma, \nnn}}
\newcommand{\Vs}{{\bf V}_{\sigma, \nnn}}

\newcommand{\ds}{{\rm ds}}

\def\w{{\boldsymbol w}}

\let\phi\varphi
\let\eps\varepsilon
\def\etaG{{\eta_\Gamma}}
\def\zetaG{{\zeta_\Gamma}}

\newcommand\emb{{\hookrightarrow}}
\renewcommand{\hat}{\widehat}

\def\mz{m_0}

\def\mgz{m_{\Gamma0}}

\def\Accorpa #1#2 #3 {\gdef #1{\eqref{#2}--\eqref{#3}}%
\wlog{}\wlog{\string #1 -> #2 - #3}\wlog{}}
  
\def\dd#1 {{\rm d#1}}

\def\step #1 \par{\medskip\noindent{\it #1.}\quad}

\newcommand{\abs}[1]{\left| #1 \right|}
\newcommand{\bigabs}[1]{\big| #1 \big|}
\newcommand{\norm}[1]{\| #1 \|}

\newcommand{\bigang}[2]{ \big< #1 , #2  \big>}
\newcommand{\scp}[2]{ \left( #1 , #2  \right)}

\newcommand{\meano}[1]{{\langle #1 \rangle}_{\Omega}}
\newcommand{\meang}[1]{{\langle #1 \rangle}_{\Gamma}}

\def\betaG{{\beta_\Gamma}}
\def\beG{{\beta_{\Gamma,\eps}}}
\def\beps{{\beta_\eps}}

\begin{document}


\title{\bfseries\Large
    Two-phase flows through porous media\\
    described by a Cahn--Hilliard--Brinkman model \\
    with dynamic boundary conditions 
}

\author{\normalsize 
    Pierluigi Colli \footnotemark[1] 
    \and\normalsize  
    Patrik Knopf \footnotemark[2] 
    \and\normalsize  
    Giulio Schimperna \footnotemark[1]  
    \and\normalsize  
    Andrea Signori \footnotemark[3] 
}

\renewcommand{\thefootnote}{\fnsymbol{footnote}}

\footnotetext[1]{
Dipartimento di Matematica ``F. Casorati'', Universit\`a di Pavia, and Research Associate at the IMATI~--~C.N.R. Pavia, 27100 Pavia, Italy 
({\tt \href{mailto:pierluigi.colli@unipv.it}{pierluigi.colli@unipv.it}, 
\href{mailto:giusch04@unipv.it}{giusch04@unipv.it}}).}

\footnotetext[2]{Department of Mathematics, University of Regensburg, 93040 Germany 
({\tt\href{mailto:patrik.knopf@ur.de}{patrik.knopf@ur.de}}).}

\footnotetext[3]{Dipartimento di Ma\-te\-ma\-ti\-ca, Politecnico di Milano, 20133 Milano, Italy \\ ({\tt \href{mailto:andrea.signori@polimi.it}{andrea.signori@polimi.it}}), Alexander von Humboldt Research Fellow.}

\date{}

\maketitle

\renewcommand*{\thefootnote}{\arabic{footnote}}

\begin{center}
	\scriptsize
	{
		\textit{This is a preprint version of the paper. Please cite as:} \\  
		P.~Colli, P.~Knopf, G.~Schimperna, A.~Signori 
        \textit{J. Evol. Equ.} \textbf{24}:85, 55 pp. (2024), \\
		\url{https://doi.org/10.1007/s00028-024-00999-y}
	}
\end{center}

\vspace{2ex}

%
%

\begin{small}
\begin{center}
    \textbf{Abstract}
\end{center}
We investigate a new diffuse-interface model that describes creeping two-phase flows (i.e., flows exhibiting a low Reynolds number), especially flows that permeate a porous medium. The system of equations consists of a Brinkman equation for the volume averaged velocity field and a convective Cahn--Hilliard equation with dynamic boundary conditions for the phase-field, which describes the location of the two fluids within the domain. The dynamic boundary conditions are incorporated to model the interaction of the fluids with the wall of the container more precisely. In particular, they allow for a dynamic evolution of the contact angle between the interface separating the fluids and the boundary, and for a convection-induced motion of the corresponding contact line. For our model, we first prove the existence of global-in-time weak solutions in the case where regular potentials are used in the Cahn--Hilliard subsystem. In this case, we can further show the uniqueness of the weak solution under suitable additional assumptions. We further prove the existence of weak solutions in the case of singular potentials. Therefore, we regularize such singular potentials by a Moreau--Yosida approximation, such that the results for regular potentials can be applied, and eventually pass to the limit in this approximation scheme.
\\[1ex]

\textbf{Keywords:} two-phase flows, porous media, Cahn--Hilliard equation, Brinkman equation, dynamic boundary conditions, bulk-surface interaction.
\\[1ex]	
\textbf{Mathematics Subject Classification:}
		35D30, 
	    35K35, 
	    35K86, 
		35B65, 
		76D03, 
		76T06  
\end{small}

\vspace{2.5ex}

\begin{small}
	\setcounter{tocdepth}{2}
	\setlength\parskip{0ex}
	\hypersetup{linkcolor=black}
	\tableofcontents
\end{small}


\setlength\parskip{1ex}
\allowdisplaybreaks

\section{Introduction}
\label{SEC:INTRO}
\setcounter{equation}{0}

The mathematical study of two-phase flows is an important topic in many areas of applied science such as engineering, chemistry and biology. To predict the motion of two fluids, it is crucial to understand how the interface between these fluids evolves.
To provide a mathematical description of this interface, two fundamental methods have been developed: the \textit{sharp-interface approach} and the \textit{diffuse-interface approach}. 
In the former, the interface is represented as a hypersurface evolving in the surrounding domain. The occurring quantities (e.g., the velocity fields) are then described by a free boundary problem.
In the latter, the fluids are represented by a \textit{phase-field} function which is expected to attain values close to $1$ in the region where the first fluid is present, and close to $-1$ in the region where the second fluid is located. However, unlike in sharp-interface models, this phase-field does not jump between the values $1$ and $-1$ but exhibits a continuous transition between these values in a thin tubular neighborhood around the boundary between the fluids. This tubular neighborhood is referred to as the \textit{diffuse interface} and its thickness is proportional to a small parameter $\epsilon>0$.
For a comparison of sharp-interface methods and diffuse-interface methods, we refer to \cite{du2020phase, AbelsGarckeReview, giga2017variational,pruss2016moving}.
We point out that, even though the sharp-interface and the diffuse-interface approach are conceptually different, they can, in general, be related by the \textit{sharp-interface limit} in which a parameter related to the thickness of the diffuse interface is sent to zero.

In the context of diffuse-interface models, such models in which the phase-field is described by a Cahn--Hilliard type equation have become particularly popular. One of the most widely used models for describing the motion of two viscous, incompressible fluids with matched (constant) densities is the \textit{Model H}.
It was first proposed in \cite{HH} and was later rigorously derived in \cite{GurtinPolignoneVinals}. The PDE system consists of an incompressible Navier--Stokes equation coupled with a convective Cahn--Hilliard equation.
In terms of mathematical analysis, the Model H was investigated quite extensively, see, e.g., in \cite{Abels2009, boyer1999, GalGrasselli2010, GMT2019}.
Further generalizations of this model can be found in \cite{lowengrub1998quasi, boyer2002theoretical, ding2007diffuse, shen2013mass,shokrpour2018diffuse,giga2017variational,heida2012development,freistuhler2017phase}.

One drawback of the Model H is that it can merely be used to describe the situation in which the fluids have the same individual density. To overcome this issue, a thermodynamically consistent diffuse-interface model for incompressible two-phase flows with possibly \textit{unmatched} densities was derived in the seminal work \cite{AGG}. This model is usually referred to as the \textit{AGG model}.
Concerning mathematical analysis of this model, we refer the reader to
\cite{abels2013existence,abels2013incompressible,AbelsWeber2021, giorgini2021well, giorgini2022-3D,Abels2023}.
The connection between the AGG model and the two-phase Navier--Stokes free boundary problem is explained in \cite{AGG,AbelsGarckeReview}.

Even though the AGG model and the Model H subject to the classical boundary conditions (i.e., a \textit{no-slip} boundary condition for the velocity field and homogeneous Neumann boundary conditions for the convective Cahn--Hilliard equation) are well suited to describe the motion of the fluids in the the interior of the considered domain, they still inherit some limitations from the underlying (convective) Cahn--Hilliard system with homogeneous Neumann boundary conditions. 
The main limitations are:
\begin{enumerate}[label={{(L\arabic*)}}, ref={{(L\arabic*)}}]
    \item\label{L1} The homogeneous Neumann condition on the phase-field enforces that the diffuse interface always intersects the boundary at a perfect ninety degree contact angle. This will not be fulfilled in many applications. In general, the contact angle might even change dynamically over the course of time.
    \item\label{L2} The no-slip boundary condition on the velocity field makes the model not very suitable to describe general moving contact line phenomena. As the trace of the velocity field at the boundary is fixed to be identically zero, any motion of the contact line of the diffuse interface can be caused only by diffusive but not directly by convective effects.
    \item\label{L3} The mass of the fluids in the bulk is conserved. Therefore, a transfer of material between the bulk and the boundary (caused, e.g., by absorption processes or chemical reactions) cannot be described. 
\end{enumerate}
A more detailed discussion can be found in \cite{Giorgini-Knopf}.

To overcome the aforementioned restrictions (L1) and (L2), a class of dynamic boundary conditions was derived in \cite{Qian-Wang-Sheng}. It involves an Allen--Cahn type dynamic boundary condition for the phase-field coupled to a generalized Navier-slip boundary condition for the velocity field. The Model H subject to this boundary condition was analyzed in \cite{GGM2016} whereas the AGG model subject to this boundary condition was investigated in \cite{GGW2019}.

Recently, a thermodynamically consistent generalization of the AGG model subject to another class of dynamic boundary conditions was derived in \cite{Giorgini-Knopf}. Here, the boundary condition consists of a convective surface Cahn--Hilliard equation and a generalized Navier-slip boundary condition. Compared to the models studied in \cite{Qian-Wang-Sheng,GGW2019,GGM2016}, the Navier--Stokes--Cahn--Hilliard system introduced in \cite{Giorgini-Knopf} provides more regularity for the boundary quantities and therefore, the uniqueness of weak solutions can be established in two space dimensions. Moreover, due to the fourth-order dynamic boundary condition of Cahn--Hilliard type, the model in \cite{Giorgini-Knopf} is not only capable of overcoming the limitations (L1) and (L2) but also (L3). 

In the present paper, we particularly want to consider the situation of {\it creeping flows} meaning that the Reynolds number 
\begin{align*}
    \mathrm{Re} = \frac{u L}{\nu}
\end{align*}
associated with the fluids is very small ($\mathrm{Re} \ll 1$). This occurs if the flow speed $u$ and/or the characteristic length $L$ of the flow are small compared to the kinematic viscosity $\nu$. In this situation, it is not necessary to describe the time evolution of the velocity field by the full Navier--Stokes equation. Since advective inertial forces are small compared to viscous forces, the material derivative can be neglected. This leads to the \textit{Stokes equation}. If a creeping flow through a porous medium is to be considered, an additional term accounting for the permeability needs to be included. The velocity field is then usually determined by \textit{Darcy's law} or the \textit{Brinkman equation}. For a derivation of these velocity equations via homogenization techniques, we refer, for instance, to \cite{Allaire1991,Hoefer2023,Masmoudi2002,Giunti2019,Feireisl2016} and the references therein.

Therefore, in this paper, we study the following \textit{Cahn--Hilliard--Brinkman} system with dynamic boundary conditions:
\begin{subequations}
\label{CHB}
\begin{alignat}{2}
	\label{eq:1}
	& - \div (2 \nu(\phi) \D\bv)
	+ \lambda(\phi) \bv 	
    + \nabla  p
	= \mu \nabla \phi  
	\quad && \text{in $Q$,}
	\\
	\label{eq:2}
	& \div (\bv) =0 
	\quad && \text{in $Q$,}
	\\[1ex]
	\label{eq:3}
	& \delt  \phi
	+ \div (\phi \bv)
	- \div(M_\Omega(\phi) \nabla \mu )
	=0
	\quad && \text{in $Q$,}
	\\
	\label{eq:4}
	& \mu = -\epsilon \Delta \phi + \tfrac{1}{\epsilon} F'(\phi)	
	\quad && \text{in $Q$,}
	\\
	\label{eq:5}
	& \delt \psi 
	+ \div_\Gamma (\psi \bv)
	- \div_\Gamma (M_\Gamma(\psi) \nabla_\Gamma\theta )
	=0
	\quad && \text{on $\Sigma$,}
	\\
	\label{eq:6}
	& \theta = - \epsilon_\Gamma \Delta_\Gamma \psi + \tfrac{1}{\epsilon_\Gamma} G'(\psi) + \deln \phi	
	\quad && \text{on $\Sigma$,}
	\\[1ex]
	\label{eq:7}
	& K \deln \phi = \psi - \phi
	\quad && \text{on $\Sigma$,}
	\\
	\label{eq:8}
	& M_\Omega(\phi) \deln \mu = \bv \cdot \nnn = 0
	\quad && \text{on $\Sigma$,}
	\\
	\label{eq:9}
	& [2 \nu(\phi) \,\D\bv\, \nnn + \gamma(\psi) \bv]_\tau = - [\psi \nabla_\Gamma \theta]_\tau
	\quad && \text{on $\Sigma$,}
	\\[1ex]
	\label{eq:10}
	& \phi(0)= \phi_0
	\quad && \text{in $\Omega$,}
	\\
	\label{eq:11}
	& \psi(0)=\psi_0
	\quad && \text{on $\Gamma$.}
\end{alignat}
\end{subequations}
It can be regarded as a variant of the Navier--Stokes--Cahn--Hilliard model derived in \cite{Giorgini-Knopf}, where the incompressible Navier--Stokes equation is replaced by the incompressible \textit{Brinkman/Stokes equation} $\big($\eqref{eq:1},\eqref{eq:2}$\big)$ to describe the situation of a creeping two-phase flow. 
Such Cahn--Hilliard--Brinkman models for creeping two-phase flows through porous media have applications in petroleum engineering, especially concerning oil recovery from hydrocarbon reservoirs (see, e.g., \cite{Zhang2002,Feng2018}). They are also commonly used to describe two-phase mixtures in Hele--Shaw cells (see, e.g., \cite{Dede2018,Giorgini2018}). Furthermore, Cahn--Hilliard--Brinkman models are used in mathematical biology, especially in the context of tumor growth (see, e.g., \cite{Ebenbeck2019,Knopf2022} and the references therein).
For derivations of Cahn--Hilliard--Brinkman type models, we refer to \cite{Rohde,Schreyer,Chen2019} and the references cited therein.

In system \eqref{CHB}, $\Omega \subset \R^d$ with $d \in \{2,3\}$ is a bounded domain with boundary $\Gamma:=\partial \Omega$, $T>0$ is a given final time, and for brevity, the notation $Q:= \Omega \times (0,T)$ and $\Sigma= \Gamma \times (0,T)$ is used. 
The vector-valued function $\bv:Q\to\R^d$ stands for the volume averaged velocity field associated with the fluid mixture
and 
\begin{align*}
    \D\bv = \frac 12 \big(\nabla \bv + (\nabla \bv)^\top\big)    
\end{align*}
denotes the associated {\it symmetric gradient}. For the sake of simplicity, we will usually refrain from writing the trace operator. For instance, we will often simply write $\bv$ instead of $\bv|_\Gamma$. Nevertheless, in some instances, where confusion may arise, we will employ the explicit notation. 
For any vector field $\w$ on the boundary, we will write $\w_\tau:=\w -(\w \cdot \nnn)\nnn$ to denote its tangential component.
The symbols $\nabla_\Gamma$ and $\div_\Gamma$ denote the surface gradient and the surface divergence, respectively, and $\Delta_\Gamma$ stands for the Laplace--Beltrami operator. 

The functions $\phi:Q\to\R$ and $\mu:Q\to\R$ denote the phase-field and the chemical potential in the bulk, respectively, whereas $\psi:\Sigma\to\R$ and $\theta:\Sigma\to\R$ represent the phase-field and the chemical potential on the boundary, respectively. Furthermore, the parameters $\epsilon$ and $\epsilon_\Gamma$ are positive real numbers which are related to the thickness of the diffuse interface in the bulk and on the surface, respectively. Therefore, these constants are usually chosen to be quite small. However, as their values do not have any impact on the mathematical analysis, we will simply fix $\epsilon = \epsilon_\Gamma = 1$ in the subsequent sections. The phase-fields $\phi$ and $\psi$ are directly related by the coupling condition \eqref{eq:7}, where $K$ is a given nonnegative constant. 

From a physical point of view, the \textit{kinematic viscosity} $\nu(\phi)$ and the \textit{permeability coefficient} $\lambda(\phi)$ in the Brink\-man/Stokes equation \eqref{eq:1} can be expressed as
\begin{align*}
    \nu(\phi) = \frac{\eta(\phi)}{\varrho}
    \quad\text{and}\quad
    \lambda(\phi) = \frac{\eta(\phi)}{\varkappa},
\end{align*}
where $\eta(\phi)>0$ denotes the \textit{dynamic viscosity}, and the constants $\varrho>0$ and $\varkappa>0$ stand for the \textit{porosity} and the \textit{intrinsic permeability} of the porous medium, respectively. 
If both $\nu(\varphi)$ and $\lambda(\varphi)$ are positive, \eqref{eq:1} is the \textit{(quasi-stationary) Brinkman equation} which describes the flow through a porous medium.
It is equipped with the incompressibility condition \eqref{eq:2} and the \textit{inhomogeneous Navier-slip boundary condition} \eqref{eq:9} in which the \textit{slip parameter} $\gamma(\psi)\ge 0$ may depend on the surface phase-field $\psi$.
However, if the porosity $\varkappa$ is large compared to the viscosity $\eta(\phi)$, the function $\lambda(\varphi)$ is very small and can be neglected. 
In this case we enter the \textit{Stokes regime}, where no porous media is considered (or the effects of the porous medium are at least negligible). In the formal limit $\varkappa \to \infty$ or $\lambda(\varphi) \to 0$, \eqref{eq:1} degenerates to the \textit{Stokes equation}. In our analysis, we will be able to handle the Brinkman case ($\nu(\phi)>0$ and $\lambda(\phi)>0$) and the Stokes case ($\nu(\phi)>0$ and $\lambda(\phi)\equiv 0$) simultaneously. On the other hand, if $\lambda(\phi)$ remains positive and the porosity $\varrho$ is large compared to the dynamic viscosity $\eta(\phi)$ such that $\nu(\phi)$ can be neglected, \eqref{eq:1} degenerates to \textit{Darcy's law}. However, we are not able to handle this case in terms of mathematical analysis as due to the absence of spatial derivatives of the velocity field in \eqref{eq:1}, we would not obtain enough regularity to define the trace of $\bv$ on the boundary in a reasonable manner.

The functions $F'$ and $G'$ are the derivatives of double-well potentials $F$ and $G$, respectively.
Especially in applications related to materials science, a physically relevant choice for $F$ and/or $G$
is the \textit{logarithmic potential}, which is also referred to as the \textit{Flory--Huggins potential}. It is given by
\begin{align}
    \label{DEF:F:LOG}
    W_\mathrm{log}(s) = \frac{\Theta}{2} \big[(1+s)\,\ln(1+s) +(1-s)\,\ln(1-s)\big]+\frac{\Theta_c}{2}(1-s^2),
\end{align}
for all $s\in (-1,1)$. Here, ${\Theta>0}$ is the absolute temperature of the mixture, and $\Theta_c$ is a critical temperature such that phase separation will occur in case $0<\Theta<\Theta_c$. The logarithmic potential is classified as a singular potential since its derivative $F'$ diverges to $\pm\infty$ when its argument approaches $\pm 1$. It is often approximated by a \textit{polynomial double-well potential}
\begin{align}
    \label{DEF:F:REG}
    W_\mathrm{pol}(s) = \frac{\alpha}{4}(s^2-1)^2
    \quad\text{for all $s\in (-1,1)$},
\end{align}
where $\alpha>0$ is a suitable constant. 
Another very commonly used singular potential is the \textit{double-obstacle potential}, which is given by
\begin{align}
    \label{DEF:F:DOB}
    W_\mathrm{obst}(s)=
    \begin{cases}
    \frac{1}{2}(1-s^2)&\text{if $|s|\leq 1$,}\\
    +\infty&\text{else}.
    \end{cases}
\end{align}

In the case $K=0$, the convective bulk-surface Cahn--Hilliard subsystem \eqref{eq:3}--\eqref{eq:8} is a special case of the one introduced in \cite{Giorgini-Knopf} since for the chemical potential $\mu$, a homogeneous Neumann type boundary condition is imposed in \eqref{eq:8}. This corresponds to the choice $L=\infty$ in \cite{Giorgini-Knopf}. Therefore, by system \eqref{CHB}, we describe a situation where no transfer of material between bulk and boundary occurs. However, its is important that due to the boundary conditions \eqref{eq:5}--\eqref{eq:9}, the model \eqref{CHB} allows for dynamic changes of the contact angle as well as for a convection-induced motion of the contact line. This means that the limitations \ref{L1} and \ref{L2} explained above can be overcome. It is worth mentioning that this setup of dynamic boundary conditions for the Cahn--Hilliard equation (without coupling to a velocity equation) was originally derived in \cite{Liu2019} by the Energetic Variational Approach. This system was further investigated in \cite{Colli2020,Garcke2020,Miranville2020,Knopf2020}. For similar works on the Cahn--Hilliard equation with Cahn--Hilliard type dynamic boundary conditions, we refer to \cite{Knopf2021,Knopf2021a,Garcke2022,GMS,Colli2015,Colli2022,Colli2022a,Wu2022}.

In contrast to the model introduced in \cite{Giorgini-Knopf}, the phase-fields $\phi$ and $\psi$ are not just coupled by the trace relation $\phi\vert_\Sigma = \psi$ on $\Sigma$, but by the more general Robin type coupling condition $K\deln\phi = \psi-\phi$ with $K\ge 0$ (see \eqref{eq:7}). This also includes the trace relation via the choice $K=0$. The coupling condition \eqref{eq:7} was first used in \cite{Colli2019} for an Allen--Cahn type dynamic boundary condition, and later in \cite{Knopf2020} for a Cahn--Hilliard type dynamic boundary condition. In particular, it was rigorously shown in \cite{Knopf2020} that the Dirichlet type coupling condition $\phi\vert_\Sigma = \psi$ on $\Sigma$ can be recovered in the asymptotic limit $K\to 0$. From a physical point of view, the boundary condition \eqref{eq:7} with $K>0$ makes sense if the materials on the boundary may be different from those in the bulk. For instance, this might be the case if the materials on the boundary are transformed by chemical reactions.

An important property of the system \eqref{CHB} (for any $K\ge 0$) is its thermodynamic consistency with respect to the free energy functional
\begin{align}
    \label{DEF:E}
	E_K(\phi,\psi) 
	&:= \iO \Big(\frac \epsilon 2 |\nabla \phi|^2 + \frac 1 \epsilon F(\phi) \Big)
	   + \iG \Big(\frac {\epsilon_\Gamma}2 |\nabla_\Gamma \psi|^2 + \frac 1 {\epsilon_\Gamma} G(\psi) \Big)
    \notag\\
	&\quad + \frac{\sigma(K)}2 \iG (\psi-\phi)^2,
\end{align}
where 
$\sigma(K)=K^{-1}$ if $K>0$ and $\sigma(K)=0$ if $K=0$.
This means that sufficiently regular solutions of \eqref{CHB} satisfy the \textit{energy dissipation law}
\begin{align*}
\begin{aligned}
	\frac{\mathrm d}{\mathrm dt} E_K(\phi,\psi) 
	&= - \iO \lambda(\phi) |\bv|^2
    - \iO M_\Omega(\phi)|\nabla \mu|^2
	- \iG M_\Gamma(\psi)|\nabla_\Gamma \theta|^2
    \\
	&\quad - 2 \iO \nu(\phi)|\D\bv|^2
	- \iG \gamma(\psi) |\bv|^2,
\end{aligned}
\end{align*}
on $[0,T]$, where all the terms on the right-hand side are non-positive and can be interpreted as the dissipation rate. Compared to the model in \cite{Giorgini-Knopf}, the additional term $\iO \lambda(\varphi)\abs{\bv}^2$ arises due to dissipative effects caused by the porous medium.

As mentioned above, due to the usage of the no-mass-flux condition $M_\Omega(\phi) \deln \mu=0$ on $\Sigma$ (see \eqref{eq:8}), we do not describe any transfer of material between bulk and surface. This entails that the bulk mass and the surface mass are conserved separately, i.e., sufficiently regular solutions satisfy the mass conservation laws
\begin{align}
	\label{mass:cons:bulk}
	\frac 1 {|\Omega|} \iO  \phi(t) & =  \frac 1 {|\Omega|}\iO \phi_0 =:\mz
	\\
	\frac 1 {|\Gamma|} \iG  \psi(t) & = \frac 1 {|\Gamma|}\iG \psi_0 =: \mgz
	\label{mass:cons:bd}
\end{align}
for all $t\in [0,T]$.

\paragraph{Goals and novelties of this paper.}
After collecting some notations, assumptions, preliminaries and important tools, we introduce our main results in Section~\ref{SEC:NOT:RES}. Let us now briefly discuss the results with emphasis on the objectives of this work.

\begin{enumerate}[label={\textbf{(\Roman*)}}]
    \item Our first main goal is to establish the weak well-posedness to system \eqref{CHB} in the case of regular potentials $F$ and $G$ for all choices $K\ge 0$. 
    
    Here, the case $K=0$ is the most delicate one, because then the boundary conditions \eqref{eq:7} and $\eqref{eq:8}_1$ are very difficult to combine. The reason is that the Dirichlet type boundary condition \eqref{eq:7} already fixes one degree of freedom in the space of test functions, whereas the Neumann boundary condition $\eqref{eq:8}_1$ does not. Therefore, it is not possible to directly construct a weak solution by a Faedo--Galerkin approach.
    
    In previous works in the literature (e.g., \cite{Garcke2020,Knopf2020,Knopf2021a}), where the bulk-surface Cahn--Hilliard system was considered without any velocity field, such issues of unmatched boundary conditions could be overcome by exploiting the gradient flow structure of the system and employing a minimizing movement scheme for the construction of a weak solution. However, as in our case the bulk-surface Cahn--Hilliard system is coupled to a velocity equation, the whole system is not a gradient flow anymore and therefore, a minimizing movement technique is not applicable.

    To overcome this issue, our strategy is to first prove the existence of weak solutions to \eqref{CHB} for any $K>0$. In this case, \eqref{eq:7} is a Robin type boundary condition, which, in contrast to the Dirichlet type boundary condition associated with $K=0$,  does not fix any degree of freedom in the space of test functions. Therefore,  it can be combined very well with the Neumann boundary condition $\eqref{eq:8}_1$. Having a weak solution to \eqref{CHB} for any $K>0$ at hand, we can finally prove the existence a weak solution to \eqref{CHB} with $K=0$ by passing to the asymptotic limit $K\to 0$.

    The existence of a weak solution to \eqref{CHB} for any $K>0$ is stated in Theorem~\ref{THM:WEAK}. For the proof, we use a semi-Galerkin scheme, where only the phase-fields and the chemical potentials are discretized using a Galerkin scheme, but the corresponding velocity field is obtained by directly solving the Brinkman subsystem on the continuous level. For a Cahn--Hilliard--Brinkman model without dynamic boundary conditions, such an approach had already been employed in \cite{Ebenbeck2019}. Based on this ansatz, a sequence of approximate solutions can be constructed by means of the Cauchy--Peano theorem and eventually, after deriving suitable uniform bounds, we show that this sequence converges to a weak solution of \eqref{CHB} with $K>0$. A posteriori, we establish higher regularity of the phase-field functions provided that the domain is sufficiently smooth. The corresponding proof can be found in Subsection~\ref{SUBSEC:EXREG}.

    In Theorem~\ref{THM:Kto0}, we construct a weak solution to \eqref{CHB} in the case $K=0$ by using the result of Theorem~\ref{THM:WEAK}, and passing to the limit $K\to 0$ on the level of weak solutions. A posteriori, as in the case $K>0$, higher regularity of the phase-fields can be shown if the domain is sufficiently regular.
    The proof is carried out in Subsection~\ref{SUBSEC:EXREG:0}.

    \item The second main goal is to prove the uniqueness as well as the continuous dependence on the initial data of the weak solution to \eqref{CHB} with regular potentials $F$ and $G$ in all cases $K\ge 0$, provided that the mobility functions and the viscosity function are constant and the domain is sufficiently regular. 
    
    The result is stated in Theorem~\ref{THM:UQ} and the proof is presented in Subsection~\ref{SUBSEC:UNIREG}. To prove the assertion, we adapt the uniqueness proof devised in \cite{Giorgini-Knopf} for a Navier--Stokes--Cahn--Hilliard model to our situation. We point out that the proof in \cite{Giorgini-Knopf} only worked in two dimensions, whereas our new proof for the system \eqref{CHB} is also valid in three dimensions. Moreover, our new proof requires a certain bulk-surface chain rule for time derivatives that is established in Proposition~\ref{PROP:A} in the Appendix.

    \item Our third main result is to prove the existence of a weak solution to system \eqref{CHB} in the case of possibly singular potentials $F$ and $G$ for all choices $K\ge 0$. 

    The result is stated in Theorem~\ref{THM:WEAK:SING} and the proof is presented in Section~\ref{SEC:POTSING}. The idea of the proof is to approximate the convex parts of the singular potentials $F$ and $G$ by regular potentials $F_\eps$ and $G_\eps$ (with $\eps>0$) that are constructed via a Moreau--Yosida regularization. For the approximate regular potentials, the existence of a corresponding weak solutions is already known from Theorem~\ref{THM:WEAK}. The most technical part of the proof of Theorem~\ref{THM:WEAK:SING} is to derive uniform bounds on the potential terms involving $F_\eps'$ and $G_\eps'$. This requires a certain condition on the singular potentials, namely that $G'$ dominates $F'$ in a suitable way (cf~\ref{ass:2:pot:dominance}).
    Eventually, by passing to the limit $\eps\to 0$, the claim of Theorem~\ref{THM:WEAK:SING} is established.
\end{enumerate}

\section{Preliminaries and main results}
\label{SEC:NOT:RES}
\setcounter{equation}{0}

\subsection{Notation}
Throughout the manuscript, $\Omega$ is a bounded domain in~$\R^d$, $d \in \{2,3\}$,
with \Lip\ boundary $\Gamma:=\partial\Omega$
and $\nnn$ is the associated outward unit normal vector field.
We write $|\Omega|$ and $|\Gamma|$ to denote the Lebesgue measure of $\Omega$ and the Hausdorff measure of $\Gamma$, respectively.
For any given Banach space $X$, we denote its norm by $\norma{\,\cdot\,}_X$,
its dual space by $X^*$ and the duality pairing between $X^*$ and $X$ by $\<\cdot,\cdot>_X$.
Besides, if $X$ is a Hilbert space, we write $(\cdot,\cdot)_X$ to denote the corresponding inner product.
For every $1 \leq p \leq \infty$, $k \geq 0$ and $s>0$, the standard Lebesgue spaces, Sobolev--Slobodeckij spaces and Sobolev spaces defined on $\Omega$ are denoted by $L^p(\Omega)$, $W^{k,p}(\Omega)$ and $H^s(\Omega)$, and their standard norms are denoted by $\norma{\,\cdot\,}_{L^p(\Omega)}$ $\norma{\,\cdot\,}_{W^{k,p}(\Omega)}$ and $\norma{\,\cdot\,}_{H^s(\Omega)}$, respectively. It is well known that the spaces $H^0(\Omega)= L^2(\Omega)$ and $H^k(\Omega)= W^{k,2}(\Omega)$ for all $k\in\N$ can be identified, and these spaces are Hilbert spaces.
The Lebesgue spaces, Sobolev--Slobodeckij spaces and Sobolev spaces on the boundary $\Gamma$ are defined analogously. For more details, we refer to \cite{Triebel,Triebel2}.
We usually utilize bold letters to represent spaces for vector- or matrix-valued functions. For example, we denote $\LL^p(\Omega)$ instead of writing $L^p(\Omega;\R^d)$ or $L^p(\Omega;\R^{d\times d})$, and so on.
Moreover, for any Banach spaces $X$ and $Y$, their intersection $X\cap Y$ is also a Banach space subject to the norm
\begin{align*}
	\norma{v}_{X\cap Y}:= \norma{v}_X + \norma{v}_Y, 
	\quad 
	v \in X \cap Y.
\end{align*}
As some spaces will appear very frequently, we introduce the following shortcuts:
\begin{align*}
    &\begin{aligned}
      &H  := \Lx2 , \quad  
      &&H_\Gamma := L^2(\Gamma),
      \quad 
      &&V := \Hx1,   
      \quad
      &&\VG := H^1(\Gamma),
      \\ 
      &\HH := L^2(\Omega;\R^d),
      \quad 
      &&\HHG :=  L^2(\Gamma;\R^d),
      \quad 
      &&\VV := H^1(\Omega;\R^d),
      \quad
      &&\BH := L^2(\Omega;\R^{d\times d}),
    \end{aligned}
    \\
    &\Hs := \big\{\w \in \HH : \div (\w) =0 \;\text{in $\Omega$},\; \w\vert_\Gamma \cdot \nnn =0 \;\text{on $\Gamma$} \big\},
    \quad
	\Vs := 
    \VV \cap \Hs,
    \\
    &
    \VV_\nnn := \big\{\w \in \VV : \w\vert_\Gamma \cdot \nnn =0 \;\text{on $\Gamma$} \big\},
    \quad W_\nnn := \big\{w \in H^2(\Omega) : \deln w =0 \;\text{on $\Gamma$} \big\}.
\end{align*}
We point out that in the definition of $\Hs$, the relation $\div (\w) =0$ in $\Omega$ is to be understood in the sense of distributions. This already implies $\w\vert_\Gamma \cdot \nnn \in H^{-1/2}(\Gamma)$, and therefore, the relation $\w\vert_\Gamma \cdot \nnn =0$ on $\Gamma$ is well-defined.
As $\Hs$ and $\Vs$ are closed linear subspaces of the Hilbert spaces $\HH$ and $\VV$, respectively, they are also Hilbert spaces.
We further introduce the bulk-surface product spaces
\begin{align*}
    \CH &:= H \times H_\Gamma,
    \quad 
    \CV := V \times \VG,
    \\
    \CV_K &:= 
    \begin{cases}
        \CV &\text{if $K>0$},\\
        \{ (w,w_\Gamma) \in \CV: w_\Gamma = {w}\vert_{\Gamma} \,\, \text{on $\Gamma$}\}
        &\text{if $K=0$},
    \end{cases}
\end{align*}
and endow them with the corresponding inner products 
\begin{alignat*}{2}
    \big((v,v_\Gamma),(w,w_\Gamma)\big)_\CH &:= (v,w)_{H} + {(v_\Gamma,w_\Gamma)}_{H_\Gamma}
    &&\quad
    \text{for all $(v,v_\Gamma),(w,w_\Gamma) \in{ \CH}$,}
    \\
    \big((v,v_\Gamma),(w,w_\Gamma)\big)_\CV &:= (v,w)_{V} + {(v_\Gamma,w_\Gamma)}_{\VG}
    &&\quad
    \text{for all $(v,v_\Gamma),(w,w_\Gamma) \in{ \cal V}$,}
\end{alignat*}
so that $\CH$, $\CV$ and $\CV_K$ are Hilbert spaces.
It is straightforward to check that
\begin{align*}
    &\CV^* = V^* \times V_\Gamma^*,
    \\
    &\CV_K \subset \CV
    \quad\text{and}\quad
    \CV^* \subset \CV_K^* \quad\text{if $K= 0$},
    \\
    &\CV_K = \CV
    \quad\text{and}\quad
    \CV^* = \CV_K^* \quad\text{if $K> 0$}.
\end{align*}%
For any $v\in\Vp$ and $v_\Gamma \in \VG^*$, we define the generalized mean values by
\begin{align}
  \<v>_\Omega := \frac 1{|\Omega|} \, \< v , 1 >_{V}\,,
  \quad 
  \<v_\Gamma>_\Gamma := \frac 1{|\Gamma|} \, \< v_\Gamma, 1 >_{\VG}\,,
  \label{def:mean}
\end{align}
where $1$ represents the constant function assuming value $1$ in $\Omega$ and  on $\Gamma$, respectively.
To introduce a weak formulation of \eqref{CHB}, it will be useful to define the function 
\begin{align}
    \label{DEF:SIGMA}
    \sigma: [0,\infty) \to [0,\infty),\quad
    \sigma(r) = 
    \begin{cases}
        \frac 1r &\text{if $r>0$},\\
        0 &\text{if $r=0$}
    \end{cases}
\end{align}
to handle the cases $K>0$ and $K=0$ simultaneously.

\subsection{General assumptions}

\begin{enumerate}[label={\bf (A\arabic{*})}, ref={\bf (A\arabic{*})}]
	\item \label{ass:1:dom}
	The set $\Omega \subset \R^d$ with $ d \in \{2, 3\}$ is a bounded Lipschitz domain.
	
	\item \label{ass:3:mobility}
	The mobility functions $M_\Omega : \R \to \R$ and $M_\Gamma : \R \to \R$ are continuous, bounded and uniformly positive.
	This means that there exist positive constants $M_1$, $M_2$, $M_{\Gamma,1}$ and $M_{\Gamma,2}$ such that
	\begin{align*}
	    0 < M_1 \leq M_\Omega(r) \leq M_2,
	    \quad 
	    0 < M_{\Gamma,1}\leq M_\Gamma(r) \leq M_{\Gamma,2}
	    \quad \text{for all $r \in \R.$}
	\end{align*}
	
	\item \label{ass:4:viscosity} 
    The viscosity function $\nu : \R \to \R$, the permeability function $\lambda : \R \to \R$, and the friction parameter $\gamma : \R \to \R$ are continuous and nonnegative.
    There further exist constants $\nu_1,\nu_2,\lambda_2,\gamma_2>0$ as well as $\gamma_1\ge 0$ such that
	\begin{align*}
	    0 < \nu_1 \leq \nu(r) \leq \nu_2,
	    \quad 
        0 \le \lambda(r) \leq \lambda_2,
        \quad
	    0 \le \gamma_1 \le \gamma(r) \leq \gamma_2
	    \quad \text{for all $r \in \R.$}
	\end{align*}
	Moreover, we assume that one of the following conditions holds:
    \begin{enumerate}[label={\bf (A3.\arabic{*})}, ref={\bf (A3.\arabic{*})}]
        \item\label{ass:4:viscosity:1} It holds $\gamma_1>0$.
        \item \label{ass:4:viscosity:2}
        The domain $\Omega$ has the following property:
        \begin{equation}
        \label{COND:SYMM}
            \begin{cases}
                \text{If $d=2$, $\Omega$ is not a circle.}\\
                \text{If $d=3$, $\Omega$ is not rotationally symmetric.}
            \end{cases}
        \end{equation}
    \end{enumerate}
\end{enumerate}

\subsection{Preliminaries}

In our mathematical analysis, we will need the following versions of \textit{Korn's inequality}. 
\begin{lemma} \label{LEM:KORN}
Suppose that \ref{ass:1:dom} holds.
\begin{enumerate}[label=\textnormal{(\alph*)}, ref=\textnormal{(\alph*)}, topsep=0ex]
    \item\label{Korn1}
    There exists a constant $C_\mathrm{Korn}$ depending only on $\Omega$ such that 
    \begin{equation}
        \norm{\nabla\bv}_{\BH}
        \leq C_\mathrm{Korn} \left( \norm{\D \bv}_{\BH} 
        + \norm{\bv}_{\HHG} \right)
        \quad 
        \text{for all $\bv \in \VV$.}
    \end{equation}
    \item\label{Korn2}
    If $\Omega$ additionally fulfills the condition \eqref{COND:SYMM}, then there exists a constant $C_\mathrm{Korn}^*$ depending only on $\Omega$ such that
    \begin{equation}
        \norm{\nabla\bv}_{\BH}
        \leq C_\mathrm{Korn}^* \norm{\D \bv}_{\BH} 
        \quad 
        \text{for all $\bv \in \VV$.}
    \end{equation}
\end{enumerate}
\end{lemma}

\medskip

\noindent Both \ref{Korn1} and \ref{Korn2} can be found in \cite[Appendix]{Abels2012}. In the three-dimensional case, \ref{Korn2} had already been established before in \cite[Theorem~3.5]{Necas}. As illustrated in \cite{Abels2012}, the two-dimensional version can be proved analogously.

Moreover, besides the standard Poincar\'e--Wirtinger inequality in $\Omega$, we need a Poincar\'e type inequality on $\Gamma$.

\begin{lemma} \label{LEM:POIN}
    Suppose that \ref{ass:1:dom} holds. Then, there exists a constant $C_\mathrm{P} \ge 0$ depending only on $\Omega$ such that 
    \begin{equation}
        \label{EST:POIN}
        \norm{v - \meang{v}}_{\HG} \le C_P \norm{\nabla v}_{\HHG} 
        \quad 
        \text{for all $v \in \VG$.}
        \end{equation}
\end{lemma}

\noindent
This result follows directly from the bulk-surface Poincar\'e inequality established in \cite[Lemma~A.1]{knopf-liu} (with the function in the bulk being chosen as $u\equiv 0$ and the parameters being chosen as $K=2$, $\alpha=\beta=1$).

\medskip

We further recall the following result on interpolation between Sobolev spaces:
\begin{lemma}
    \label{LEM:INT}
    Let $U\subset\R^m$ with $m \in \mathbb N$ be a bounded Lipschitz domain, and suppose that $\theta\in (0,1)$ and $r,s_0,s_1 \in\R$ satisfy
    $$r = (1-\theta) s_0 + \theta s_1.$$ 
    We further assume that $U$ is of class $C^\ell$ with an integer $\ell\ge \max\{s_0,s_1\}$.
    Then, there exist positive constants $C_U$ and $C_{\partial U}$ depending only on $U$, $r$, $s_0$, $s_1$ and $\theta$ such that the following interpolation inequalities hold:
    \begin{align}
        \label{LEM:INT:1}
        \norm{f}_{H^r(U)} 
        &\le C_U \norm{f}_{H^{s_0}(U)}^{1-\theta}
            \norm{f}_{H^{s_1}(U)}^{\theta},
        \\
        \label{LEM:INT:2}
        \norm{f}_{H^r(\partial U)} 
        &\le C_{\partial U} \norm{f}_{H^{s_0}(\partial U)}^{1-\theta}
            \norm{f}_{H^{s_1}(\partial U)}^{\theta}.
    \end{align}
\end{lemma}
\noindent Inequality \eqref{LEM:INT:1} follows from an interpolation result shown in \cite[Sec~4.3.1, Theorem~1 and Remark~1]{Triebel}, whereas \eqref{LEM:INT:2} follows from an interpolation result presented in \cite[Sec~7.4.5, Remark~2]{Triebel2}.

\subsection{The Cahn--Hilliard--Brinkman system with regular potentials}

First, we present our mathematical results for system \eqref{CHB} in the case of regular double-well potentials $F$ and $G$. As mentioned above, we simply set $\epsilon=\epsilon_\Gamma=1$, as the exact values of these interface parameters do not have any impact on the mathematical analysis as long as they are positive.

\subsubsection{Assumptions for regular potentials}

\begin{enumerate}[label={\bf (R\arabic*)}, ref={\bf (R\arabic*)}]
	\item  \label{ass:pot:reg:1}
	The potentials $F:\R\to[0,\infty)$ and $G:\R\to[0,\infty)$ are continuously differentiable,
	and there exist exponents $p,q \in \R$ with
    \begin{align}
    \label{EXP}
        p\in\begin{cases}
            [2,\infty) &\text{if $d=2$,}\\
            [2,6] &\text{if $d=3$,}
        \end{cases}
        \quad\text{and}\quad
        q\in[2,\infty)
    \end{align}
    as well as constants $c_{F'},c_{G'}\ge 0$ such that 
    \begin{align}
    \label{Ass:F'}
    \abs{F'(r)} &\le
        c_{F'}(1+\abs{r}^{p-1}),
    \\
    \label{Ass:G'}
    \abs{G'(r)} &\le  
    c_{G'}(1+\abs{r}^{q-1})
    \end{align}
    for all $r\in\R$.
	This implies that there exist constants $c_F,c_G\ge 0$ such that $F$ and $G$ fulfill the growth conditions
	\begin{align}
		\label{GR:F}
        F(r) &\le c_{F}(1+\abs{r}^p), 
        \\
		\label{GR:G}
        \quad 
        G(r) &\le c_{G}(1+\abs{r}^q)
	\end{align}
	for all $r\in\R$.
    \item  \label{ass:pot:reg:2}
    In addition to \ref{ass:pot:reg:1}, $F$ and $G$ are twice continuously differentiable and there exist
    constants $c_{F''},c_{G''}\ge 0$ such that 
    \begin{align}
    \label{Ass:F''}
    \abs{F''(r)} &\le
        c_{F''}(1+\abs{r}^{p-2}),
    \\
    \label{Ass:G''}
    \abs{G''(r)} &\le  
    c_{G''}(1+\abs{r}^{q-2})
    \end{align}
    for all $r\in\R$, where $p$ and $q$ are the exponents introduced in \eqref{EXP}.   
\end{enumerate}

\subsubsection{Definition of weak solutions for regular potentials}

\begin{definition}
    \label{DEF:WS:REG}
    Let $K\ge 0$ be arbitrary.
    Suppose that \ref{ass:1:dom}--\ref{ass:4:viscosity} and \ref{ass:pot:reg:1} are fulfilled and let $(\phi_0,\psi_0)\in \CV_K$ be any initial data.
    A quintuplet $(\bv, \phi, \mu, \psi, \theta)$ is called a weak solution of the Cahn--Hilliard--Brinkman system \eqref{CHB} if the following conditions are fulfilled:
    \begin{enumerate}[label=\textnormal{(\roman*)}, ref=\textnormal{(\roman*)}]
        \item \label{DEF:WS:REG:i} The functions $\bv$, $\phi$, $\mu$, $\psi$ and $\theta$ have the regularity
        \begin{align*}
    	\bv & \in \L2 {\Vs}, \quad \bv|_\Gamma \in \L2 {\HH_\Gamma},
    	\\
    	(\phi,\psi) & \in \H1 { \CV^*} \cap C^0([0,T];\CH)\cap \L\infty {\CV_K} ,
        \\
        (\mu,\theta) & \in \L2 \CV.
        \end{align*}
         
        \item \label{DEF:WS:REG:ii} The variational formulation 
        \begin{subequations}
        \label{wf}
        \begin{align}
            \label{wf:1}
        	& \begin{aligned}
        	    &2 \iO \nu(\phi) \D\bv : \D\w 
        	       + \iO \lambda(\phi) \bv \cdot \w 
        	       + \iG  \gamma(\psi)\bv \cdot \w 
                \\
                &\quad =
                - \iO \phi \nabla \mu \cdot \w
                - \iG \psi \nabla_\Gamma\th \cdot \w,
            \end{aligned}        	
             \\
          	\label{wf:2}
            & \<\delt \phi, \zeta>_{V}
            - \iO \phi \bv \cdot \nabla \zeta
            + \iO M_\Omega(\phi)\nabla \mu \cdot \nabla \zeta
            =0 ,
            \\
            \label{wf:3}
            & \<\delt \psi, \zetaG>_{\VG}
            - \iG \psi \bv \cdot \nabla_\Gamma \zetaG
            + \iG M_\Gamma(\psi)\nabla_\Gamma \theta \cdot \nabla_\Gamma \zetaG
            =0,
            \\
            & 
            \label{wf:4}
            \begin{aligned}
            \iO \mu \eta 
            + \iG \theta \etaG
            &= \iO \nabla \phi \cdot \nabla \eta 
            + \iO F'(\phi) \eta
            + \iG \nabla_\Gamma \psi \cdot \nabla_\Gamma \etaG 
            \\
            &\quad + \iG G'(\psi) \etaG
            + \sigma(K) \iG (\psi - \phi) (\etaG - \eta)
            \end{aligned}
        \end{align}
        \end{subequations}
        holds a.e.~in $[0,T]$ for all $\w \in \Vs$, $\zeta\in V$, $\zetaG\in V_\Gamma$ and $(\eta,\etaG) \in \CV_K$.
        
        \item \label{DEF:WS:REG:iii} The initial conditions are satisfied in the follwing sense:
        \begin{align*}
        	\phi(0)=\phi_0 \quad \aeO,
        	\quad 
        	\psi(0)=\psi_0 \quad \aeG.
        \end{align*}
        \item \label{DEF:WS:REG:iv} The weak energy dissipation law
        \begin{align}
            \label{DISSLAW:REG}
            &E_K\big(\phi(t),\psi(t)\big) 
            + 2\int_0^t\iO \nu(\phi)|\D\bv|^2
            + \int_0^t\iO \lambda(\phi) |\bv|^2
            + \int_0^t\iG \gamma(\psi) |\bv|^2
            \notag\\
            &\qquad 
            + \int_0^t\iO M_\Omega(\phi)|\nabla \mu|^2
	        + \int_0^t\iG M_\Gamma(\psi)|\nabla_\Gamma \theta|^2
            \notag\\[1ex]
            &\quad
            \le E_K(\phi_0,\psi_0)
        \end{align}
        holds for all $t\in[0,T]$. 
    \end{enumerate}
\end{definition}

\begin{remark}\label{REM:PRESS}
    We point out that the pressure $p$ does not appear in the weak formulation \eqref{wf:1} as the test functions are chosen to be divergence-free. However, provided that $\Omega$ is connected, the pressure can be reconstructed in the following way.
    Suppose that $(\bv, \phi, \mu, \psi, \theta)$ is a weak solution in the sense of Definition~\ref{DEF:WS:REG}. We define
    \begin{align*}
        \mathcal{F} : W_\nnn \to \R, \quad 
        \mathcal{F}(\Phi) &= \iO 2\nu(\phi) \D\bv : \D \nabla \Phi 
            + \iO \lambda(\phi) \bv \cdot \nabla \Phi
        	+ \iG  \gamma(\psi)\bv \cdot \nabla \Phi
        \\
        &\quad 
            + \iO \phi \nabla \mu \cdot \nabla \Phi
            + \iG \psi \nabla_\Gamma\th \cdot \nabla \Phi \,.
    \end{align*}
    Note that in the integrals over $\Gamma$, we can actually replace $\nabla \Phi$ by $\nabla_\Gamma \Phi$ as $\bv$ and $\nabla_\Gamma \theta$ are tangential vector fields, i.e., their normal component is zero.
    In view of the regularities in Definition~\ref{DEF:WS:REG}\ref{DEF:WS:REG:i}, it is straightforward to check that $\mathcal{F} \in W_\nnn^*$. Hence, according to \cite[Lemma~3.6.1]{AbelsHabil}, there exists a unique function $p\in L^2(\Omega)$ such that
    \begin{align}
        \label{wf:p}
        \iO p \, \Delta \Phi = \mathcal{F}(\Phi) 
        \quad\text{for all $\Phi \in W_\nnn$.}
    \end{align}
    Let now $\hat\w \in \VV_\nnn$ be arbitrary. As $\Omega$ is additionally assumed to be connected, there exists a Leray decomposition ${\hat\w =\w + \nabla\Phi}$ with $\w \in \Vs$ and $\Phi\in W_\nnn$ (see, e.g., \cite[Theorem~IV.3.5.]{boyer_book}). In particular, we thus have $\div \hat \w = \div(\nabla \Phi) = \Delta \Phi$. Hence, combining \eqref{wf:1} and \eqref{wf:p}, we conclude that the variational formulation
    \begin{align*}
        &2 \iO \nu(\phi) \D\bv : \D\hat\w 
            - \iO p \div \hat\w
            + \iO \lambda(\phi) \bv \cdot \hat\w 
            + \iG  \gamma(\psi)\bv \cdot \hat\w 
        \\
        &\quad =
            - \iO \phi \nabla \mu \cdot \hat\w
            - \iG \psi \nabla_\Gamma\th \cdot \hat\w        
    \end{align*}
    holds for all $\hat\w \in \VV_\nnn$. 
    We have thus reconstructed the pressure $p\in L^2(\Omega)$.
\end{remark}

\subsubsection{Existence of a weak solution in the case \texorpdfstring{$K>0$}{}}

We first show the existence of a weak solution to the Cahn--Hilliard--Brinkman system \eqref{CHB} in the case $K>0$.

\begin{theorem}
\label{THM:WEAK}
Let $K>0$ be arbitrary.
Suppose that \ref{ass:1:dom}--\ref{ass:4:viscosity} and \ref{ass:pot:reg:1} are fulfilled and let $(\phi_0,\psi_0)\in \CV_K$ be any initial data.
Then, the Cahn--Hilliard--Brinkman system \eqref{CHB} possesses at least one weak solution $(\bv, \phi, \mu, \psi, \theta)$ in the sense of Definition~\ref{DEF:WS:REG}, which further satisfies $(\mu,\theta)  \in \L4 \CH$.

Let us now assume that the domain $\Omega$ is of class $C^\ell$ with $\ell \in \{2,3\}$. If $d=3$, we further assume $p\le 4$, and if $\ell=3$, we further assume that \ref{ass:pot:reg:2} holds.
Then, we have the additional regularities
\begin{subequations}
\label{higher:reg:phi}
\begin{align}
    \label{reg:L4H2}
    (\phi,\psi) &\in
    L^4\big(0,T; H^2(\Omega) \times H^2(\Gamma) \big) \qquad \text{in case $\ell\in\{2,3\}$}, \\
    \label{reg:L2H3}
    (\phi,\psi) &\in
    L^2\big(0,T; H^3(\Omega) \times H^3(\Gamma) \big) \qquad \text{in case $\ell=3$},
\end{align}
\end{subequations}
and the equations \eqref{eq:4}, \eqref{eq:6} and \eqref{eq:7} are fulfilled in the strong sense, that is, almost everywhere in $Q$ and on $\Sigma$, respectively. In the case $\ell=3$, we further have
\begin{align}\label{CH1:reg}
    (\phi,\psi) \in C^0\big([0,T]; \CV \big).
\end{align}
\end{theorem}

\subsubsection{The limit \texorpdfstring{$K\to0$}{} and existence of a weak solution in the case \texorpdfstring{$K=0$}{}}

We now investigate the limit $K \to 0$ in which the boundary condition \eqref{eq:7} formally tends to the Dirichlet condition $\psi= \phi\vert_{\Gamma}$ almost everywhere on $\Sigma$.
In the following theorem, we send $K \to 0$ in system \eqref{CHB} to prove the existence of a weak solution to \eqref{CHB} in the case $K=0$, and we further specify the convergence properties of this asymptotic limit.

\begin{theorem}
\label{THM:Kto0}
Suppose that \ref{ass:1:dom}--\ref{ass:4:viscosity} and \ref{ass:pot:reg:1} are fulfilled and let $(\phi_0,\psi_0)\in \CV_0$ be any initial data.
Let $(\Kn)_{n\in\N}$ be a sequence of positive real numbers such that $\Kn\to 0$ as $n\to \infty$.
For any $n\in\N$, let $(\bv^\Kn, \phi^\Kn, \mu^\Kn, \psi^\Kn, \theta^\Kn)$ denote any weak solution corresponding to $\Kn>0$ in the sense of Definition~\ref{DEF:WS:REG}.
Then, there exists a quintuplet of functions $(\bv^0, \phi^0,  \mu^0, \psi^0,  \theta^0)$ with $\phi^0\vert_\Gamma = \psi^0$ {\aeS} such that for any $s\in[0,1)$,
\begin{subequations}
\label{CONV:K}
\begin{alignat}{2}
    \label{CONV:V:K}
    \bv^\Kn &\to \bv^0
    &&\quad\text{weakly in $\L2 \Vs$},
    \\
    \label{CONV:V:K:bd}
    \bv^\Kn|_\Gamma &\to \bv^0|_\Gamma
    &&\quad\text{weakly in $\L2 \HHG$},
    \\
    \label{CONV:PHI:K}
    \phi^\Kn &\to \phi^0
    &&\quad\text{weakly-$^*$ in $\L\infty{V}$, weakly in $\H1 {V^*}$}, \notag\\
    &&&\qquad\text{strongly in $C^0([0,T];H^s(\Omega))$, and \aeQ},
    \\
    \label{CONV:PSI:K}
    \psi^\Kn &\to \psi^0
    &&\quad\text{weakly-$^*$ in $\L\infty {\VG}$, weakly in $\H1 {\VG^*}$}, \notag\\
    &&&\qquad\text{strongly in $C^0([0,T];H^s(\Gamma))$, and \aeS},
    \\
    \label{CONV:MU:K}
    \mu^\Kn &\to \mu^0
    &&\quad\text{weakly in $\L2 V $},
    \\
    \label{CONV:THETA:K}
    \theta^\Kn &\to \theta^0
    &&\quad\text{weakly in $\L2 \VG $},
    \\
    \label{CONV:PP:K}
    \phi^\Kn\vert_{\Gamma} - \psi^\Kn &\to 0
    &&\quad\text{strongly in $\L\infty \HG$, and \aeS}, 
\end{alignat}
\end{subequations}
as $n\to\infty$ along a non-relabeled subsequence.
Moreover, the limit $(\bv^0, \phi^0, \mu^0, \psi^0, \theta^0)$ is a weak solution of the Cahn--Hilliard--Brinkman model \eqref{CHB} in the sense of Definition~\ref{DEF:WS:REG} with $K=0$. 

Let us now assume that the domain $\Omega$ is of class $C^\ell$ with $\ell \in \{2,3\}$. If $d=3$, we further assume $p\le 4$, and if $\ell=3$, we further assume that \ref{ass:pot:reg:2} holds. Then, we have the additional regularities
\begin{subequations}
\label{higher:reg:phi:0}
\begin{align}
    \label{reg:L4H2:0}
    (\phi^0,\psi^0) &\in
    L^2\big(0,T; H^2(\Omega) \times H^2(\Gamma) \big) \qquad \text{in case $\ell\in\{2,3\}$}, \\
    \label{reg:L2H3:0}
    (\phi^0,\psi^0) &\in
    L^2\big(0,T; H^3(\Omega) \times H^3(\Gamma) \big) \qquad \text{in case $\ell=3$},
\end{align}
\end{subequations}
and the equations \eqref{eq:4} and \eqref{eq:6} are fulfilled in the strong sense. Moreover, in the case $\ell=3$, we further have
\begin{align}
    \label{CH1:reg*}
    (\phi^0,\psi^0) &\in C^0\big([0,T]; \CV_0 \big) \cap L^4\big(0,T; H^2(\Omega) \times H^2(\Gamma) \big), 
    \\
    \label{CH1:reg*:2}
    (\mu^0,\theta^0) & \in L^4\big(0,T; \CH \big).
\end{align}
\end{theorem}

\subsubsection{Stability and uniqueness of the weak solution in the general case \texorpdfstring{$K\ge0$}{}}

In the case of regular potentials, constant mobilities, and a constant viscosity, we are able to prove the uniqueness of the weak solutions established in Theorem~\ref{THM:WEAK} provided that the following assumption on the potentials $F$ and $G$ holds.
\begin{enumerate}[label={\bf (R*)}, ref={\bf (R*)}]
	\item  \label{ass:pot:reg*}
	The potentials $F:\R\to[0,\infty)$ and $G:\R\to[0,\infty)$ are three times continuously differentiable, 
	and there exist exponents $p,q\in\R$ with
    \begin{align}
    \label{EXP*}
        p\in \begin{cases}
            [3,\infty) &\text{if $d=2$,}\\
            [3,4] &\text{if $d=3$,}
        \end{cases}
        \qquad\text{and}\qquad
        q\in[3,\infty),
    \end{align}%
    as well as constants $c_{F^{(3)}}, c_{G^{(3)}} \ge 0$ such that the third-order derivatives satisfy
    \begin{align}
    \label{Ass:F'''}
    \abs{F^{(3)}(r)} &\le
        c_{F^{(3)}}(1+\abs{r}^{p-3}),
    \\
    \label{Ass:G'''}
    \abs{G^{(3)}(r)} &\le  
    c_{G^{(3)}}(1+\abs{r}^{q-3})
    \end{align}
    for all $r\in\R$.
\end{enumerate}
We point out that \ref{ass:pot:reg*} implies \ref{ass:pot:reg:1} and \ref{ass:pot:reg:2} with $p$ and $q$ being chosen as in \eqref{EXP*}. 

\medskip

\begin{theorem}
    \label{THM:UQ}
    Suppose that \ref{ass:1:dom} and \ref{ass:pot:reg*} are fulfilled with $\Omega$ being of class $C^3$, and let $K\ge 0$ be arbitrary.  
    In addition to \ref{ass:3:mobility} and \ref{ass:4:viscosity}, we further assume that $\gamma$ and $\lambda$ are Lipschitz continuous and that the functions $\nu$, $M_\Omega$ and $M_\Gamma$ reduce to positive constants denoted by the same symbols.
    For any $i\in\{1,2\}$, let $(\phi_{0,i},\psi_{0,i}) \in \CV_K$ be any pair of initial data, and let $(\bv_i,\phi_i, \mu_i, \psi_i, \theta_i)$ be a corresponding weak solution in the sense of Definition~\ref{DEF:WS:REG}. 
    Then, the stability estimate
    \begin{align} 
        & \non \norma{\bv_1-\bv_2}_{\L2 {\VV}}
        + \norma{\phi_1-\phi_2}_{\L\infty V}
        + \norma{\mu_1-\mu_2}_{\L2 V}
        + \norma{\psi_1-\psi_2}_{\L\infty \VG}
        \\ & \quad 
        + \norma{\th_1-\th_2}_{\L2 \VG}
        \leq C_S \big(\norma{\phi_{0,1}-\phi_{0,2}}_V + \norma{\psi_{0,1}-\psi_{0,2}}_{\VG} \big)
        \label{est:CD}
    \end{align}
    holds for a constant $C_S\ge 0$ depending only on $K$, $\Omega$, $T$, the initial data and the constants introduced in \ref{ass:1:dom}--\ref{ass:4:viscosity} and \ref{ass:pot:reg*}. In particular, choosing $(\phi_{0,1},\psi_{0,1})=(\phi_{0,2},\psi_{0,2})$, this entails the uniqueness of the corresponding weak solution.
\end{theorem}

\subsection{The Cahn--Hilliard--Brinkman system with singular potentials}

We now consider the system \eqref{CHB} for a general class of singular potentials. For those, we manage to establish just existence of weak solutions due to the lower regularity at disposal. Recall that $\epsilon=\epsilon_\Gamma=1$ as mentioned above.

\subsubsection{Assumptions for singular potentials}

For the potentials $F$ and $G$, we now make the following assumptions.

\begin{enumerate}[label={\bf (S\arabic*)}, ref={\bf (S\arabic*)}]
	\item \label{ass:1:pot}
	The potentials $F$ and $G$ can be decomposed as $F=\hat \b + \hat \pi$ and $G={\hat \b}_\Gamma + {\hat \pi}_\Gamma$.
 
	Here, $\hat \b,{\hat \b}_\Gamma: \R \to [0,  \infty]$ are lower semicontinuous and convex functions with $\hat \b(0)=0$ and 
    ${\hat \b}_\Gamma(0)=0$.
	For brevity, we define 
	\begin{align*}
	    \b := \partial \hat \b
	    \quad\text{and}\quad
	    \b_\Gamma := \partial {\hat \b_\Gamma},
	\end{align*}
    where $\partial$ indicates the subdifferential of the respective function.
	Moreover, we suppose that  $\hat \pi,{\hat \pi}_\Gamma \in C^1(\R)$ with Lipschitz continuous derivatives $\pi:=\hat \pi'$ and $\pi_\Gamma:={\hat \pi}_\Gamma'$.

    We point out that $\beta$ and $\betaG$ are maximal monotone graphs in $\R \times \R$ whose effective domains are denoted by $D(\b)$ and $D(\b_\Gamma)$, respectively.
	In particular, as $0$ is a minimum point of both $\hat \b$ and ${\hat \b}_\Gamma$, it turns out that $0 \in \b(0)$ and $0 \in {\b_\Gamma}(0)$.
    Finally, we denote by $\beta^\circ$ the \textit{minimal section} of the graph $\beta$, which is defined as
    \begin{align*}
        \beta^\circ(r) \coloneqq \Big\{ r^\ast \in D(\beta) : \abs{r^\ast} = \min_{s\in \beta(r)} \abs{s} \Big\}
        \qquad\text{for all $r\in D(\beta)$}
    \end{align*}
    (see, e.g., \cite{brezis}).
    The same definition applies to $\beta^\circ_\Gamma$ for~$\betaG$.

	\item  \label{ass:2:pot:dominance}
	We also assume the growth condition 
    \begin{align}   
    \lim_{r \to +\infty} \frac{\hat \b(r)}{|r|^{2}} = +\infty \,. \label{pier2} 
    \end{align}
	Moreover, we demand $D(\b_\Gamma) \subseteq D(\b)$, and postulate that the boundary graph dominates the bulk graph in the following sense:
    \begin{align}
    	\label{domination}
    	\exists \, \kappa_1, \kappa_2 >0: \quad |\b^\circ(r)| \leq \kappa_1 |\b_\Gamma^\circ(r)| + \kappa_2
    	\quad \text{for every $r \in D(\b_\Gamma)$.}
    \end{align}
    Here, $\b^\circ $ and $\b_\Gamma^\circ$ are the minimal sections introduced in \ref{ass:1:pot}.
\end{enumerate}

Note that all the examples of potentials given in \eqref{DEF:F:LOG}--\eqref{DEF:F:DOB} fulfill the assumptions \ref{ass:1:pot} and \ref{ass:2:pot:dominance}, provided that the boundary potential dominates the one in the bulk as demanded in \eqref{domination}.
In particular, the only scenario where a singular and a regular potential may coexist is the case in which the boundary potential is the singular one.
This assumption has first been made in \cite{CaCo} and was used afterwards in several contributions in the literature, see, e.g., \cite{CF6, CGS_dom, CGS2018a,CGS2018b,CGS2018c, CGS2017, Colli2022, Colli2022a, Colli2020, Colli2015}. However, in some other works such as \cite{GMS2009,GMS2010} different compatibility conditions were assumed.

\subsubsection{Definition of  weak solutions for singular potentials}

\begin{definition}
    \label{DEF:WS:SING}
    Let $K\ge 0$ be arbitrary.
    Suppose that \ref{ass:1:dom}--\ref{ass:4:viscosity}, \ref{ass:1:pot} and \ref{ass:2:pot:dominance} are fulfilled and let $(\phi_0,\psi_0)\in \CV_K$ be any initial data satisfying
    \begin{subequations}
    \label{initial:data:weak}
    \begin{alignat}{2}
	\label{initial:data:weak:omega}
	&\hat\b(\phi_0) &\in \Lx1,
	\quad 
	&\mz:=\<\phi_0>_\Omega \in {\rm int}( D(\b)),
	\\ 	
	\label{initial:data:weak:gamma}
	&\hat\b_\Gamma(\psi_0) &\in L^1(\Gamma),
	\quad 
	&\mgz:=\<\psi_0>_\Gamma\in {\rm int}( D(\b_\Gamma)).
    \end{alignat}
    \end{subequations}
    Then, $(\bv, \phi, \xi, \mu, \psi, \xi_\Gamma, \theta)$ is called a weak solution of the Cahn--Hilliard--Brinkman system \eqref{CHB} if
    the following conditions are fulfilled:
    \begin{enumerate}[label=\textnormal{(\roman*)}, ref=\textnormal{(\roman*)}]
        \item The functions $\bv$, $\phi$, $\xi$, $\mu$, $\psi$, $\xi_\Gamma$ and $\theta$ have the regularity
        \begin{alignat*}{3}
        	\bv & \in \L2 {\Vs} ,
         \quad \bv|_\Gamma \in \L2 {\HH_\Gamma},
        	\\
         	(\phi,\psi) & \in \H1 {\CV^*} \cap \C0 {{\CH}}\cap \L\infty {\CV_K},
        	\quad 
            \\ 	
            (\xi,\xi_\Gamma) &\in \L2 \CH,
            \\
            (\mu,\theta) & \in  \L2 \CV .
        \end{alignat*}
        
        \item The variational formulation 
        \begin{subequations}
        \label{wf*}
        \begin{align}
            \label{wf:1*}
        	& \begin{aligned}
        	& 2 \iO \nu(\phi) \D\bv : \D\w 
        	+ \iO \lambda(\phi) \bv \cdot \w 
        	+ \iG  \gamma(\psi)\bv \cdot \w 
            \\
            &\quad =
            - \iO \phi \nabla \mu \cdot \w
            - \iG \psi \nabla_\Gamma\th \cdot \w,
        	\end{aligned}
             \\
            \label{wf:2*}
            & \<\delt \phi, \zeta>_{V}
             - \iO \phi \bv \cdot \nabla \zeta
             + \iO M_\Omega(\phi)\nabla \mu \cdot \nabla \zeta 
             =0 ,
             \\
            \label{wf:3*}
            & \<\delt \psi, \zetaG>_{\VG}
             - \iG \psi \bv \cdot \nabla_\Gamma \zetaG
             + \iG M_\Gamma(\psi)\nabla_\Gamma \theta \cdot \nabla_\Gamma \zetaG
             =0,
             \\ 
            \label{wf:4*}
            &\begin{aligned}
            \iO \mu \eta 
            + \iG \theta \etaG
            &= \iO \nabla \phi \cdot \nabla \eta 
            + \iO \xi\eta  + \iO \pi(\phi)\eta 
            + \iG \nabla_\Gamma \psi \cdot \nabla_\Gamma \etaG 
            \\
            &\quad
            + \iG \xi_\Gamma \etaG + \iG \pi_\Gamma(\psi) \etaG
            + \sigma(K) \iG (\psi - \phi) (\etaG - \eta)
            \end{aligned}
        \end{align}
        \end{subequations}
        holds a.e.~in $[0,T]$ for all $\w \in \Vs$, $\zeta\in V$, $\zetaG\in V_\Gamma$ and $(\eta,\etaG) \in \CV_K$, where
        \begin{align*}
        	\xi \in \b(\phi) \quad \aeQ,
        	\quad 
        	\xi_\Gamma \in \b_\Gamma(\psi) \quad \aeS.
        \end{align*}
        \item The initial conditions are satisfied in the follwing sense:
        \begin{align*}
        	\phi(0)=\phi_0 \quad \aeO,
        	\quad 
        	\psi(0)=\psi_0 \quad \aeG.
        \end{align*}
        \item The weak energy dissipation law
        \begin{align*}
            &E_K\big(\phi(t),\psi(t)\big) 
            + 2\int_0^t\iO \nu(\phi)|\D\bv|^2
            + \int_0^t\iO \lambda(\phi) |\bv|^2
            + \int_0^t\iG \gamma(\psi) |\bv|^2
            \\
            &\qquad 
            + \int_0^t\iO M_\Omega(\phi)|\nabla \mu|^2
	        + \int_0^t\iG M_\Gamma(\psi)|\nabla_\Gamma \theta|^2
            \\[1ex]
            &\quad
            \le E_K(\phi_0,\psi_0)
        \end{align*}
        holds for all $t\in[0,T]$. 
    \end{enumerate}
\end{definition}

Notice that, if $\Omega$ is connected, the pressure $p$ can be reconstructed from \eqref{wf:1*} by proceeding as in Remark~\ref{REM:PRESS}.

\subsubsection{Existence of a weak solution}

\begin{theorem}
\label{THM:WEAK:SING}
Let $K\ge 0$ be arbitrary.
Suppose that \ref{ass:1:dom}--\ref{ass:4:viscosity} and \ref{ass:1:pot}--\ref{ass:2:pot:dominance} are fulfilled.
Let $(\phi_0,\psi_0)\in \CV_K$ denote any initial data satisfying \eqref{initial:data:weak}. In the case $K=0$, let the domain $\Omega$ be of class $C^2$.
Then, the Cahn--Hilliard--Brinkman system \eqref{CHB} admits at least one weak solution $(\bv, \phi, \xi, \mu, \psi, \xi_\Gamma, \theta)$ in the sense of Definition~\ref{DEF:WS:SING}.
In all cases, if the domain $\Omega$ is at least of class $C^2$, it holds that
\begin{alignat}{2}
    \label{L2H3:reg}  
    \phi \in L^2(0,T; \Hx2),
    \quad 
    \psi \in L^2(0,T; H^2(\Gamma))
\end{alignat}
and the equations
\begin{alignat}{2}
\label{eq:4pier}
	\mu &= -\Delta \phi + \xi + \pi (\phi)	
	\quad &&\text{in $Q$,}
	\\
	\label{eq:5pier}
	\theta &=  - \Delta_\Gamma \psi + \xi_\Gamma + \pi_\Gamma (\psi) + \deln \phi
	\quad &&\text{on $\Sigma$},
    \\
    \label{eq:6pier}
    K\deln \phi &= \phi - \psi \quad &&\text{on $\Sigma$}
\end{alignat}
are fulfilled in the strong sense.
\end{theorem}

\section{Analysis of the Cahn--Hilliard--Brinkman system with regular potentials}
\label{SEC:REGPOT}

\subsection{Existence of weak solutions in the case \texorpdfstring{$K>0$}{}}
\label{SUBSEC:EXREG}

\begin{proof}[Proof of Theorem~\ref{THM:WEAK}]

We intend to construct a weak solution to system \eqref{CHB} by discretizing the weak formulation \eqref{wf} by means of a Faedo--Galerkin scheme. In this proof, the letter $C$ will denote generic positive constants that may depend on $K$, $\Omega$, $T$, the initial data and the constants introduced in \ref{ass:1:dom}--\ref{ass:4:viscosity}, and may change their value from line to line. 
Recall that, as $K$ is assumed to be positive here, we have $\sigma(K)=\frac 1K$.

\subsubsection{Construction of local-in-time approximate solutions}

It is well known that the Poisson--Neumann eigenvalue problem 
\begin{align}
\label{EVP:PN}
    -\Delta u = \lambda_\Omega u \;\;\text{in $\Omega$}, 
    \quad
    \deln u  = 0 \;\;\text{on $\Gamma$}
\end{align}
possesses countably many eigenvalues and a corresponding sequence $\{u_i\}_{i\in\mathbb N}\subset V$ of $H$-normalized eigenfunctions which form an orthonormal basis of $H$ and an orthogonal Schauder basis of $V$. Similarly, invoking the spectral theorem for compact self-adjoint operators, it follows that the eigenvalue problem 
\begin{align}
\label{EVP:LB}
    -\Delta_\Gamma v = \lambda_\Gamma v \;\;\text{on $\Gamma$}
\end{align}
for the Laplace--Beltrami operator
possesses countably many eigenvalues and a corresponding sequence $\{v_i\}_{i\in\mathbb N}\subset V_\Gamma$ of $H_\Gamma$-normalized eigenfunctions which form an orthonormal basis of $H_\Gamma$ and an orthogonal Schauder basis of $V_\Gamma$.
For any $k\in\mathbb N$, we now define the finite-dimensional subspaces
\begin{align*}
    V_k &:= {\rm span}\{ u_i : 1\leq i \leq k\} 
    \subset V,
    \\
    V_{\Gamma,k} &:= {\rm span}\{ v_j : 1\leq j \leq k\}
    \subset V_\Gamma,
    \\
    \CV_k &:= {\rm span}\{ (u_i,v_j) : 1\leq i,j \leq k\}
    \subset \CV.
\end{align*}
We point out that, due to the above considerations, the inclusions 
\begin{align*}
    \bigcup_{k\in\mathbb N} V_k \subseteq V,
    \qquad
    \bigcup_{k\in\mathbb N} V_{\Gamma,k} \subseteq V_\Gamma,
    \qquad
    \bigcup_{k\in\mathbb N} \CV_k \subseteq \CV
\end{align*}
are dense.
In order to construct a sequence of approximate solutions, we use a semi-Galerkin approach. This means that only the quantities $\phi$, $\psi$, $\mu$ and $\theta$ are approximated by a Galerkin-scheme, and the approximate velocity field is obtained by directly solving the corresponding Brinkman subsystem. This approach has already been employed in \cite{Ebenbeck2019} for a Cahn--Hilliard--Brinkman model without dynamic boundary conditions and a no-friction boundary condition for the velocity equation. Compared to the present paper, some of the steps are carried out in \cite{Ebenbeck2019} in more detail, so we recommend it as a reference work.

To construct a sequence of approximate solutions, we now make the ansatz
\begin{align}
\label{DEF:APPROX}
\begin{aligned}
	\phi_k (x,t) & = \sum_{i=1}^k a_i^k (t) u_i(x),
	\quad 
    && \psi_k (x,t) = \sum_{i=1}^k b_i^k (t) v_i(x),
	\\
	\mu_k (x,t) & = \sum_{i=1}^k c_i^k (t) u_i(x),
	\quad 
	&& \theta_k (x,t) = \sum_{i=1}^k d_i^k (t) v_i(x)
 \end{aligned}
\end{align}
for every $k\in\mathbb N$, where the coefficients $\mathbf{a}^k := (a_1^k,...,a_k^k)^\top$, $\mathbf{b}^k := (b_1^k,...,b_k^k)^\top$, $\mathbf{c}^k := (c_1^k,...,c_k^k)^\top$, $\mathbf{d}^k := (d_1^k,...,d_k^k)^\top$ are still to be determined.

Let now $k\in \N$ and $t\in [0,T]$ be arbitrary. 
For any choice of $\mathbf{a}^k$, $\mathbf{b}^k$, $\mathbf{c}^k$ and $\mathbf{d}^k$, 
and $(\phi_k,\psi_k,\mu_k,\theta_k)$ as defined in \eqref{DEF:APPROX}, we
consider the bilinear form
\begin{align}
\label{DEF:BKT}
    \begin{aligned}
    &\mathcal{B}_{k,t}: \Vs \times \Vs \to \R,
    \\
    &(\bv,\w) \mapsto 
    2 \iO \nu\big(\phi_k(t)\big)\, \D\bv : \D\w 
    + \iO \lambda\big(\phi_k(t)\big) \, \bv \cdot \w 
    + \iG \gamma\big(\psi_k(t)\big) \, \bv \cdot \w,
    \end{aligned}
\end{align}
which is related to the weak formulation of the Brinkman equation with Navier-slip boundary condition.
It is obvious that $\mathcal{B}_{k,t}$ is symmetric, and, in view of \ref{ass:4:viscosity}, it is easy to see that $\mathcal{B}_{k,t}$ is continuous.
We further recall that every $\bv\in \Vs$ satisfies $\div(\bv) = 0$ \aeO, and $\bv\cdot\nnn = 0$ \aeG. 
Now, if \ref{ass:4:viscosity:1} holds, we use Korn's inequality (Lemma~\ref{LEM:KORN}\ref{Korn1}) to deduce 
\begin{align}
    \label{EST:BKT:1}
    \mathcal{B}_{k,t}(\bv,\bv) 
    \ge 2\nu_1 \iO \, \D\bv:\D\bv + \gamma_1 \iG |\bv|^2 
    \ge \min\{2\nu_1,\gamma_1\} \big( \norma{\D\bv}_{\BH}^2 + \norm{\bv}_{\HHG}^2 \big)
    \ge C \norma{\bv}_{\VV}^2
\end{align}
for all $\bv\in \Vs$. On the other hand, if \ref{ass:4:viscosity:2} holds, we use Korn's inequality (Lemma~\ref{LEM:KORN}\ref{Korn2}) to conclude
\begin{align}
    \label{EST:BKT:2}
    \mathcal{B}_{k,t}(\bv,\bv) 
    \ge 2\nu_1 \iO \, \D\bv:\D\bv 
    \ge C \norma{\bv}_{\VV}^2
\end{align}
for all $\bv\in \Vs$.
This means that the bilinear form $\mathcal{B}_{k,t}$ is coercive in $\Vs$. Hence, the Lax--Milgram lemma implies that there exists a unique function $\bv_k(t) \in \Vs$ solving
\begin{align}
    \label{WF:BKT}
    \mathcal{B}_{k,t}\big(\bv_k(t),\w\big) = - \iO \phi_k(t) \nabla \mu_k(t) \cdot \w
            - \iG \psi_k(t) \nabla_\Gamma\th_k(t) \cdot \w
\end{align}
for all $\w\in \Vs$. As $t\in[0,T]$ was arbitrary, this defines a function $\bv_k:[0,T] \to \Vs$. 
We point out that by construction, $\bv_k$ depends continuously on the coefficients $\mathbf{a}^k$, $\mathbf{b}^k$, $\mathbf{c}^k$ and $\mathbf{d}^k$.

We now want to adjust the coefficient vectors $\mathbf{a}^k$, $\mathbf{b}^k$, $\mathbf{c}^k$ and $\mathbf{d}^k$ such that the discretized weak formulation
\begin{subequations}
    \label{wf:gal}
    \begin{alignat}{2}
        & 2 \iO \nu(\phi_k) \D\bv_k : \D\w 
        + \iO \lambda(\phi_k) \bv_k \cdot \w 
        + \iG  \gamma(\psi_k)\bv_k \cdot \w 
        \notag\\
        &\quad =
        - \iO \phi_k \nabla \mu_k \cdot \w
        - \iG \psi_k \nabla_\Gamma\th_k \cdot \w,
        \label{wf:1:gal}
         \\
        \label{wf:2:gal}
        & \<\delt \phi_k, \zeta>_{V}
         - \iO \phi_k \bv_k \cdot \nabla \zeta
         + \iO M_\Omega(\phi_k)\nabla \mu_k \cdot \nabla \zeta 
         =0 ,
         \\
        \label{wf:3:gal}
        & \<\delt \psi_k, \zetaG>_{\VG}
         - \iG \psi_k \bv_k \cdot \nabla_\Gamma \zetaG
         + \iG M_\Gamma(\psi_k)\nabla_\Gamma \theta_k \cdot \nabla_\Gamma \zetaG
         =0,
         \\
     & 
     \label{wf:4:gal}
        \iO \mu_k \eta 
      + \iG \theta_k \etaG
         = \iO \nabla \phi_k \cdot \nabla \eta 
      + \iO F'(\phi_k) \eta
      + \iG \nabla_\Gamma \psi_k \cdot \nabla_\Gamma \etaG 
      \notag\\
         &\hspace{20ex}
      + \iG G'(\psi_k) \etaG
      + \frac 1K \iG (\psi_k - \phi_k) (\etaG - \eta)
    \end{alignat}
\end{subequations}
for all test functions $\w \in \Vs$, $\zeta\in V_k$, $\zetaG\in V_{\Gamma,k}$ and $(\eta,\etaG) \in \CV_k$, and the initial conditions
\begin{align}
    \label{wf:ini:gal}
    \phi_k(0) = \phi_{0,k} := \mathbb{P}_{V_k}(\phi_0)
    \quad\text{and}\quad
    \psi_k(0) = \psi_{0,k} := \mathbb{P}_{V_{\Gamma,k}}(\psi_0)
\end{align}
are fulfilled. 
With the symbol $\mathbb{P}_{V_k}$ we denote the $H$-orthogonal projection of $V$ onto $V_k$ whereas $\mathbb{P}_{V_{\Gamma,k}}$ denotes the $H_\Gamma$-orthogonal projection of $V_\Gamma$ onto $V_{\Gamma,k}$. 

Choosing $\zeta=u_j$ in \eqref{wf:2:gal} and $\zeta_\Gamma=v_j$ in \eqref{wf:3:gal} for $j=1,...,k$, we infer that $(\mathbf{a}^k,\mathbf{b}^k)^\top$ is determined by a system of $2k$ nonlinear ordinary differential equations subject to the initial conditions
\begin{align*}
    [\mathbf{a}^k]_i(0) = a_i^k(0) = \scp{\phi_0}{u_i}_{H}
    \quad\text{and}\quad
    [\mathbf{b}^k]_i(0) = b_i^k(0) = \scp{\psi_0}{v_i}_{H_\Gamma}
\end{align*}
for all $i\in\{1,...,k\}$.
In particular, since the functions $M_\Omega$ and $M_\Gamma$ are continuous and $\bv_k$ depends continuously on the coefficients $\mathbf{a}^k$, $\mathbf{b}^k$, $\mathbf{c}^k$ and $\mathbf{d}^k$, the same holds for the right-hand side of this ODE system. Moreover, choosing $(\eta,\eta_\Gamma)=(u_j,0)$ and $(\eta,\eta_\Gamma)=(0,v_j)$ for $j=1,...,k$ in \eqref{wf:4:gal}, respectively, we find that the coefficients $\mathbf{c}^k$ and $\mathbf{d}^k$ are explicitly given by $2k$ algebraic equations whose right-hand side depends continuously on $\mathbf{a}^k$ and $\mathbf{b}^k$. This allows us to replace $\mathbf{c}^k$ and $\mathbf{d}^k$ in the right-hand side of the aforementioned ODE system to obtain a closed ODE system for the vector-valued function $(\mathbf{a}^k,\mathbf{b}^k)^\top$ whose right-hand side depends continuously on $(\mathbf{a}^k,\mathbf{b}^k)^\top$. Consequently, the Cauchy--Peano theorem implies the existence of at least one local-in-time solution 
$$(\mathbf{a}^k,\mathbf{b}^k)^\top:[0,T_k^*)\cap [0,T] \to \R^{2k}$$ 
to the corresponding initial value problem. Here, we take $T_k^*>0$ as large as possible meaning that $[0,T_k^*)\cap [0,T]$ is the right-maximal time interval of this solution. We can now reconstruct 
$$(\mathbf{c}^k,\mathbf{d}^k)^\top:[0,T_k^*)\cap [0,T] \to \R^{2k}$$ 
by the aforementioned system of $2k$ algebraic equations. 
Without loss of generality, we now assume $T_k^*\le T$ to simplify the notation.
Recalling the ansatz \eqref{DEF:APPROX} as well as the construction of $\bv_k$, we obtain an approximate solution $(\bv_k,\phi_k,\mu_k,\psi_k,\theta_k)$ with
\begin{align*}
    \bv_k \in C^0([0,T_k^*);\Vs),
    \quad
    (\phi_k,\psi_k) \in C^1([0,T_k^*);\CV),
    \quad\text{and}\quad
    (\mu_k,\theta_k) \in C^0([0,T_k^*);\CV),
\end{align*}
which fulfills the discretized weak formulation \eqref{wf:gal} on the time interval $[0,T_k^*)$. 


%

\subsubsection{Uniform estimates} \label{SSSEC:UNI}

Let now $T_k\in (0,T_k^*)$ be arbitrary.
We derive suitable estimates for the approximate solutions $(\bv_k, \phi_k, \mu_k, \psi_k, \theta_k)$ that are uniform in $k$ and $T_k$.
These estimates will allow us to extend the approximate solutions onto the whole interval $[0,T]$ and to extract suitable convergent subsequences.

\step First estimate

We first test \eqref{wf:1:gal} by $\bv_k$, \eqref{wf:2:gal} by $\mu_k$, \eqref{wf:3:gal} by $\theta_k$, and \eqref{wf:4:gal} by $-  (\delt\phi_k,\delt\psi_k)$. We then add these equations and integrate the resulting equation with respect to time from $0$ to an arbitrary $t\in [0,T_k]$. We obtain
\allowdisplaybreaks[0]
\begin{align}
    \label{EST:1}
    &E_K\big(\phi_k(t),\psi_k(t)\big)
    +  2\int_0^t \iO \nu(\phi_k) \abs{\D \bv_k}^2
    +  \int_0^t \iO \lambda(\phi_k) \abs{\bv_k}^2
    +  \int_0^t \iG \gamma(\psi_k) \abs{\bv_k}^2
    \notag\\
    &\quad
    + \int_0^t \iO M_\Omega(\phi_k) \abs{\nabla \mu_k}^2
    + \int_0^t \iG M_\Gamma(\psi_k)\abs{\nabla_\Gamma \theta_k}^2  
    \leq E_K(\phi_{0,k},\psi_{0,k})
\end{align}
\allowdisplaybreaks
for every $t \in [0,T_k]$. Due to \eqref{wf:ini:gal} and the assumptions on the initial data, we have
\begin{align}
    \label{EST:INI:1}
    \norma{\phi_{0,k}}_{V} \le C\norma{\phi_0}_{V} \le C
    \quad\text{and}\quad
    \norma{\psi_{0,k}}_{\VG} \le C\norma{\psi_0}_{\VG} \le C.
\end{align}
In view of the growth conditions from \ref{ass:pot:reg:1}, this directly implies
\begin{align}
    \label{EST:INI:2}
    \norma{F\big(\phi_{0,k}\big)}_{L^1(\Omega)} \le C,
    \quad
    \norma{G\big(\psi_{0,k}\big)}_{L^1(\Gamma)} \le C
    \quad\text{and thus},\quad
    E_K(\phi_{0,k},\psi_{0,k}) \le C.
\end{align}
Hence, using the conditions in \ref{ass:3:mobility} and \ref{ass:4:viscosity}, a straightforward computation yields
\begin{align}
    \label{EST:UNI:1}
    & \sqrt{2\nu_1} \norma{\D\bv_k}_{\Lk2 {\BH}} 
    + \sqrt{\gamma_1} \norma{\bv_k}_{\Lk2 \HHG}
    + \norma{\nabla \phi_k}_{\Lk\infty {\HH}}
	+ \norma{\nabla_\Gamma \psi_k}_{\Lk\infty {\HH_\Gamma}}
	\notag\\
	& \quad 
    + \norma{\nabla \mu_k}_{\Lk2 {\HH}}
	+ \norma{\nabla_\Gamma \theta_k}_{\Lk2 {\HH_\Gamma}}
	\leq C.
\end{align}
Invoking Korn's inequality (see Lemma~\ref{LEM:KORN}), we directly infer
\begin{align}
    \label{EST:UNI:2}
    \norma{\bv_k}_{\Lk2 {\Vs}} \le C.
\end{align}
Next, taking $\zeta=\frac 1 {|\Omega|}$ in \eqref{wf:2:gal}, and $\zeta_\Gamma=\frac 1{|\Gamma|}$ in \eqref{wf:3:gal}, we infer
\begin{align*}
    \<\phi_k(t)>_\Omega
    =
    \<\phi_{k,0}>_\Omega,
    \quad 
    \<\psi_k(t)>_\Gamma
    = \<\psi_{{k,0}}>_\Gamma
    \quad \text{for all $t \in [0,T_k]$}.
\end{align*}
Hence, in view of \eqref{EST:INI:1} and \eqref{EST:UNI:1}, we use the Poincar\'e--Wirtinger inequality in $\Omega$ and Poincar\'e's inequality on $\Gamma$ (see~Lemma~\ref{LEM:POIN}) to conclude
\begin{align}
    \label{EST:UNI:3}
	& \norma{ \phi_k}_{\Lk\infty {V}}
	+ \norma{\psi_k}_{\Lk\infty {\VG}}
	\leq C.
\end{align}

\step 
Second estimate

Let now $\zeta\in \Lk2 V$ and $\zetaG \in \Lk2 \VG$ be arbitrary test functions. 
Testing \eqref{wf:2:gal} with $\ov \zeta:= \mathbb{P}_{V_k}(\zeta)$ and exploiting \eqref{EST:UNI:1}--\eqref{EST:UNI:3} along with Sobolev's embeddings, we obtain
\begin{align}
    \label{EST:T1}
    &\left| \int_0^{T_k} \< \delt\phi_k , \zeta >_{V} \;\right| 
    = \left| \int_0^{T_k} \< \delt\phi_k , \ov \zeta >_{V} \;\right| 
    \notag \\
    &\quad = \left| \int_0^{T_k}\!\!\iO \phi_k \bv_k \cdot \nabla \ov\zeta
         - \int_0^{T_k}\!\!\iO M_\Omega(\phi_k)\nabla \mu_k \cdot \nabla \ov\zeta \;\right| 
    \notag \\[1ex]
    &\quad
    \le
    \left(\norma{\phi_k}_{\Lk\infty {L^6(\Omega)}}\norma{\bv_k}_{\Lk2 {{\LL^{3}(\Omega)}}}
	+ M_2 \norma{\nabla\mu_k}_{\Lk2 \HH} \right)
    \norma{\ov\zeta}_{\Lk2 V}
     \notag \\[1ex]
    &\quad
    \le
     C \norma{\zeta}_{\Lk2 V}.
\end{align}
Hence, taking the supremum over all $\zeta\in \Lk2 V$ with $\norma{\zeta}_{\Lk2 V} \le 1$, we deduce
\begin{align}
    \label{EST:UNI:7}
    \norma{\delt \phi_k}_{\Lk2 {V^*}} \le C.
\end{align}
Proceeding similarly and testing \eqref{wf:3:gal} with ${\ov \zeta_\Gamma}:= \mathbb{P}_{V_{\Gamma,k}}(\zetaG)$, we obtain the estimate
\begin{align}
    \label{EST:T2}
    &\left| \int_0^{T_k} \< \delt\psi_k , \zetaG >_{ \VG} \;\right| 
    \notag\\[1ex]
    &\le \Bigl(C\norma{\psi_k}_{\Lk\infty {\VG}}\norma{\bv_k}_{\Lk2 {\Vs}}
    + M_{\Gamma,2}\norma{\nabla_\Gamma \theta_k}_{\Lk2 {\HH_\Gamma}} \Bigr)\norma{{\ov \zeta}_\Gamma}_{\Lk2 \VG}
    \notag\\[1ex]
    &\leq C \norma{\zetaG}_{\Lk2 \VG}.
\end{align}
Taking the supremum over all $\zetaG \in\Lk2 \VG$ with $\norma{\zetaG}_{\Lk2 \VG}\le 1$, we conclude
\begin{align}
    \label{EST:UNI:8}
	\norma{\delt \psi_k}_{\Lk2 {V_\Gamma^*}}
	\leq C.
\end{align}

\step 
Third estimate

Next, we want to derive uniform bounds on $\mu_k$ in $\Lk4 H \cap \Lk2 V$ and on $\th_k$ in $\Lk4 {\HG} \cap \Lk2 {\VG}$. 
Therefore, we choose arbitrary functions $\eta\in \Lk1 V$ and $\eta_\Gamma\in \Lk1 \VG$ and we set $\ov\eta:= \mathbb{P}_{V_k}(\eta)$ and $\ov{\eta_\Gamma}:= \mathbb{P}_{V_{\Gamma,k}}(\eta_\Gamma)$.
Testing \eqref{wf:4:gal} by $(\ov \eta,\ov \etaG)$, recalling the growth conditions from \ref{ass:pot:reg:1} as well as the uniform bounds \eqref{EST:UNI:1} and \eqref{EST:UNI:2}, we use H\"older's inequality and Sobolev's embedding theorem to derive the estimate
\begin{align*}
    &\abs{\int_0^{T_k}  \<(\mu_k,\th_k), (\eta,\etaG) >_{\CV}} 
    =  \abs{ \int_0^{T_k} \<(\mu_k,\th_k), (\ov\eta,\ov\etaG) >_{\CV} }   
    \\[1ex]
    & \quad 
    \leq 
    \int_0^{T_k} \Big[ \norma{\nabla \phi_k}_{\HH}\norma{\nabla \ov\eta}_{\HH}
        + \norma{F'(\phi_k)}_{L^{6/5}(\Omega)} \norma{\ov\eta}_{L^{6}(\Omega)}
        + \norma{\nabla_\Gamma\psi_k}_{\HH_\Gamma} \norma{\nabla_\Gamma \ov\etaG}_{\HH_\Gamma} 
    \\ 
    & \qquad\qquad\quad
        + \norma{G'(\psi_k)}_{\HG} \norma{\ov\etaG}_{\HG}
        + \tfrac 1K \norma{\psi_k - \phi_k}_{\HG} \norma{\ov\etaG - \ov\eta}_{\HG}
    \Big] 
    \\[1ex]
    & \quad 
    \leq
    C ( 1+ \norma{\phi_k}_{\Lk\infty V}^{p-1} + \norma{\psi_k}_{\Lk\infty \VG}^{q-1}) \norma{(\eta,\etaG)}_{\Lk1 \CV}
\end{align*}
in $[0,T_k]$.
Taking the supremum over all $(\eta,\etaG) \in \Lk1\CV$ with $\norma{(\eta,\etaG)}_{\Lk1 \CV}\le 1$, and using \eqref{EST:UNI:3}, we infer
\begin{align}
    \label{EST:UNI:4}
    \norma{(\mu_k,\th_k)}_{\Lk\infty {\CV^*}} \leq C.
\end{align}
We further have
\begin{align*}
   \norma{(\mu_k,\th_k)}^2_\CH
   & =  \<(\mu_k,\th_k),(\mu_k,\th_k) >_{\CV}
    \\ 
   & 
   \leq C   \norma{(\mu_k,\th_k)}_{{\CV^*}}   \big(\norma{( \mu_k,\th_k)}_\CH + \norma{(\nabla \mu_k,\nabla_\Gamma\th_k )}_{\HH\times \HH_\Gamma}\big)
   \\ 
   &   
   \leq \frac 12 \norma{(\mu_k,\th_k)}^2_\CH 
   + C \norma{(\mu_k,\th_k)}_{{\CV^*}}  \norma{(\nabla \mu_k,\nabla_\Gamma\th_k)}_{\HH\times \HH_\Gamma}
   + C \norma{(\mu_k,\th_k)}_{{\CV^*}}^2.
\end{align*}
Hence, squaring and integrating this estimate with respect to time, we use \eqref{EST:UNI:1} and \eqref{EST:UNI:4} to conclude
\begin{align}
    \label{EST:UNI:5}
    \norma{(\mu_k,\th_k)}_{\Lk4 \CH} \leq C.
\end{align}
In particular, we thus have
\begin{align}
    \label{EST:UNI:6}
	\norma{\mu_k}_{\Lk2 {V}}
	+ \norma{\th_k}_{\Lk2 {\VG}}
	\leq C.
\end{align}

\step Overall estimate

Combining \eqref{EST:UNI:1}--\eqref{EST:UNI:3}, \eqref{EST:UNI:7}, \eqref{EST:UNI:8}, \eqref{EST:UNI:5} and \eqref{EST:UNI:6}, we obtain the overall uniform estimate
\begin{align}
    \label{EST:UNI:OV*}
    & \norma{\bv_k}_{\Lk2 {\Vs} \cap \Lk2 \HHG}
    + \norma{\phi_k}_{H^1(0,T_k;{V^*}) \cap \Lk\infty {V}}
	+ \norma{\psi_k}_{H^1(0,T_k;{\VG^*})\cap \Lk\infty {\VG}}
	\notag\\
	& \quad 
    + \norma{\mu_k}_{\Lk4 {H}\cap \Lk2 {V}}
	+ \norma{\theta_k}_{\Lk4 {\HG}\cap \Lk2 {\VG}}
	\leq C.
\end{align}

\subsubsection{Extension of the approximate solution onto the whole time interval~\texorpdfstring{$[0,T]$}{}}

In Step~1, we constructed the coefficients
$(\mathbf a^k,\mathbf b^k)^\top$ as a solution of a nonlinear system of ODEs existing on its right-maximal time interval $[0,T_k^*)\cap [0,T]$. We now assume that $T_k^*\le T$.
By the definition of the approximate solutions given in \eqref{DEF:APPROX} and the uniform bound \eqref{EST:UNI:OV*}, we infer that for any $T_k\in[0,T_k^*)$, all $t\in [0,T_k]$, and all $i\in\{1,...,k\}$, 
\begin{align*}
    |a_i^k(t)| + |b_i^k(t)| 
    &= \bigabs{ \scp{\phi_k(t)}{u_i}_{H} } 
    + \bigabs{ \scp{\psi_k(t)}{v_i}_{\HG} }  
    \\
    &\le \norm{\phi_k}_{\Lk\infty H}
    + \norm{\psi_k}_{\Lk\infty \HG}
    \le C.
\end{align*}
This means that the solution $(\mathbf a^k,\mathbf b^k)^\top$ is bounded on the time interval $[0,T_k^*)$ by a constant that is independent of $T_k$ and $k$.
Hence, according to classical ODE theory, the solution can thus be extended beyond the time $T_k^*$. However, as the solution was assumed to be right-maximal, this is a contradiction. 
We thus have $T_k^*>T$, which directly implies $[0,T_k^*)\cap [0,T] = [0,T]$. This means that the solution $(\mathbf a^k,\mathbf b^k)^\top$ of the ODE system actually exists on the whole time interval $[0,T]$. 
As the coefficients $\mathbf c^k$ and $\mathbf d^k$ can be reconstructed from $\mathbf a^k$ and $\mathbf b^k$ by the corresponding system of algebraic equations, they also exist on the whole time interval $[0,T]$. Recalling \eqref{DEF:APPROX} and the construction of $\bv_k$ this directly entails that the approximate solution $(\bv_k, \phi_k, \mu_k, \psi_k, \theta_k)$ actually exists in $[0,T]$.
Hence, choosing $T_k = T$ in \eqref{EST:UNI:OV*}, we eventually conclude 
\begin{align}
    \label{EST:UNI:OV}
    & \norma{\bv_k}_{\L2 {\Vs} \cap\L2 \HHG}
    + \norma{\phi_k}_{\H1 {V^*} \cap \L\infty {V}}
	+ \norma{\psi_k}_{\H1 {\VG^*} \cap \L\infty {\VG}}
	\notag\\
	& \quad 
    + \norma{\mu_k}_{\L4 {H} \cap \L2 {V}}
	+ \norma{\theta_k}_{\L4 {\HG} \cap \L2 {\VG}}
	\leq C.
\end{align}

\subsubsection{Convergence to a weak solution as \texorpdfstring{$k\to\infty$}{}} 
\label{SSSEC:CONV}

Considering the uniform estimate \eqref{EST:UNI:OV}, we use the Banach--Alaoglu theorem and the Aubin--Lions--Simon lemma to infer that there exist functions $\bv$, $\phi$, $\mu$, $\psi$ and $\theta$ such that for any $s\in [0,1)$,
\begin{subequations}
\label{CONV}
\begin{alignat}{2}
    \label{CONV:V}
    \bv_k &\to \bv
    &&\quad \text{weakly in $\L2 \Vs $},
    \\
    \label{CONV:V:BD}
    \bv_k|_\Gamma & \to \bv|_\Gamma
    &&\quad \text{weakly in $\L2 \HHG$},
    \\
    \label{CONV:PHI}
    \phi_k &\to \phi
    &&\quad\text{weakly-$^*$ in $\L\infty V$, weakly in $\H1 {V^*}$}, \notag\\
    &&&\qquad\text{strongly in $C^0([0,T];H^s(\Omega))$, and \aeQ},
    \\
    \label{CONV:PSI}
    \psi_k &\to \psi
    &&\quad\text{weakly-$^*$ in $\L\infty \VG$, weakly in $\H1 {\VG^*}$}, \notag\\
    &&&\qquad\text{strongly in $C^0([0,T];H^s(\Gamma))$, and \aeS},
    \\
    \label{CONV:MU}
    \mu_k &\to \mu
    &&\quad\text{weakly in $\L2 V \cap \L4 H$},
    \\
    \label{CONV:THETA}
    \theta_k &\to \theta
    &&\quad\text{weakly in $\L2 \VG \cap \L4 \HG$},
\end{alignat}
\end{subequations}
as $k\to \infty$ along a non-relabeled subsequence. 
In particular, this shows that the functions $\bv$, $\phi$, $\psi$, $\mu$ and $\theta$ have the regularity demanded in Definition~\ref{DEF:WS:REG}\ref{DEF:WS:REG:i}.

Due to the trace theorem, the strong convergence from \eqref{CONV:PHI} (with $s>\frac 12$) directly yields
\begin{align}
    \label{CONV:PHI:G}
    \phi_k|_\Gamma \to \phi|_\Gamma \quad\text{strongly in $C^0([0,T];\HG)$}.
\end{align}
Recalling the growth conditions from \ref{ass:pot:reg:1}, we further deduce from the uniform bound \eqref{EST:UNI:OV} that 
\begin{align*}
    \norma{F'(\phi_k)}_{L^{6/5}(Q)} \le C 
    \quad\text{and}\quad
    \norma{G'(\phi_k)}_{L^{2}(\Sigma)} \le C.
\end{align*}
Hence, there exist weakly convergent subsequences in the respective spaces. As $F'$ and $G'$ are continuous, we use the pointwise convergences from \eqref{CONV:PHI} and \eqref{CONV:PSI} to conclude
\begin{alignat}{2}
    \label{CONV:F'}
    F'(\phi_k) &\to F'(\phi)
    &&\quad\text{weakly in $L^{6/5}(Q)$ and \aeQ},\\
    \label{CONV:G'}
    G'(\psi_k) &\to G'(\psi)
    &&\quad\text{weakly in $L^{2}(\Sigma)$ and \aeS}
\end{alignat}
since the weak limit and the pointwise limit must coincide (see, e.g., \cite[Proposition~9.2c]{DiBenedetto}). Furthermore, it follows from the pointwise convergences in \eqref{CONV:PHI} and \eqref{CONV:PSI} that, as $k\to\infty$,
\begin{alignat}{3}
    \label{CONV:MNL:AE}
    M_\Omega(\phi_k) &\to M_\Omega(\phi),
    &\quad \nu(\phi_k) &\to \nu(\phi),
    \quad \lambda(\phi_k) \to \lambda(\phi)
    &&\quad\text{\aeQ},\\
    \label{CONV:MGG:AE}
    M_\Gamma(\psi_k) &\to M_\Gamma(\psi),
    &\quad \gamma(\psi_k) &\to \gamma(\psi)
    &&\quad\text{\aeS}
\end{alignat}
as the functions $M_\Omega$, $M_\Gamma$, $\nu$, $\lambda$ and $\gamma$ are continuous. Since, due to \ref{ass:3:mobility} and \ref{ass:4:viscosity}, these functions are also bounded, we use Lebesgue's dominated convergence theorem 
along with the weak convergences in \eqref{CONV}
to infer that
\begin{alignat}{2}
    \label{CONV:NUPHI}
    \nu(\phi_k) \D\bv_k &\to \nu(\phi) \D\bv
    &&\quad\text{weakly in $L^2(Q;\R^{d\times d})$},
    \\
    \label{CONV:LAMPHI}
    \lambda(\phi_k) \bv_k &\to \lambda(\phi) \bv
    &&\quad\text{weakly in $L^2(Q;\R^d)$},
    \\
    \label{CONV:MPHI}
    M_\Omega(\phi_k) \nabla\mu_k &\to M_\Omega(\phi) \nabla\mu
    &&\quad\text{weakly in $L^2(Q;\R^d)$},
    \\
    \label{CONV:GV}
    \gamma(\psi_k) \bv_k &\to \gamma(\psi) \bv
    &&\quad\text{weakly in $L^2(\Sigma;\R^d)$},
    \\
    \label{CONV:MGPSI}
    M_\Gamma(\psi_k) \nabla_\Gamma\theta_k &\to M_\Gamma(\psi) \nabla_\Gamma\theta
    &&\quad\text{weakly in $L^2(\Sigma;\R^d)$}.
\end{alignat}
Combining the convergences \eqref{CONV:V}--\eqref{CONV:THETA}, \eqref{CONV:PHI:G}--\eqref{CONV:G'} and \eqref{CONV:NUPHI}--\eqref{CONV:MGPSI}, it is straightforward to pass to the limit as $k\to \infty$ in the discretized weak formulation \eqref{wf:gal} to conclude that the quintuplet $(\bv,\phi,\mu,\psi,\theta)$ fulfills the variational formulation \eqref{wf} for all test functions $\w \in \Vs$,
\begin{align*}
    \zeta \in \bigcup_{k\in\mathbb N} V_k \subseteq V,
    \qquad
    \zetaG \in \bigcup_{k\in\mathbb N} V_{\Gamma,k} \subseteq V_\Gamma,
    \qquad
    (\eta,\etaG) \in \bigcup_{k\in\mathbb N} \CV_k \subseteq \CV.
\end{align*}
Hence, because of density, \eqref{wf} holds true
for all test functions $\w \in \Vs$, $\zeta\in V$, $\zetaG\in V_\Gamma$ and $(\eta,\etaG) \in \CV = \CV_K$. This verifies Definition~\ref{DEF:WS:REG}\ref{DEF:WS:REG:ii}.

Moreover, we deduce from \eqref{wf:ini:gal} that
\begin{align*}
    \big( \phi_k(0), \psi_k(0) \big) \to \big( \phi_0, \psi_0 \big)
    \quad\text{strongly in $\CH$}
\end{align*}
as the orthogonal projections converge strongly in $H$ and in $\HG$, respectively. On the other hand, it follows from the strong convergences in \eqref{CONV:PHI} and \eqref{CONV:PSI} that
\begin{align*}
    \big( \phi_k(0), \psi_k(0) \big) \to \big( \phi(0), \psi(0) \big)
    \quad\text{strongly in $\CH$.}
\end{align*}
Hence, due to the uniqueness of the limit, this verifies Definition~\ref{DEF:WS:REG}\ref{DEF:WS:REG:iii}.

We still need to establish the weak energy dissipation law. 
Therefore, let $\rho \in C^\infty([0,T])$ be an arbitrary nonnegative test function. 
Employing the convergences \eqref{CONV:PHI} and \eqref{CONV:PSI}, the weak lower semicontinuity of the mappings 
\begin{align*}
    L^2(0,T;V) &\ni \zeta \mapsto \int_0^T \norma{\nabla \zeta(t)}_{\HH}^2 \rho(t), \\
    L^2(0,T;\VG) &\ni \xi \mapsto \int_0^T \norma{\nabla_\Gamma \xi(t)}_{\HH_\Gamma}^2 \rho(t),
\end{align*}
as well as Fatou's lemma, we deduce
\begin{align}
\label{LIMINF:E}
    \int_0^T E_K\big(\phi(t),\psi(t)\big)\, \rho(t) 
    \le \underset{k\to\infty}{\lim\inf}\;  \int_0^T E_K\big(\phi_k(t),\psi_k(t)\big) \, \rho(t) .
\end{align}
Proceeding similarly as above, we derive the convergences
\begin{alignat}{2}
    \label{CONV:NUPHI:SQ}
    \sqrt{\nu(\phi_k)}\, \D\bv_k &\to \sqrt{\nu(\phi)} \, \D\bv
    &&\quad\text{weakly in $L^2(Q;\R^{d\times d})$},
    \\
    \label{CONV:LAMPHI:SQ}
    \sqrt{\lambda(\phi_k)} \,\bv_k &\to \sqrt{\lambda(\phi)}\, \bv
    &&\quad\text{weakly in $L^2(Q;\R^d)$},
    \\
    \label{CONV:MPHI:SQ}
    \sqrt{M_\Omega(\phi_k)}\, \nabla\mu_k &\to \sqrt{M_\Omega(\phi)} \,\nabla\mu
    &&\quad\text{weakly in $L^2(Q;\R^d)$},
    \\
    \label{CONV:GV:SQ}
    \sqrt{\gamma(\psi_k)}\, \bv_k &\to \sqrt{\gamma(\psi)}\, \bv
    &&\quad\text{weakly in $L^2(\Sigma;\R^d)$},
    \\
    \label{CONV:MGPSI:SQ}
    \sqrt{M_\Gamma(\psi_k)}\, \nabla_\Gamma\theta_k &\to \sqrt{M_\Gamma(\psi)}\, \nabla_\Gamma\theta
    &&\quad\text{weakly in $L^2(\Sigma;\R^d)$}.
\end{alignat}
Hence, employing \eqref{EST:1}, \eqref{LIMINF:E} and weak lower semicontinuity, we eventually obtain
\begin{align*}
    &\int_0^T E_K\big(\phi(t),\psi(t)\big) \rho(t)
    \notag\\
    &\qquad
    +  \int_0^T  \iO \Big[
        2\nu(\phi) \abs{\D \bv}^2 \rho(t)
        +  \lambda(\phi) \abs{\bv}^2 \rho(t)
        +  M_\Omega(\phi) \abs{\nabla \mu}^2 \rho(t)
        \Big]
    \notag\\
    &\qquad
    +  \int_0^T  \iG \Big[ \gamma(\psi) \abs{\bv}^2 \rho(t)
        + M_\Gamma(\psi)\abs{\nabla_\Gamma \theta}^2 \rho(t) 
        \Big]
    \notag\\[1ex]
    &\quad
    \leq \underset{k\to\infty}{\liminf} \Bigg\{ \int_0^T E_K\big(\phi_k(t),\psi_k(t)\big) \rho(t)
    \notag\\
    &\qquad
    +  \int_0^T  \iO \Big[
        2\nu(\phi_k) \abs{\D \bv_k}^2 \rho(t)
        +  \lambda(\phi_k) \abs{\bv_k}^2 \rho(t)
        +  M_\Omega(\phi_k) \abs{\nabla \mu_k}^2 \rho(t)
        \Big]
    \notag\\
    &\qquad
    +  \int_0^T  \iG \Big[ \gamma(\psi_k) \abs{\bv_k}^2 \rho(t)
        + M_\Gamma(\psi_k)\abs{\nabla_\Gamma \theta_k}^2 \rho(t) 
        \Big]
    \Bigg\}
    \notag\\[1ex]
    &\quad
    \leq
    \underset{k\to\infty}{\lim} \int_0^T E_K\big(\phi_{0,k},\psi_{0,k}\big) \rho(t)
    = \int_0^T E_K\big(\phi_{0},\psi_{0}\big) \rho(t).
\end{align*}
Here, invoking the growth conditions from \ref{ass:pot:reg:1}, the equality in the last line follows by means of Lebesgue's general convergence theorem (see \cite[Section~3.25]{Alt}) since the orthogonal projections in \eqref{wf:ini:gal} converge strongly in $V$ and in $\VG$, respectively. As the nonnegative test function $\rho$ was arbitrary, this proves that the weak energy dissipation law stated in \eqref{DISSLAW:REG} holds for almost all $t\in[0,T]$. As the time integral in this inequality is continuous with respect to $t$ and since the mapping $t\mapsto E_K\big(\phi(t),\psi(t)\big)$ is lower semicontinuous, we conclude that \eqref{DISSLAW:REG} actually holds true for all $t\in[0,T]$. This means that Definition~\ref{DEF:WS:REG}\ref{DEF:WS:REG:iv} is verified.

We have thus shown that the quintuplet $(\bv,\phi,\mu,\psi,\theta)$ is a weak solution in the sense of Definition~\ref{DEF:WS:REG}.

\subsubsection{Additional regularity for the phase-fields} \label{SSSEC:REG}

It remains to prove the additional regularities. Without loss of generality, we merely consider the case $d=3$ 
as the case $d=2$ can be handled analogously but is even easier due to the better Sobolev embeddings in two dimensions. We deduce from \eqref{wf:4} written for the solution $(\bv, \phi, \mu, \psi, \theta)$ that there exists a null set $\mathcal N\subset[0,T]$ such that 
\begin{align}
    \label{REG:K:1*}
    &\iO \nabla \phi(t) \cdot \nabla \eta 
    + \iG \nabla_\Gamma \psi(t) \cdot \nabla_\Gamma \etaG 
    + \frac 1K \iG \big( \psi(t) - \phi(t) \big) (\etaG-\eta)
    \notag\\
    &\quad
    =  \iO \Big( \mu(t) - F'\big(\phi(t)\big) \Big)\, \eta 
      + \iG \Big( \theta(t) - G'\big(\psi(t)\big) \Big)\, \etaG
\end{align}
holds for all $t\in [0,T]\setminus\mathcal N$ and all test functions $(\eta,\etaG) \in \CV_0$. 

Let now $t\in [0,T]\setminus\mathcal N$ be arbitrary. We infer from \eqref{REG:K:1*} that the pair $\big(\phi(t),\psi(t)\big)$ is a weak solution of the bulk-surface elliptic problem 
\begin{subequations}
\begin{alignat}{2}
    - \Delta \phi(t) &= f(t)
    &&\quad\text{in $\Omega$},\\
    - \Delta_\Gamma \psi(t) + \deln \phi(t)&= g(t)
    &&\quad\text{on $\Gamma$},\\
    K \deln \phi(t)&= \psi(t) - \phi(t)
    &&\quad\text{on $\Gamma$},
\end{alignat}
\end{subequations}
where
\begin{align*}
    f(t) := \mu(t) - F'\big(\phi(t)\big)
    \quad\text{and}\quad
    g(t) := \theta(t) - G'\big(\psi(t)\big).
\end{align*}
For the definition of a weak solution to such bulk-surface elliptic problems, we refer to \cite[Definition~3.1]{knopf-liu}.

Let us first consider the case $\ell=2$.
As we assumed that the growth conditions in \ref{ass:pot:reg:1} are fulfilled with $p\le 4$, we have
\begin{align}
    \label{EST:F':H*}
    \norma{F'\big(\phi(t)\big)}_H &\le C + C\norma{\phi}_{L^6(\Omega)}^3 \le C,
    \\
    \label{EST:G':H*}
    \norma{G'\big(\psi(t)\big)}_{\HG} &\le C + C\norma{\psi}_{L^{2(q-1)}(\Gamma)}^{q-1} \le C.
\end{align}
Hence, applying regularity theory for elliptic problems with bulk-surface coupling (see \cite[Theorem~3.3]{knopf-liu}), we find that
$\big(\phi(t),\psi(t)\big) \in H^2(\Omega)\times H^2(\Gamma)$ with
\begin{align*}
    &\norma{\phi(t)}_{H^2(\Omega)}^2 + \norma{\psi(t)}_{H^2(\Gamma)}^2
    \le C \norma{f(t)}_{H}^2 + C\norma{g(t)}_{\HG}^2
    \notag\\
    &\quad \le C + C \norma{\mu(t)}_{H}^2 + C \norma{\theta(t)}_{\HG}^2.
\end{align*}
Since $\mu\in \L4{H}$ and $\theta\in \L4{\HG}$, this proves \eqref{reg:L4H2}.

We now consider the case $\ell=3$.
Recalling that the growth conditions in \ref{ass:pot:reg:1} are fulfilled with $p\le 4$, we use \eqref{EST:UNI:KN} to derive the estimates
\allowdisplaybreaks[0]
\begin{align*}
    \norma{F''\big(\phi(t)\big) \nabla\phi(t)}_\HH 
    &\le C\norma{\nabla\phi(t)}_{\HH} + C \norma{|\phi(t)|^2 \, \nabla\phi(t)}_{\HH}
    \\
    &\le C + C \norma{\phi(t)}_{L^6(\Omega)}^2 \, \norma{\nabla\phi(t)}_{{\LL^{6}(\Omega)}} 
    \\
    &\le C + C \norma{\phi(t)}_{H^2(\Omega)}
\end{align*}
\allowdisplaybreaks
and
\begin{align*}
    \norma{G''\big(\psi(t)\big) \nabla_\Gamma \psi(t)}_{\HH_\Gamma} 
    &\le C\norma{\nabla\psi(t)}_{\HH_\Gamma} + C \norma{|\psi(t)|^{q-2}\, \nabla\psi(t)}_{\HH_\Gamma}    
    \\
    &\le C + C \norma{\psi(t)}_{L^{{2(q-2)}}}^{q-2} \norma{\nabla\psi(t)}_{{\LL^{4}(\Gamma)}}    
    \\
    &\le C + C \norma{\psi(t)}_{H^2(\Gamma)}. 
\end{align*}
In combination with \eqref{EST:F':H*} and \eqref{EST:G':H*}, these estimates directly imply
\begin{align*}
    \norma{F'\big(\phi(t)\big)}_V \le C + C\norma{\phi(t)}_{H^2(\Omega)},
    \quad
    \norma{G'\big(\psi(t)\big)}_{\VG} \le C + C\norma{\psi(t)}_{H^2(\Gamma)}.
\end{align*}
Now, applying regularity theory for elliptic problems with bulk-surface coupling (see \cite[Theorem~3.3]{knopf-liu}), we infer
\begin{align*}
    &\norma{\phi(t)}_{H^3(\Omega)}^2 + \norma{\psi(t)}_{H^3(\Gamma)}^2
    \le C \norma{f(t)}_{V}^2 + C\norma{g(t)}_{\VG}^2
    \notag\\
    &\quad \le C + C \norma{\mu(t)}_{V}^2 + C \norma{\theta(t)}_{\VG}^2 + C\norma{\phi(t)}_{H^2(\Omega)}^2 + C\norma{\psi(t)}_{H^2(\Gamma)}^2.
\end{align*}
Recalling $\mu\in \L2{V}$, $\theta\in \L2{\VG}$ and that \eqref{reg:L4H2} with $\ell=2$ is already verified, this proves \eqref{reg:L2H3}. 
By means of Proposition~\ref{PROP:A}\ref{CR2}, we directly infer $(\phi,\psi) \in C^0([0,T];\CV_0)$, which proves \eqref{CH1:reg}. 

This means that all assertions are verified and thus, the proof is complete. 

\end{proof}

\subsection{The limit \texorpdfstring{$K\to0$}{} and existence of a weak solution in the case \texorpdfstring{$K=0$}{}}
\label{SUBSEC:EXREG:0}

\begin{proof}[Proof of Theorem~\ref{THM:Kto0}]
In this proof, the letter $C$ will denote generic positive constants that may depend on $\Omega$, $T$, the initial data and the constants introduced in \ref{ass:1:dom}--\ref{ass:4:viscosity}, but not on $\Kn$ or $n$. Such constants may also change their value from line to line.

First of all, as the initial data were prescribed as $(\phi_0,\psi_0)\in \CV_0$, they satisfy the Dirichlet type coupling condition $\phi_0\vert_{\Gamma} = \psi_0$ {\aeG}. In view of the definition of the energy functional in \eqref{DEF:E}, this means that the $\Kn$-depending term in the energy $E_\Kn(\phi_0,\psi_0)$ vanishes. It thus holds
\begin{align}
    \label{EST:EN}
    E_\Kn(\phi_0,\psi_0) = E_0(\phi_0,\psi_0) \le C
    \quad\text{for all $n\in\N$}.
\end{align}
According to Definition~\ref{DEF:WS:REG}\ref{DEF:WS:REG:iv}, the solutions $(\bv^\Kn, \phi^\Kn, \mu^\Kn, \psi^\Kn, \theta^\Kn)$ satisfy the weak energy dissipation law. By the definition of $E_\Kn$, we have
\begin{align*}
    &E_\Kn\big(\phi^\Kn(t),\psi^\Kn(t)\big) 
    + 2\int_0^t\iO \nu(\phi^\Kn)|\D\bv^\Kn|^2
    + \int_0^t\iO \lambda(\phi^\Kn) |\bv^\Kn|^2
    \\
    &\qquad 
    + \int_0^t\iG \gamma(\psi^\Kn) |\bv^\Kn|^2
    + \int_0^t\iO M_\Omega(\phi^\Kn)|\nabla \mu^\Kn|^2
    + \int_0^t\iG M_\Gamma(\psi^\Kn)|\nabla_\Gamma \theta^\Kn|^2
    \\[1ex]
    &\quad
    \le E_\Kn(\phi_0,\psi_0) \le C
\end{align*}
for all $t\in [0,T]$ and all $n\in{\N}$. 
In particular, recalling that the potentials $F$ and $G$ are nonnegative, this directly yields
\begin{align}
    \label{EST:PP:KN}
    \norma{\phi^\Kn - \psi^\Kn}_{\HG}^2 \le C\Kn\quad\text{for all $t\in [0,T]$ and all $n\in\N$.}
\end{align}
Testing \eqref{wf:2} and \eqref{wf:3} written for $(\bv^\Kn, \phi^\Kn, \mu^\Kn, \psi^\Kn, \theta^\Kn)$ by the constant functions $\frac{1}{|\Omega|}$ and $\frac{1}{|\Gamma|}$, respectively, we infer
\begin{align*}
    \meano{\phi^\Kn(t)} = \meano{\phi_0} 
    \quad\text{and}\quad
    \meang{\psi^\Kn(t)} = \meang{\psi_0} 
\end{align*}
for all $t\in[0,T]$ and all $n\in\N$. Hence, proceeding similarly as in the proof of Theorem~\ref{THM:WEAK} (Subsection~\ref{SSSEC:UNI}, First and Second estimates), we derive the uniform bound
\begin{align}
    \label{EST:UNI:KN:1}
    &\norma{\bv^\Kn}_{\L2 {\Vs} \cap \L2 \HHG}
    + \norma{\nabla\mu^\Kn}_{\L2 {\HH}}
	+ \norma{\nabla_\Gamma\theta^\Kn}_{\L2 {\HHG}}
    \notag\\[1ex] 
    &\quad + \norma{\phi^\Kn}_{H^1(0,T;{V^*}) \cap \L\infty {V}}
	+ \norma{\psi^\Kn}_{H^1(0,T;{\VG^*}) \cap \L\infty {\VG}}
	\leq C.
\end{align}
We now test \eqref{wf:4} written for $(\bv^\Kn, \phi^\Kn, \mu^\Kn, \psi^\Kn, \theta^\Kn)$ by $(\eta,0)$, where $\eta\in C^\infty_c(\Omega)$ is an arbitrary test function. Using \eqref{EST:UNI:KN:1} along with Hölder's inequality, we infer that
\begin{align}
\label{pier1}
    \abs{\iO \mu^\Kn \eta} 
    \le \norma{\phi^\Kn}_V \norma{\eta}_V + \norma{F'(\phi^\Kn)}_{L^{6/5}(\Omega)} \norma{\eta}_{L^{6}(\Omega)}
    \le C \norma{\eta}_V.
\end{align}
Fixing a function $\eta\in C^\infty_c(\Omega)$ with $\meano{\eta}\neq 0$, we deduce
$$ \norma{\mu^\Kn}_H \le C \big( 1 + \norma{\nabla \mu^\Kn}_\HH \big)$$
by means of a generalized Poincar\'e inequality (see, e.g., \cite[Lemma~B.63]{Ern}).
Hence, in combination with \eqref{pier1}, we conclude
\begin{align}
    \label{EST:MU:KN}
    \norma{\mu^\Kn}_{\L2 V} \le C .
\end{align}
In order to
derive an analogous estimate for $\theta^\Kn$, we first choose $\eta\equiv 1$ and $\etaG \equiv 0$ in \eqref{wf:4}. Employing \eqref{EST:UNI:KN:1}, we obtain
\begin{align}
    \label{EST:MEAN:KN}
    \abs{\frac 1{\Kn} \iG (\psi^\Kn - \phi^\Kn)} \le \norma{\mu^\Kn}_{L^1(\Omega)} + \norma{F'(\phi^\Kn)}_{L^1(\Omega)} 
    \le C \norma{\mu^\Kn}_{H} + C.
\end{align}
Let us now take $\eta\equiv 0$ and $\etaG \equiv 1$ in \eqref{wf:4}. Using \eqref{EST:UNI:KN:1} and \eqref{EST:MEAN:KN}, we deduce
\begin{align*}
    \abs{\iG \theta^\Kn} &\le \norma{G'(\psi^\Kn)}_{L^1(\Gamma)} + \abs{\frac 1{\Kn}\iG
    (\psi^\Kn - \phi^\Kn)} 
    \le C \norma{\mu^\Kn}_{H} + C.
\end{align*}
Employing Poincar\'e's inequality on $\Gamma$ (see Lemma~\ref{LEM:POIN}), we thus infer
\begin{align*}
    \norma{\theta^\Kn}_{\HG} \le C \norma{\nabla_\Gamma \theta^\Kn}_{\HHG} + C \norma{\mu^\Kn}_{H} + C.
\end{align*}
Squaring and integrating this estimate with respect to time over $[0,T]$, we eventually conclude
\begin{align}
    \label{EST:THETA:KN}
    \norma{\theta^\Kn}_{\L2\HG} \le C.
\end{align}
In summary, combining \eqref{EST:UNI:KN:1}, \eqref{EST:MU:KN} and \eqref{EST:THETA:KN}, we have thus shown
\begin{align}
    \label{EST:UNI:KN}
    &\norma{\bv^\Kn}_{\L2 {\Vs} \cap \L2 \HHG}
    + \norma{\mu^\Kn}_{\L2 {V}}
	+ \norma{\theta^\Kn}_{\L2 {\VG}}
    \notag\\[1ex] 
    &\quad + \norma{\phi^\Kn}_{H^1(0,T;{V^*}) \cap \L\infty {V}}
	+ \norma{\psi^\Kn}_{H^1(0,T;{\VG^*}) \cap \L\infty {\VG}}
	\leq C .
\end{align}

As in Subsection~\ref{SSSEC:CONV}, we deduce the existence of functions $(\bv^0, \phi^0, \mu^0, \psi^0, \theta^0)$ such that the convergences \eqref{CONV:V:K}--\eqref{CONV:THETA:K} hold along a non-relabeled subsequence. Moreover, the estimate \eqref{EST:PP:KN} directly implies \eqref{CONV:PP:K} and thus, all convergences in \eqref{CONV:K} are established. In particular, due to the trace theorem, we also have
\begin{align}
    \label{CONV:PP:G:KN}
    \phi^\Kn\vert_{\Gamma} - \psi^\Kn \;\to\; \phi^0\vert_{\Gamma} -\psi^0 \quad\text{strongly in $C^0([0,T];\HG)$}.
\end{align}
In combination with \eqref{EST:PP:KN}, this proves that $\phi^0\vert_\Gamma = \psi^0$ {\aeS} due to uniqueness of the limit.
Proceeding further as in Subsection~\ref{SSSEC:CONV}, we eventually show that the quintuplet $(\bv^0, \phi^0, \mu^0, \psi^0, \theta^0)$ is a weak solution of the Cahn--Hilliard--Brinkman system \eqref{CHB} in the sense of Definition~\ref{DEF:WS:REG}. 

It remains to verify the additional regularity assertions. 
Therefore, we proceed similarly as in Subsection~\ref{SSSEC:REG}.
Without loss of generality, we merely consider the case $d=3$. The case $d=2$ can be handled analogously but is even easier as the Sobolev embeddings in two dimensions are better. We infer from \eqref{wf:4} written for the solution $(\bv^0, \phi^0, \mu^0, \psi^0, \theta^0)$ and $K=0$ that there exists a null set $\mathcal N\subset[0,T]$ such that 
\begin{align}
    \label{REG:K:1}
    &\iO \nabla \phi^0(t) \cdot \nabla \eta 
    + \iG \nabla_\Gamma \psi^0(t) \cdot \nabla_\Gamma \etaG 
    \notag\\
    &\quad
    =  \iO \Big( \mu^0(t) - F'\big(\phi^0(t)\big) \Big)\, \eta 
      + \iG \Big( \theta^0(t) - G'\big(\psi^0(t)\big) \Big)\, \etaG
\end{align}
for all $t\in [0,T]\setminus\mathcal N$ and all test functions $(\eta,\etaG) \in \CV_0$. 

Let now $t\in [0,T]\setminus\mathcal N$ be arbitrary. We infer from \eqref{REG:K:1} that the pair $\big(\phi^0(t),\psi^0(t)\big)$ is a weak solution of the bulk-surface elliptic problem 
\begin{subequations}
\begin{alignat}{2}
    - \Delta \phi^0(t) &= f(t)
    &&\quad\text{in $\Omega$},\\
    - \Delta_\Gamma \psi^0(t) + \deln \phi^0(t)&= g(t)
    &&\quad\text{on $\Gamma$},\\
    \phi^0(t)\vert_\Gamma &= \psi^0(t)
    &&\quad\text{on $\Gamma$},
\end{alignat}
\end{subequations}
where
\begin{align*}
    f(t) := \mu^0(t) - F'\big(\phi^0(t)\big)
    \quad\text{and}\quad
    g(t) := \theta^0(t) - G'\big(\psi^0(t)\big).
\end{align*}

Let us first consider the case $\ell=2$.
Proceeding exactly as in Subsection~\ref{SSSEC:REG}, we deduce that $\big(\phi^0(t),\psi^0(t)\big) \in H^2(\Omega)\times H^2(\Gamma)$ with
\begin{align*}
    &\norma{\phi^0(t)}_{H^2(\Omega)}^2 + \norma{\psi^0(t)}_{H^2(\Gamma)}^2
    \le C \norma{f^0(t)}_{H}^2 + C\norma{g^0(t)}_{\HG}^2
    \notag\\
    &\quad \le C + C \norma{\mu^0(t)}_{H}^2 + C \norma{\theta^0(t)}_{\HG}^2
\end{align*}
thanks to regularity theory for elliptic problems with bulk-surface coupling (see \cite[Theorem~3.3]{knopf-liu}).
Since $\mu^0\in \L2{H}$ and $\theta^0\in \L2{\HG}$, this proves \eqref{reg:L4H2:0}.

We now consider the case $\ell=3$.
Proceeding analogously as in Subsection~\ref{SSSEC:REG}, we infer
\begin{align*}
    &\norma{\phi^0(t)}_{H^3(\Omega)}^2 + \norma{\psi^0(t)}_{H^3(\Gamma)}^2
    \le C \norma{f^0(t)}_{V}^2 + C\norma{g^0(t)}_{\VG}^2
    \notag\\
    &\quad \le C + C \norma{\mu^0(t)}_{V}^2 + C \norma{\theta^0(t)}_{\VG}^2 + C\norma{\phi^0(t)}_{H^2(\Omega)}^2 + C\norma{\psi^0(t)}_{H^2(\Gamma)}^2.
\end{align*}
by means of regularity theory for elliptic problems with bulk-surface coupling (see \cite[Theorem~3.3]{knopf-liu}).
Recalling $\mu^0\in \L2{V}$, $\theta^0\in \L2{\VG}$ and that \eqref{reg:L4H2:0} with $\ell=2$ is already verified, this proves \eqref{reg:L2H3:0}. 
With the help of Proposition~\ref{PROP:A}\ref{CR2}, we directly infer $(\phi^0,\psi^0) \in C^0([0,T];\CV_0)$. Moreover, via interpolation between $L^\infty(0,T;\CV_0)$ and $L^2(0,T;H^3(\Omega)\times H^3(\Gamma))$ (cf.~Lemma~\ref{LEM:INT}), we further get
\begin{equation*}
    (\phi^0,\psi^0) \in L^4(0,T;H^2(\Omega)\times H^2(\Gamma)).
\end{equation*}
This means that \eqref{CH1:reg*} is established. Eventually, a simple comparison argument based on \eqref{wf:4} yields \eqref{CH1:reg*:2}.

This means that all assertions are verified and thus, the proof is complete.
\end{proof}

\subsection{Uniqueness of the weak solution for regular potentials}
\label{SUBSEC:UNIREG}

In this subsection, we are going to prove Theorem \ref{THM:UQ} for regular potentials and $K\geq0$. To prove the theorem, we use some ideas devised in \cite{Giorgini-Knopf}.

\begin{proof}[Proof of Theorem \ref{THM:UQ}]
In this proof, the the letter $C$ will denote generic positive constants that may depend on $\Omega$, $T$, the initial data and the constants introduced in \ref{ass:1:dom}--\ref{ass:4:viscosity}. Such constants may also change their value from line to line.
We first introduce the following notation for the differences of the solution components:
\begin{align*}
    & \bv:= \bv_1-\bv_2,
    \quad 
    \phi:= \phi_1-\phi_2,
    \quad 
    \mu:= \mu_1-\mu_2,
    \quad 
    \psi:= \psi_1-\psi_2,
    \quad 
    \th:=\th_1-\th_2.
\end{align*}
Recall that $M, M_\Gamma$ and $\nu$ are assumed to be constant.
We thus infer that the quintuplet $(\bv,\phi,\mu,\psi,\theta)$
satisfies the variational formulation
\begin{subequations}
\begin{alignat}{2}
	& \non
    2 \nu \iO \D \bv : \D\w 
	+ \iO (\lambda(\phi_1)-\lambda(\phi_2)) \bv_1 \cdot \w 
	\\ & \qquad \non
    + \iO \lambda(\phi_2) \bv \cdot \w 
	+\iG (\gamma(\psi_1)-\gamma(\psi_2)) \bv_1 \cdot \w
    +  \iG \gamma(\psi_2) \bv \cdot \w 
	 \\ & \quad
	= 
	 - \iO ( \phi \nabla \mu_1 + \phi_2 \nabla \mu) \cdot \w
	 - \iG (\psi \nabla_\Gamma\th_1 + \psi_2\nabla_\Gamma \th) \cdot \w ,
	\label{wf:cd:1}
     \\
  	\label{wf:cd:2}
   & \<\delt \phi, \zeta>_{V}
	 - \iO (\ph \bv_1 + \phi_2 \bv) \cdot \nabla \zeta
	 + {M} \iO \nabla \mu \cdot \nabla \zeta 
	 =0 ,
	 \\
 	\label{wf:cd:3}
    & \<\delt \psi, \zetaG>_{\VG}
	 - \iG (\psi\bv_1 + \psi_2 \bv) \cdot \nabla_\Gamma \zetaG 
	 +  {M_\Gamma} \iG \nabla_\Gamma \theta \cdot \nabla_\Gamma \zetaG
	 =0
\end{alignat}
\end{subequations}
almost everywhere in $(0,T)$ 
for all test functions $\w \in \Vs$, $\zeta\in V$, and $\zetaG\in \VG$,
and the equations
\begin{alignat}{2}
	\label{mu:cd:strong}	
	\mu & = - \Delta \phi + F'(\ph_1)-F'(\ph_2)
	\quad && \text{ in $Q$,}
	\\
	\label{th:cd:strong}	
	\th & =- \Delta_\Gamma \psi +  G'(\psi_1) - G'(\psi_2) + \deln \phi	 
	\quad && \text{ on $\Sigma$}
\end{alignat}
are fulfilled in the strong sense due the higher regularities established in Theorem~\ref{THM:WEAK} and Theorem~\ref{THM:Kto0}.

We now test \eqref{wf:cd:1} by $\bv$, \eqref{wf:cd:2} by $\phi + \mu$, \eqref{wf:cd:3} by $\psi + \th$, and add the resulting equations. After some cancellations and rearrangements, we obtain
\begin{align}
    &  \non
    2\nu\norma{\D\bv}^2_{\mathbb{H}}
    + \iO \lambda(\phi_2)|{\bv}|^2
    + \iG \gamma(\psi_2) |{\bv}|^2
   + \<\delt \phi, \phi + \mu>_V
   \\ 
    & \qquad \non
   + {M}\norma{\nabla \mu}^2_{\HH}
     + \<\delt \psi, \psi + \th>_{\VG}
    + {M_\Gamma}\norma{\nabla_\Gamma \th}^2_{\HH_\Gamma}
    \\[1ex]
    & \quad \non
    =
    - \iO (\lambda(\ph_1)-\lambda(\ph_2))\bv_1 \cdot \bv
    - \iG (\gamma(\psi_1)-\gamma(\psi_2))\bv_1 \cdot \bv
     \\ & \qquad \non
     - \iO \phi \nabla \mu_1 \cdot \bv
	- \iG \psi \nabla_\Gamma \th_1 \cdot \bv
      \\ 
    & \qquad    \non
    + \iO (\phi \bv_1 + \phi_2 \bv) \cdot \nabla \phi
    + \iO \phi \bv_1 \cdot \nabla \mu
      \\ 
    & \qquad \non
       + \iG (\psi \bv_1 + \psi_2 \bv) \cdot \nabla_\Gamma \psi 
       + \iG \psi \bv_1 \cdot \nabla_\Gamma \th
    \\ 
    & \qquad  
    - {M} \iO \nabla \mu \cdot \nabla \phi
    - {M_\Gamma}\iG  \nabla_\Gamma \th \cdot \nabla_\Gamma \psi
    =: \sum_{i=1}^{10}I_i.
    \label{est:cd:main}
\end{align}
We point out that, as a consequence of Theorem~\ref{THM:WEAK} and Theorem~\ref{THM:Kto0}, it holds that  $(\ph_i,\psi_i) \in \L2 { (\Hx3 \times H^3(\Gamma))\cap \CV_K}$, $i=1,2$.
Next, by using \eqref{mu:cd:strong} and \eqref{th:cd:strong}, along with the chain rule formula in Proposition \ref{PROP:A}, we observe that the duality terms on the left-hand side can be reformulated as
\begin{align}
    &  \non\<\delt \ph, \phi + \mu>_V + \<\delt \psi, \psi + \th>_{\VG}
    \\ &\non\quad  = 
    \frac 12 \frac{\mathrm d}{\mathrm dt} \Big({\norma{\phi}^2_H+\norma{\psi}^2_{\HG}}\Big)
    +\< (\delt \phi,\delt \psi),
    \big( - \Delta \phi, 
    - \Delta_\Gamma \psi + \deln \phi 	
    \big)
    >_{\CV}
    \\ &\non\qquad 
    +\< \delt \phi, F'(\phi_1)-F'(\phi_2) >_{V}
    +\<\delt \psi, G'(\psi_1)-G'(\psi_2) >_{\VG}
    \\ &\quad \non =
    \frac 12 \frac{\mathrm d}{\mathrm dt}  \Big( {\norma{\phi}^2_V+\norma{\psi}^2_{\VG}} + \sigma(K)\norma{\psi- \phi}^2_{\HG}\Big)
    \\ &  \qquad
    \label{dual:term}
    +\< \delt \phi, F'(\phi_1)-F'(\phi_2) >_{V}
    +\<\delt \psi, G'(\psi_1)-G'(\psi_2) >_{\VG}.
\end{align}
Using \ref{ass:3:mobility} and \ref{ass:4:viscosity} as well as \eqref{dual:term}, we deduce from \eqref{est:cd:main} that
\begin{align*}
    &\frac 12 \frac{\mathrm d}{\mathrm dt}\Big( \norma{\phi}_V^2+ \norma{\psi}_{\VG}^2 + \sigma(K)\norma{\psi- \phi}^2_{\HG}\Big)
    \\ 
    &\quad + 2\nu \norma{\D \bv}^2_{\mathbb{H}}
    + \lambda_1 \norma{\bv}_{\HH}^2
    + \gamma_1 \norma{\bv}_{\HH_\Gamma}^2
    + {M}\norma{\nabla \mu}^2_{\HH}
    + {M_\Gamma}\norma{\nabla_\Gamma \th}^2_{\HH_\Gamma}
    \\ & \qquad 
    \le -\< \delt \phi, F'(\phi_1)-F'(\phi_2) >_{V}
    -\<\delt \psi, G'(\psi_1)-G'(\psi_2) >_{\VG}
    + \sum_{i=1}^{10}I_i .
\end{align*}
We now intend to control the terms $I_i$, $i=1,...,10$, by means of \Holder's inequality, Young's inequality, the Lipschitz continuity of $\lambda$ and $\gamma$, and integration by parts along with Sobolev's embeddings and the trace theorem. 
For a positive $\delta$ yet to be chosen, we derive the following estimates:
\begin{align*}
    {I_1 }
    & \leq
    C\norma{\ph}_{{\Lx 4}}\norma{\bv_1}_{{\LL^ 4(\Omega)}}\norma{\bv}_{{\HH}}
    \leq \d \norma{\bv}^2_{\HH}
    + \cd \norma{\bv_1}^2_{\VVV}\norma{\ph}_V^2,
    \\[1ex]
    {I_2} & \leq 
    C\norma{\psi}_{{\LGx 4}}\norma{\bv_1}_{{\LL^4(\Gamma)}}\norma{\bv}_{{\HHG}}
    \leq \delta \norma{\bv}^2_{\VV}
    + \cd \norma{\bv_1}^2_{\VV}\norma{\psi}^2_{\VG},
    \\[1ex]
    {I_3 +I_4}
    & 
    =
   \iO \mu_1 \nabla \ph \cdot \bv 
   	- \iG \psi \nabla_\Gamma \th_1 \cdot \bv
    \\
    & 
    {\leq 
     \norma{\mu_1}_{{\Lx 4}}\norma{\nabla \phi}_{{\HH}} \norma{\bv}_{{\LL^4(\Omega)}}
     +
     \norma{\psi}_{{\LGx 4}}\norma{\nabla_\Gamma \th_1}_{{\HHG}}\norma{\bv}_{{\LL^4(\Gamma)}}}
    \\ &  
    {\leq 
    {C} \norma{\mu_1}_V\norma{\phi}_{V} \norma{\bv}_{\VVV}
    +C
 \norma{\th_1}_{\VG}\norma{\psi}_{\VG}\norma{\bv}_{\bf V}}
    \\ & 
   {\leq 
   2\delta \norma{\bv}^2_{\bf V}
    + \cd \norma{\mu_1}^2_V\norma{\phi}^2_V
    + \cd \norma{\th_1}_{\VG}^2\norma{\psi}_{\VG}^2,}
   \\[1ex]
   {I_5 + I_6} & 
  \leq
   (\norma{\phi}_{{\Lx 4}}\norma{\bv_1}_{{\LL^4(\Omega)}} + 
    \norma{\phi_2}_{{\Lx 4}}\norma{\bv}_{{\LL^4(\Omega)}} )\norma{\nabla \phi}_{{\HH}}
    \\ & \quad
    + \norma{\phi}_{\Lx4 }\norma{\bv_1}_{{\LL^4(\Omega)}}\norma{\nabla \mu}_{\HH}
    \\ & 
     \leq \frac M4 \norma{\nabla \mu}_{\HH}^2 
     +\delta \norma{\bv}_{\VVV}^2
    +  (C\norma{\bv_1}_{\VV}+C\norma{\bv_1}_{\VV}^2+\cd\norma{\phi_2}_V^2) \norma{\phi}_V^2,
    \\[1ex]
     {I_7 + I_8} & 
    \leq
    {(\norma{\psi}_{{\LGx 4}}\norma{\bv_1}_{{\LL^4(\Gamma)}} + \norma{\psi_2}_{{\LGx 4}}\norma{\bv}_{{\LL^4(\Gamma)}})\norma{\nabla_\Gamma \psi}_{{\HHG}}}
    \\ & \quad 
    + \norma{\psi}_{\LGx4 }\norma{\bv_1}_{{\LL^4(\Gamma)}}\norma{\nabla_\Gamma \th}_{\HHG}
    \\ & 
     {\leq \frac {M_\Gamma}4 \norma{\nabla_\Gamma \th}_{\HHG}^2 
     +\delta \norma{\bv}_{\VVV}^2
    + (C\norma{\bv_1}_{\VV} +C\norma{\bv_1}_{\VV}^2 + \cd \norma{\psi_2}_{\VG}^2 )
    \norma{\psi}_{\VG}^2,}
    \\[1ex]
    {I_9+ I_{10} }
    & \leq
    {\frac {{M}}4 \norma{\nabla \mu}^2_{\HH}
    + \frac {{M_\Gamma}}4 \norma{\nabla_\Gamma \th}^2_{\HH_\Gamma}
    + C \norma{\nabla \phi}^2_{\HH}
    + C \norma{\nabla_\Gamma \psi}^2_{\HH_\Gamma}.}
\end{align*}
Furthermore, the terms in the last line of \eqref{dual:term} can be estimated by
\begin{align}
	& \non
	\big| \< \delt \phi, F'(\phi_1)-F'(\phi_2)>_{V} \big|
	+ \big| \< \delt \psi, G'(\psi_1)-G'(\psi_2)>_{\VG} \big|
	\\ & \quad \label{duality:est}
	\leq 
    \norma{\delt \phi}_{V^*}
        \norma{F'(\phi_1)-F'(\phi_2)}_V 
    + \norma{\delt \psi}_{\VG^*} 
        \norma{G'(\psi_1)-G'(\psi_2)}_{\VG} .
\end{align}
By means of a comparison argument in \eqref{wf:cd:2}, we obtain
\begin{align}
     \non
   \norma{\delt \phi}_{V^*} 
     &= \sup_{\norma{\zeta}_{{V}}\leq1} |
    \<\delt \phi, \zeta>_{V} |
    \\ &  \non
    \leq C (
    \norma{\phi}_{{\Lx 4}}\norma{\bv_1}_{{\LL^4(\Omega)}}
    +\norma{\phi_2}_{{\Lx 4}}\norma{\bv}_{{\LL^4(\Omega)}}
    +\norma{\nabla \mu}_{\HH})
    \\ &  \non
    \leq C (
    \norma{\phi}_{V}\norma{\bv_1}_{{\VV}}
    +\norma{\phi_2}_{{V}}\norma{\bv}_{{\VV}}
    +\norma{\nabla \mu}_{\HH}).
\end{align}
Similarly, using  \eqref{wf:cd:3}, we derive the estimate
\begin{align*}
    \norma{\delt \psi}_{\VG^*} 
     = \sup_{\norma{\zeta_\Gamma}_{{\VG}}\leq1} |
    \<\delt \psi, \zeta_\Gamma>_{\VG} |
    \leq C (
    \norma{\psi}_{\VG}\norma{\bv_1}_{{\VV}}
    +\norma{\psi_2}_{{\VG}}\norma{\bv}_{{\VV}}
    +\norma{\nabla_\Gamma \th}_{\HH_\Gamma}).
\end{align*}
By employing equation \eqref{wf:cd:2} as well as \ref{ass:pot:reg*}, we can bound the norm on the \rhs\ of \eqref{duality:est} as follows:
\begin{align}
    \label{EST:F':DIFF}
    & \norma{F'(\phi_1)-F'(\phi_2)}_V ^2
    = \norma{F'(\phi_1)-F'(\phi_2)}_H^2 + \norma{F''(\phi_1)\nabla\phi_1-F''(\phi_2)\nabla\phi_2}_{\HH}^2
    \notag\\[1ex] 
    & \quad 
    = \iO |F'(\phi_1)-F'(\phi_2)|^2 
    + \iO |F''(\phi_1) \nabla \phi|^2
    + \iO |F''(\phi_1)-F''(\phi_2)|^2 |\nabla \phi_2|^2
    \notag\\[1ex]
    & \quad
    \leq \iO \Big|\int_0^1 F'' \big(s \phi_1 + (1-s)\phi_2 \big) \ds\Big|^2 \phi^2 
    \; + \; C(\norma{\phi_1}_{{\Lx{\infty}}}^{2(p-2)}+1) \norma{\nabla \phi}^2_{\HH}
    \notag\\ & \qquad  
    +  \iO \Big|\int_0^1 F^{(3)} \big(s \phi_1 + (1-s)\phi_2 \big) \ds\Big|^2 \phi^2 \, |\nabla \phi_2|^2
    \notag\\[1ex]
    & \quad
    \leq C (\norma{\phi_1}^{2(p-2)}_{{\Lx{3(p-2)}}}+\norma{\phi_2}^{2(p-2)}_{{\Lx{3(p-2)}}}+1)\norma{\phi}_V^2
    + C(\norma{\phi_1}_{{\Lx \infty}}^{2(p-2)}+1) \norma{\phi}^2_{V}
    \notag\\ & \qquad  
    +  C (\norma{\phi_1}^{2(p-3)}_{{\Lx{12(p-3)}}}+\norma{\phi_2}^{2(p-3)}_{{\Lx{12(p-3)}}}+1)
    \norma{\phi_2}_{W^{1,4}(\Omega)}^2 \norma{\phi}_V^2.
\end{align}
We now recall the restrictions on $p$ and $q$ demanded in \eqref{EXP*}. In particular, we have $p\le 4$ if $d=3$. In the case $d=2$ we assume, without loss of generality, that $p\ge 5$.
Using Agmon's inequality as well as interpolation between Sobolev spaces (see Lemma~\ref{LEM:INT}), we derive the estimates
\begin{align*}
    &\norma{\phi_1}_{{\Lx \infty}}^{2(p-2)} 
    \le C \norma{\phi_1}_{H^{s}(\Omega)}^{2(p-2)}
    \le C \norma{\phi_1}_{H^1(\Omega)}^{2p-8} \norma{\phi_1}_{H^2(\Omega)}^{4}
    \le C \norma{\phi_1}_{H^2(\Omega)}^{4}
    &&\text{for $d=2$},
    \\[1ex]
    &\norma{\phi_1}_{{\Lx \infty}}^{2(p-2)} 
    \le C \norma{\phi_1}_{H^1(\Omega)}^{(p-2)} \norma{\phi_1}_{H^2(\Omega)}^{(p-2)}
    \le C \norma{\phi_1}_{H^2(\Omega)}^{4}
    &&\text{for $d=3$},
    \\[1ex]
    &\norma{\phi_2}_{W^{1,4}(\Omega)}^2
    \le C \norma{\phi_2}_{H^1(\Omega)}^{\frac 12} \norma{\phi_2}_{H^2(\Omega)}^{\frac 32}
    \le C \norma{\phi_2}_{H^2(\Omega)}^{\frac 32}
    &&\text{for $d=2,3$},
\end{align*}
where, in the first inequality, $s=\frac{2p}{2(p-2)}\in (1,2)$.
We thus infer from \eqref{EST:F':DIFF} that
\begin{align*}
    \norma{F'(\phi_1)-F'(\phi_2)}_V ^2
    \le C \Lambda \norma{\phi}_V^2
\end{align*}
with a the time-dependent function $\Lambda$ that is given by 
\begin{align*}
    \Lambda &:= 
    (1 + \norma{\phi_1}^{4}_{H^2{(\Omega)}})
    + \big( 1 + \norma{\phi_2}^{\frac 32}_{H^2{(\Omega)}} \big)
    \sum_{i=1,2} \big(1 
        + \norma{\phi_i}^{2(p-2)}_{{\Lx{3(p-2)}}}
        + \norma{\phi_i}^{2(p-3)}_{{\Lx{12(p-3)}}}
    \big).
\end{align*}
From \eqref{reg:L4H2} and \eqref{CH1:reg*}, we know that $\phi_2\in \L4 {\Hx2}$.

In the case $d=2$, we simply have 
\begin{align*}
    \phi_i \in \L{\infty} {\Lx{3(p-2)}} \cap \L{\infty} {\Lx{12(p-3)}}, \quad i=1,2,
\end{align*}
due to the Sobolev embedding $H^1(\Omega) \emb L^r(\Omega)$ for all $r\in (1,\infty)$.

In the case $d=3$, we use interpolation between Sobolev spaces (Lemma~\ref{LEM:INT}) along with the continuous embedding $H^{(4+\rho)/\rho}(\Omega) \emb L^{6\rho/(\rho-8)}(\Omega)$ to derive the estimate
\begin{align*}
    \int_0^T \norm{u(t)}_{L^{6\rho/(\rho-8)}(\Omega)}^\rho 
    &\le C \int_0^T \norm{u(t)}_{H^{(4+\rho)/\rho}(\Omega)}^\rho
    \le C \int_0^T \norm{u(t)}_{H^{3}(\Omega)}^2 \norm{u(t)}_{H^{1}(\Omega)}^{\rho-2}
    \\
    &\le  C \norm{u}_{L^2(0,T;H^{3}(\Omega))}^2 \norm{u}_{L^\infty(0,T;H^{1}(\Omega))}^{\rho-2}
\end{align*}
for any $\rho>8$ and any function $u\in \L\infty {\Hx1} \cap \L2 {\Hx3}$.
This proves the continuous embedding
\begin{align}
    \label{EMBED}
    \L\infty {\Hx1} \cap \L2 {\Hx3} \emb \L {\rho} {\Lx {\frac{6\rho}{\rho-8}}} \quad \text{for any $\rho > 8$.}
\end{align}%
Since $\phi_i\in \L\infty {\Hx1} \cap \L2 {\Hx3}$, $i=1,2$, we infer
\begin{align*}
    \phi_i \in \L{16} {\Lx{12}} , \quad i=1,2,
\end{align*}
by choosing $\rho=16$ in \eqref{EMBED}.

In summary, by means of \Holder's inequality, we conclude
\begin{align*}
    t \mapsto {\Lambda(t)} \in L^{1}(0,T) \quad\text{for $d=2,3$.}
\end{align*}

Arguing in a similar fashion, and
recalling \eqref{Ass:G'''} as well as the regularity in Theorem~\ref{THM:WEAK}, we find that 
\begin{align*}
    &  \norma{G'(\psi_1)-G'(\psi_2)}_{\VG} ^2
    \\[1ex] 
    &  \quad 
    \leq 
    (\norma{\psi_1}^{2(q-2)}_{\VG}
    + \norma{\psi_2}^{2(q-2)}_{\VG} + 1) \norma{\psi}_{\VG}^2
    + C  (\norma{\psi_1}_{L^\infty(\Gamma)}^{2(q-2)}+1) \norma{\psi}_{\VG}^2
    \\ & \qquad 
    + C \big( 1+ \norma{\psi_1}_{\VG}^{{2(q-3)}}
    + \norma{\psi_2}_{\VG}^{2(q-3)}\big) 
    \norma{\psi_2}_{H^2(\Gamma)}^2 \norma{\psi}^2_{\VG}
    \\[1ex] 
    &  \quad
    \leq C (\norma{\psi_1}_{L^\infty(\Gamma)}^{2(q-2)} 
        + \norma{\psi_2}_{H^2(\Gamma)}^2
        + 1) \norma{\psi}_{\VG}^2 .
\end{align*}
In view of \eqref{EXP*}, we assume, without loss of generality, that $q\ge 5$.
Recalling that the boundary $\Gamma$ is a $(d-1)$-dimensional submanifold of $\R^d$ with $d\in\{2,3\}$, we have $H^s(\Gamma)\emb L^\infty(\Gamma)$ for every $s>1$.
Hence, via interpolation between Sobolev spaces (see Lemma~\ref{LEM:INT}) we obtain the estimate
\begin{align*}
    \norma{\psi_1}_{L^\infty(\Gamma)}^{2(q-2)} 
    \le C \norma{\psi_1}_{H^{s}(\Gamma)}^{2(q-2)} 
    \le C \norma{\psi_1}_{\VG}^{2q-8} \norma{\psi_1}_{H^2(\Gamma)}^{4}
    \le C \norma{\psi_1}_{H^2(\Gamma)}^{4},
\end{align*}
where $s=\frac{2q}{2(q-2)} \in (1,2)$.
We thus conclude that
\begin{align*}
    \norma{G'(\psi_1)-G'(\psi_2)}_{\VG} ^2 
    \leq C \Theta \norma{\psi}_{\VG}^2 
\end{align*}
with a time-dependent function $\Theta$ that is given by
\begin{align*}
    t \mapsto \Theta (t) :=  C (1 + \norma{\psi_1(t)}^4_{H^2(\Gamma)} + \norma{\psi_2(t)}^2_{H^2(\Gamma)}) \in L^1(0,T).
\end{align*}

Therefore, upon collecting the above computations, the integral in \eqref{duality:est} can be estimated with the help of Young' inequality as
\begin{align*}
    &- \< \delt \phi, F'(\phi_1)-F'(\phi_2)>_{V}
	- \< \delt \psi, G'(\psi_1)-G'(\psi_2)>_{\VG}
	\\ & \quad 
    \leq 
    C \norma{\delt \phi}_{V^*}
        \norma{F'(\phi_1)-F'(\phi_2)}_V 
    + C \norma{\delt \psi}_{\VG^*} 
        \norma{G'(\psi_1)-G'(\psi_2)}_{\VG}
    \\ 
    & \quad \leq 
    \delta (\norma{\delt \phi}_{V^*}^2 + \norma{\delt \psi}_{\VG^*}^2)
    +
    \cd (\norma{F'(\phi_1)-F'(\phi_2)}_V^2
    +\norma{G'(\psi_1)-G'(\psi_2)}_{\VG}^2)
    \\ &  \quad 
    \leq 
    \delta C (\norma{\nabla \mu}^2_{\HH}+\norma{\nabla_\Gamma \th}^2_{\HHG})
    + \delta C\norma{\phi_2}_{\L\infty V}^2\norma{\bv}_{\VVV}^2
    + \cd (\norma{\bv_1}_{\VVV}^2 + \Lambda)  \norma{\phi}_V^2
    \\ & \qquad 
    + \delta C \norma{\psi_2}_{\L\infty {\VG}}^2\norma{\bv}_{\VVV}^2
    + \cd (\norma{\bv_1}_{\VVV}^2 + \Theta)  \norma{\psi}_{\VG}^2
\end{align*}
for a constant $\delta>0$ yet to be chosen.
Finally, we adjust $\delta \in (0,1)$ in such a way that 
\begin{align*}
    \delta \max \Big\{4, C , C \norma{\phi_2}_{\L\infty V}^2, C \norma{\psi_2}_{\L\infty {\VG}}^2\Big\} \leq \frac 14 \min \Big\{ {{M}}, {{M_\Gamma}}, {C_\star(\nu,\gamma_1)}\Big\}.
\end{align*}
Here, the constant $C_\star(\nu,\gamma_1)$ results from Korn's inequality (see Lemma~\ref{LEM:KORN}) and is chosen such that $2\nu \norma{\D \bv}^2_{\mathbb{H}} +  \gamma_1 \norma{\bv}^2_{\HH_\Gamma} \geq C_\star(\nu,\gamma_1) \norma{\bv}^2_{\VV}$.
Thus, we integrate over time and employ Gronwall's lemma to deduce that 
\begin{align*} 
    & \non \norma{\bv_1-\bv_2}_{\L2 {\VV}}
    + \norma{\phi_1-\phi_2}_{\L\infty V}
    + \norma{\nabla\mu_1-\nabla\mu_2}_{\L2 \HH}
    \\ & \qquad 
     + \norma{\psi_1-\psi_2}_{\L\infty \VG}
   + \norma{\nabla_\Gamma\th_1-\nabla_\Gamma\th_2}_{\L2 \HHG}
     \\ & \quad 
   \leq C (\norma{\phi_{0,1}-\phi_{0,2}}_V + \norma{\psi_{0,1}-\psi_{0,2}}_{\VG} ).
\end{align*}
Finally, by a comparison argument in \eqref{mu:cd:strong} and \eqref{th:cd:strong}, we infer that $(\mu,\theta)$ is bounded in $\L2 {\CH}$ by the same \rhs\ as the above inequality. This leads to \eqref{est:CD} and thus, the proof is complete.
\end{proof}

\section{Analysis of the Cahn--Hilliard--Brinkman system with singular potentials}
\label{SEC:POTSING}

We are now dealing with the proof of the existence of weak solutions for singular potentials.
Our strategy is to approximate the convex parts of the singular potentials $F$ and $G$ satisfying \ref{ass:1:pot} and \ref{ass:2:pot:dominance} by means of a Moreau--Yosida regularization. In this way, the approximate potentials are regular and exhibit quadratic growth and we can thus use Theorem~\ref{THM:WEAK} and Theorem~\ref{THM:Kto0} to obtain suitable approximate solutions. We then derive uniform estimates with respect to the approximation parameter, and eventually pass to the limit.
In the forthcoming analysis, the splitting $F'=\beta + \pi$ and $G'=\betaG + \pi_\Gamma$ from \ref{ass:1:pot} will be adopted.

\subsection{Yosida regularizations}
As mentioned, we rely on a Yosida regularization acting on the graphs $\beta$ and $\betaG$. For any $\eps \in (0,1)$, we approximate the maximal monotone graphs $\beta$ and $\betaG$ by
\begin{align*}
    \beps(r) & : = \frac 1{\eps} \Big(r - \big(I + \eps \beta\big)^{-1} (r)\Big),
    \quad 
    \beta_{\Gamma,\eps}(r): =\frac 1{\eps} \Big(r - \big(I + \eps \betaG\big)^{-1} (r)\Big),
    \quad r \in \R.
\end{align*}
It is well-known that $\beta_\eps$ and $\beta_{\Gamma,\eps}$ are single-valued and can be interpreted as monotone functions $\beta_\eps:\R\to\R$ and $\beta_{\Gamma,\eps}:\R\to\R$. Moreover,
the condition in \eqref{domination} implies that 
\begin{equation}
	\bigl |\beta_\varepsilon (r)\bigr | 
	\le \kappa_1 \bigl |\beta _{\Gamma,\varepsilon } (r)\bigr |+ \kappa_2
	\quad 
	\text{for all $r \in \mathbb{R}$ and all $\varepsilon \in (0,1)$}
	\label{pier3}
\end{equation} 
(see, e.g., \cite[Appendix]{CF6}), where $\kappa_1$ and $\kappa_2$ are the constants introduced in \eqref{domination}.
Next, we define $F_\eps  := \hat \beps + \hat \pi,$
$ G_\eps := \hat \beta_{\Gamma,\eps} + \hat \pi_\Gamma,$
where
\begin{align*}
    \hat \beps (r) := \int_0^r \beps (s) \,{\rm ds},
    \quad 
    \hat \beta_{\Gamma,\eps} (r) := \int_0^r \beta_{\Gamma,\eps} (s) \, {\rm ds},
    \quad r \in \R
\end{align*}
are actually the Moreau--Yosida regularizations of the singular parts $\hat \b$ and $\hat \b_\Gamma$ of the potentials $F$ and $G$. Now,
it is well-known that for every $r\in\R$,
\begin{subequations}
\label{pier6}
\begin{alignat}{4}
    \label{pier6-1}
    &0 \leq \hat \beps(r) \leq \hat \beta(r) 
    &&\quad \forall \eps\in(0,1),
    \quad 
    &\hat \beps(r) &
    \nearrow \hat \beta(r) 
    &&\quad\text{monotonically as $\eps \to 0,$}
    \\
    \label{pier6-2}
    &|\beps(r)| \leq |\beta^\circ(r)| 
    &&\quad  \forall \eps\in(0,1),
    \quad 
    &\beps(r) &\to \beta^\circ(r)
    &&\quad\text{as $\eps \to 0$.}
\end{alignat}
\end{subequations}
Analogous properties hold for $\beta_{\Gamma,\eps}$.
Moreover, owing to the growth condition \eqref{pier2}, $\hat \beps$ fulfills the following growth condition:
\begin{align} 
\begin{aligned}
  & \hbox{For every $M>0$ there exist $C_M>0$ and $\eps_M \in (0,1)$ such that}
  \\
  & \hat \beps (r) 
  \geq M \, r^2 - C_M
  \quad \hbox{for every $r\in\R$ and every $\eps\in(0,\eps_M)$}\,.
  \label{pier4}
\end{aligned}
\end{align}
This property is checked in detail in the paper \cite[beginning of~Section~3]{CGSS1}. Obviously, as a consequence, a similar condition holds for $\hat \beta_{\Gamma,\eps}$ since \eqref{pier3} entails that
\begin{align}
    \hat \beps(r) 
    \leq 
    \kappa_1 \hat \beta_{\Gamma,\eps} (r)
    + \kappa_2 |r|
    \quad \text{for every $r \in \R, \ \eps \in (0,1)$},
    \label{pier5}
\end{align}
thanks to $\beps(0)=\beta_{\Gamma,\eps}(0)=0$
and since $\hat \beps$ and $\hat \beta_{\Gamma,\eps}$ have the same sign. 
Due to their construction by the Yosida approximation, $\beta_\eps$ and $\beta_{\Gamma,\eps}$ are Lipschitz continuous and have at most linear growth. Hence, $\hat \beta_\eps$ and $\hat \beta_{\Gamma,\eps}$ have at most quadratic growth.

In the following, we assume that $\eps\in (0,\eps_1)$, where $\eps_1$ is given by \eqref{pier4} with $M=1$.
Then, \eqref{pier4} with $M=1$ and \eqref{pier5} along with the (at most) quadratic growth of $\hat \pi$ and ${\hat \pi}_\Gamma$ (cf.~\ref{ass:1:pot}), imply that both $F_\eps$ and $G_\eps$ are bounded from below by negative constants independent of $\eps\in (0,\eps_1)$. We can thus assume, without loss of generality, that $F_\eps$ and $G_\eps$ are nonnegative (otherwise, we add the modulus of their lower bounds to $\hat\pi$ or $\hat\pi_\Gamma$, repsectively). 
In summary, this entails that the approximate potentials $F_\eps$ and $G_\eps$ satisfy assumption \ref{ass:pot:reg:1} with $p=q=2$.

Now, the approximating system we aim to solve consists of \eqref{wf:1*}--\eqref{wf:4*} with $\beta=\beps$ and $\betaG=\beta_{\Gamma,\eps}$.
The regularity of the approximate potentials, in particular, implies that the inclusions $\xi_{\eps} \in \beps(\phi_{\eps})$ $\aeQ$ and $\xi_{\Gamma,\eps} \in \beta_{\Gamma,\eps}(\psi_{\eps})$ $\aeS$  turn into the identities $\xi_\eps = \beps(\phi_\eps)$ \aeQ\ and  $\xi_{\Gamma,\eps} = \beta_{\Gamma,\eps}(\psi_\eps)$ \aeS, respectively.

Therefore, as an immediate consequence of Theorem~\ref{THM:WEAK} and Theorem~\ref{THM:Kto0}, we obtain the following existence result.
\begin{corollary}\label{COR:EXISTENCE:APPROX}
Let $K \geq 0$, suppose that \ref{ass:1:dom}--\ref{ass:4:viscosity} hold, and let $(\phi_0,\psi_0) \in \CV_K$ be arbitrary initial data.
Then, for every $\eps\in (0,1)$, the approximate problem described above admits at least a weak solution $(\bv_\eps, \phi_\eps, \mu_\eps, \psi_\eps, \theta_\eps)$ in the sense of Definition~\ref{DEF:WS:REG}
with 
\begin{align*}
    \xi_\eps 
    := \beps(\phi_\eps) & \in \L\infty V, 
    \\
    \xi_{\Gamma, \eps} 
    := \beta_{\Gamma ,\eps}(\psi_\eps) & \in \L\infty {V_\Gamma} .
\end{align*}
Moreover, if  the domain $\Omega$ is of class $C^2$, it additionally holds
\begin{align*}
    (\phi_\eps,\psi_\eps) &\in
    L^2\big(0,T; H^2(\Omega) \times H^2(\Gamma) \big) , 
\end{align*}
and the equations \eqref{eq:4pier}--\eqref{eq:6pier} are fulfilled in the strong sense by
$\phi_\eps$, $\xi_\eps$, $\mu_\eps$, $\psi_\eps$, $\xi_{\Gamma,\eps}$, and $\theta_\eps$.
\end{corollary}

\subsection{Uniform estimates}
This section is devoted to derive estimates, uniform with respect to $\eps$, on the approximate solutions $(\bv_\eps, \phi_\eps, \xi_\eps, \mu_\eps, \psi_\eps, \xi_{\Gamma,\eps}, \theta_\eps)$. Those will be a key point to obtain suitable convergence properties that allow us to pass to the limit as $\eps \to 0$ later on. 
In the following, the letter $C$ will denote generic positive constants that may depend on $\Omega$, $T$, the initial data and the constants introduced in \ref{ass:1:dom}--\ref{ass:4:viscosity}, but not on $\eps$. These constants may also change their value from line to line.

\step First estimate

To begin with, we  test  \eqref{wf:2} by $\frac 1 {|\Omega|}$ and \eqref{wf:3} by $\frac 1{|\Gamma|}$ to infer that mass conservation for both 
$\phi_\eps$ and $\psi_\eps$ holds as claimed in \eqref{mass:cons:bulk}--\eqref{mass:cons:bd}. Recalling \eqref{def:mean} and \eqref{initial:data:weak} we have 
\begin{equation}
\label{pier7}
\<\phi_\eps (t) >_\Omega = \<\phi_0>_\Omega =\mz , \quad 
\<\psi_\eps(t) >_\Gamma =\<\psi_0>_\Gamma=: \mgz \quad \hbox{ for all }\, t\in [0,T]. 
\end{equation}
This property is intrinsically independent of $\eps$.

We now consider the weak energy dissipation law, already proved in the cases of regular potentials, to $(\bv_\eps, \phi_\eps, \mu_\eps, \psi_\eps, \theta_\eps)$, which reads as
\begin{align}
\label{pier9a.1}
	& 	
	\frac 12\norma{\nabla \phi_\eps(t)}^2_{\HH}
	+ \iO F_\eps(\phi_\eps(t))
	+ \frac 12\norma{\nabla_\Gamma \psi_\eps(t)}^2_{\HHG}
	\notag \\ & 
	 \qquad  
	 + \iG G_\eps(\psi_\eps(t))
	+ \frac{\sigma(K)}2 \norma{(\psi_\eps - \phi_\eps)(t)}^2_{\HG}
	 \notag \\ & 
	 \qquad 
	 + 2 \int_0^t\!\!\iO \nu(\phi_\eps)|\D\bv_\eps|^2
	 + \int_0^t\!\!\iO\lambda(\phi_\eps)|\bv_\eps|^2
	 + \int_0^t\!\!\iG \gamma(\psi_\eps)|\bv_\eps|^2
	 \notag \\ & 
	 \qquad 
  + \int_0^t\!\!\iO M_\Omega(\phi_\eps) |\nabla \mu_\eps|^2
	 + \int_0^t\!\!\iG M_\Gamma(\psi_\eps)|\nabla_\Gamma \theta_\eps|^2
	 \notag \\ & 
	 \quad\leq 
	 	\frac 12\norma{\nabla \phi_0}^2_{{\HH}}
	+ \iO F_\eps(\phi_0)
	+ \frac 12\norma{\nabla_\Gamma \psi_0}^2_{\HHG}
	+ \iG G_\eps(\psi_0)
	+ \frac{\sigma(K)}2 \norma{\psi_0 - \phi_0}^2_{\HG}
\end{align}
for all $t\in [0,T]$. Now, recalling that $F_\eps$ and $G_\eps$ satisfy assumption \ref{ass:pot:reg:1} with $p=q=2$, we observe that
\begin{equation}
    \label{pier9a.2}
	\frac 12\norma{\nabla \phi_0}^2_{\HH}
	+ \iO F_\eps(\phi_0)
	+ \frac 12\norma{\nabla_\Gamma \psi_0}^2_{\HHG}
	+ \iG G_\eps(\psi_0)
	+ \frac{\sigma(K)}2 \norma{\psi_0 - \phi_0}^2_{\HG}
	\leq C
\end{equation}
since $(\phi_0,\psi_0)\in \CV_K$ satisfies \eqref{initial:data:weak} and \eqref{pier6-1} holds.
At this point, we recall that the potentials $F_\eps$ and $G_\eps$ were assumed (without loss of generality) to be nonnegative.
Hence, in view of \ref{ass:3:mobility} and \ref{ass:4:viscosity} and thanks to \eqref{pier7} and the Poincar\'e--Wirtinger inequality in $\Omega$ and Poincar\'e's inequality on $\Gamma$ (see~Lemma~\ref{LEM:POIN}), it is not difficult to infer that 
\begin{align}
\label{pier9a}
	& \norma{\phi_\eps}_{\L\infty V}
	+ \| F_\eps(\phi_\eps)\|_{\L\infty {\Lx1}}
	+ \norma{\psi_\eps}_{\L\infty {\VG}}
	+ \|G_{\eps}(\psi_\eps)\|_{\L\infty {L^1(\Gamma)}}
	\notag \\
	& \quad 
	+ \norma{\bv_\eps}_{\L2 {{\bf V}} \cap \L2 {\HH_\Gamma}}
	+ \norma{\nabla \mu_\eps}_{\L2 {\HH}}
	+ \norma{\nabla_\Gamma \theta_\eps}_{\L2 {\HH_\Gamma}}
	\leq C.
\end{align}

\step Second estimate 

We proceed as in the derivation of \eqref{EST:UNI:7} and \eqref{EST:UNI:8} in the proof of Theorem~\ref{THM:WEAK}. Indeed, let us
take an arbitrary test function $\zeta \in L^2(0,T;V) $ in \eqref{wf:2}, then integrate over time and use \Holder's inequality to obtain that
\begin{align*}
	\Big|
        \ioT  \<\delt \phi_\eps, \zeta>_V
    \Big|
	&\leq 
	C \ioT (\norma{\phi_\eps}_{\Lx6}\norma{\bv_\eps}_{\LL^3(\Omega)} + M_2\norma{\nabla \mu_\eps}_{\HHH})\norma{\nabla \zeta}_{\HHH}
	\\ &  
    \leq 
    C \left(\norma{\phi_\eps}_{\L\infty {V}}\norma{\bv_\eps}_{\L2 {{\bf V}}}
	+  \norma{\nabla\mu_\eps}_{\L2 \HHH} \right)
    \norma{\zeta}_{\L2 V}
     \notag \\
    &
    \le
     C \norma{\zeta}_{\L2 V}.
\end{align*}
Taking the supremum over all $\zeta\in \L2 V$ with $\norma{\zeta}_{\L2 V} \leq 1$, we infer
\begin{align}
	\label{pier9b}
	\norma{\delt \phi_\eps}_{\L2 {\Vp}}\leq C .
\end{align}
The same argument, acting on equation \eqref{wf:3}, leads us to infer as well that 
\begin{align}
	\label{pier9c}
	\norma{\delt \psi_\eps}_{\L2 {V^*_\Gamma}}\leq C .
\end{align}

\step Third estimate 

To handle the cases $K>0$ and $K=0$ simultaneously, we introduce the following notation:
\begin{align}
    \label{DEF:ALPHA}\qquad
    \alpha(K) : = 
    \begin{cases}
        0  &\text{if $K>0$},\\
        1 &\text{if $K=0$}.
    \end{cases}
\end{align}  
We now test \eqref{wf:4} by the pair 
\begin{align*}
    (\eta,\etaG) = 
    \begin{cases}
        (\phi_\eps - \mz , \psi_\eps - \mz) &\text{if $K=0$},\\
        (\phi_\eps - \mz , \psi_\eps - \mgz) &\text{if $K>0$},
    \end{cases}
\end{align*}
which clearly belongs to $\CV_K$.
After some rearrangements, as well as adding and subtracting the constant $m_{\Gamma0}$ multiple times in the case $K=0$, we deduce
\begin{align}
\label{pier8}
	& \norma{\nabla \phi_\eps}^2_{\HH}
	 + \iO \b_\eps (\phi_\eps) (\phi_\eps - \mz)
	 + \norma{\nabla_\Gamma \psi_\eps}^2_{\HH_\Gamma}
	 + \iG \b_{\Gamma,\eps} (\psi_\eps) (\psi_\eps - \mgz)
	 \notag \\ & \quad 
	 =
	 \iO (\mu_\eps -  \<\mu_\eps>_{\Omega} ) (\phi_\eps - \mz)
	+ \iG (\theta_\eps -  \<\th_\eps>_{\Gamma} ) (\psi_\eps - \mgz)
	\notag\\ & \qquad 
	 + \sigma(K) \iG (\psi_\eps - \phi_\eps) \big(\phi_\eps -\psi_\eps - (\mz - \mgz) \big)
	 \notag\\ & \qquad 
	 -\iO \pi (\phi_\eps) (\phi_\eps - \mz)
	 - \iG \pi_{\Gamma} (\psi_\eps) (\psi_\eps - \mgz)
	 \notag\\
	 & \qquad 
	 {}+ \alpha (K) \iG ( G'_\eps(\psi_\eps)- \theta_\eps) (\mz - \mgz ).
\end{align}
Note that the subtracted mean values $\<\mu_\eps>_{\Omega}$ and $\<\th_\eps>_{\Gamma}$ in the first two summands on the right-hand side of \eqref{pier8} do not change the values of these integrals since, due to \eqref{pier7}, we have $\<\phi_\eps - \mz>_{\Omega} = 0$ and 
$\<\psi_\eps - \mgz>_{\Gamma} = 0$.

To deal with the terms on the \lhs\ of \eqref{pier8}, we recall that due to assumption \eqref{initial:data:weak}, $\mz$ and $\mgz$ lie in the interior of the domains $D(\beta)$ and $D(\beta_\Gamma)$, respectively.
We can thus exploit a useful property (see, e.g.,  \cite[Appendix, Prop.\ A.1]{MZ} and/or the detailed proof given in \cite[p.\ 908]{GMS2009}), namely there exist positive constants $c_1,c_2$ and a nonnegative constant $c_3$ such that
\begin{align} \label{mz:proof}
	&c_1 \norma{\b_\eps(\phi_\eps)}_{\Lx1}
	+ c_2 \norma{\b_{\Gamma,\eps}(\psi_\eps)}_{\LGx1} - c_3
	\notag \\
	&\quad \leq 
	 \iO \b_\eps (\phi_\eps) (\phi_\eps - \mz)
	+ \iG \b_{\Gamma,\eps} (\psi_\eps) (\psi_\eps - \mgz).
\end{align}
For the integrals in the second line of \eqref{pier8}, we employ H\"older's inequality along with the Poincar\'e--Wirtinger inequality in $\Omega$ and Poincar\'e's inequality on $\Gamma$ (see~Lemma~\ref{LEM:POIN}) to obtain that 
\begin{align}
\label{pier8a}
    &  \iO (\mu_\eps -  \<\mu_\eps>_{\Omega} ) (\phi_\eps - \mz)
	+ \iG (\theta_\eps -  \<\th_\eps>_{\Gamma} ) (\psi_\eps - \mgz)
	\notag \\ & \quad 
    \leq C \Bigl( \norma{\nabla \mu_\eps}_{\HH} \norma{\nabla  \phi_\eps}_{\HH} +  \norma{\nabla_\Gamma \th_\eps}_{\HH_\Gamma} \norma{\nabla_\Gamma  \psi_\eps}_{\HH_\Gamma} \Bigr).
\end{align}
Moreover, integrals in the third and the fourth line of \eqref{pier8} can be bounded by virtue of estimate \eqref{pier9a} as well as the Lipschitz continuity of $\pi$ and $ \pi_\Gamma$, so that
\begin{align*}
	& 
	\sigma(K) \iG (\psi_\eps - \phi_\eps) \big(\phi_\eps -\psi_\eps - (\mz - \mgz) \big)
	 \notag\\ 
	 &\quad
   -\iO \pi (\phi_\eps) (\phi_\eps - \mz)
	 - \iG \pi_{\Gamma} (\psi_\eps) (\psi_\eps - \mgz) 
    \leq 
	C\bigl(\norma{\phi_\eps }^2_H+\norma{\psi_\eps }^2_{\HG} +1\bigr). 
\end{align*}
It remains to estimate the integral in the last line of \eqref{pier8}, which is only present in the case $K=0$. Recall that if $K=0$, we assumed $\Omega$ to be of class $C^2$. Hence, 
we know from Corollary~\ref{COR:EXISTENCE:APPROX} that $(\phi_\eps,\psi_\eps) \in L^2\big(0,T; H^2(\Omega) \times H^2(\Gamma) \big)$, and that the equations
\begin{align}
\label{pier10}
 \mu_\eps &= - \Delta \phi_\eps +  F'_\eps(\phi_\eps)	
    &&\text{a.e. in $Q$},
 \\
\label{pier11}
\theta_\eps &= - \Delta_\Gamma \psi_\eps  - G'_\eps (\psi_\eps) + \deln \phi_\eps 	
	&&\text{a.e. on $\Sigma$}
\end{align}
hold in the strong sense.
Then, with the help of \eqref{pier11} and a simple integration by parts, it is not difficult to conclude that 
\begin{align}
	 &\alpha (K) \iG ( G'_\eps(\psi_\eps)- \theta_\eps) (\mz - \mgz ) 
	 \nonumber \\
	 &\quad{} = - \alpha (K) \iG \deln \phi_\eps (\mz - \mgz )
 \leq 
C\, \alpha(K) \norma{\deln \phi_\eps}_{\HG}.
    \label{pier8aa}
\end{align} 

In the following, we write $\Phi_\eps$ to denote generic nonnegative functions 
\begin{align}
    t\mapsto \Phi_\eps(t) \in L^2(0,T) 
    \qquad\text{with}\qquad
    \norma{\Phi_\eps}_{L^2(0,T)} \le C
    \quad\text{for all $\eps>0$}
\end{align}
i.e., the $L^2$-norm is bounded uniformly in $\eps$. Here, ``generic'' means that the explicit definition of the function $\Phi_\eps$ may vary throughout this proof.

All in all, collecting the inequlities \eqref{pier8}--\eqref{pier8a} and \eqref{pier8aa}, we conclude that 
\begin{align} 
\label{pier8b}
	&\norma{\b_\eps(\phi_\eps(t))}_{\Lx1} +
	\norma{\b_{\Gamma,\eps}(\psi_\eps(t))}_{\LGx1} \leq 
	 \Phi_\eps(t) + C\, \alpha(K)  \norma{\deln \phi_\eps (t)}_{\HG} 
\end{align}
for almost all $t\in (0,T)$.
Having shown \eqref{pier8b}, now we aim to prove additional $L^2$-bounds for the terms $\b_\eps(\phi_\eps)$ and
$\b_{\Gamma,\eps}(\psi_\eps)$. For that, we take advantage of the growth condition \eqref{pier3}, which follows from \eqref{domination} in \ref{ass:2:pot:dominance}. However, the related analysis has to be performed differently for the cases $K> 0$ and $K= 0$.

\step Further estimate in the case $K>0$

As $\alpha (K)=0$ in this case, \eqref{pier8b} yields 
\begin{align}
	\label{pier12}
	\norma{\b_\eps(\phi_\eps)}_{\L2 {\Lx1}}
	+\norma{\b_{\Gamma,\eps}(\psi_\eps)}_{\L2 {L^1(\Gamma)}}
	\leq C .
\end{align}
Of course, thanks to \eqref{pier9a} we also have
\begin{equation*}
    \norma{\pi(\phi_\eps)}_{\L2 {\Lx1}}
	+\norma{\pi_\Gamma(\psi_\eps)}_{\L2 {L^1(\Gamma)}} \leq C 
\end{equation*}
since $\pi$ and $ \pi_\Gamma$ are Lipschitz continuous. 
In combination with \eqref{pier12}, this entails
\begin{equation*}
    \norma{F'_\eps(\phi_\eps)}_{\L2 {\Lx1}}
	+\norma{G'_{\eps}(\psi_\eps)}_{\L2 {L^1(\Gamma)}} \leq C . 
\end{equation*}%
Consequently, by testing \eqref{wf:4} first by $(1,0)$ and then by $(0,1)$, one easily realizes that
\begin{align}
	\label{mean:chemicals}
	\norma{\<\mu_\eps>_\Omega}_{L^2(0,T)}
	+ \norma{\<\theta_\eps>_\Gamma}_{L^2(0,T)}
	\leq  C ,
\end{align}
whence, using \eqref{pier9a}, the Poincar\'e--Wirtinger inequality in $\Omega$ and Poincar\'e's inequality on $\Gamma$ (see Lemma~\ref{LEM:POIN}), we infer that
\begin{align}
	\label{chemical:L2V}
	\norma{\mu_\eps}_{\L2 V}
	+ \norma{\theta_\eps}_{\L2 {\VG}}
	\leq C .
\end{align}
Next, recalling that $\sigma(K)=\frac 1K$, we test \eqref{wf:4} by $(0,\b_{\Gamma,\eps}(\psi_\eps))$ obtaining 
\begin{align*}
	 & \norma{\b_{\Gamma,\eps} (\psi_\eps)}^2_{\HG}
	 +\iG \b'_{\Gamma,\eps}(\psi_\eps)|\nabla_\Gamma \psi_\eps|^2
	\\ & \quad 
	= 
	 \iG (\theta_\eps 
	 - \pi_\Gamma (\psi_\eps))  \b_{\Gamma,\eps}(\psi_\eps)
 -\frac 1K \iG (\psi_\eps - \phi_\eps)  \b_{\Gamma,\eps}(\psi_\eps).
\end{align*}
Observe now that the second term on the \lhs\ is nonnegative due to the monotonicity of $\b_{\Gamma,\eps} $. 
For the terms on the right-hand side, we use \Holder's inequality, Young's inequality and the trace theorem to infer  that
\begin{align*}
	 & \iG (\theta_\eps 
	 - \pi_\Gamma (\psi_\eps))  \b_{\Gamma,\eps}(\psi_\eps)
 -\frac 1K \iG (\psi_\eps - \phi_\eps)  \b_{\Gamma,\eps}(\psi_\eps).
	 \\ & \quad 
	 \leq 
	\frac 12 \norma{\b_{\Gamma,\eps} (\psi_\eps)}^2_{\HG}
	+  C \bigl( \norma{\theta_\eps}^2_{\HG} + \norma{\psi_\eps}^2_{\HG} + \norma{\phi_\eps}^2_V + 1\bigr).
\end{align*}
Hence, rearranging the terms and integrating over time we conclude that
\begin{align}
	\label{bG:L2}
	 \norma{\b_{\Gamma,\eps} (\psi_\eps)}_{\L2 {\HG}}
	\leq C.
\end{align}
Next, proceeding similarly, we test \eqref{wf:4} by $(\b_\eps(\phi_\eps),0)$. This leads us to 
\begin{align*}
	\norma{ \b_\eps(\phi_\eps) }^2_H
	+ \iO \b'_\eps(\phi_\eps)|\nabla \phi_\eps |^2
	 = \iO (\mu_\eps - \pi(\phi_\eps))\b_\eps(\phi_\eps)
	 +\frac 1K \iG (\psi_\eps - \phi_\eps) \b_\eps(\phi_\eps).
\end{align*}
Again, the second term on the \lhs\ is nonnegative owing to \ref{ass:1:pot}, whereas the first term on the right can be easily controlled by Young's inequality as 
\begin{align*}
\iO (\mu_\eps - \pi(\phi_\eps))\b_\eps(\phi_\eps)
	 \leq 
	\frac 12 \norma{ \b_\eps(\phi) }^2_H 
	+ C \bigl(\norma{\mu_\eps}^2_H+\norma{\phi_\eps}^2_H +1 \bigr).
\end{align*}
Besides, we handle the last term by combining the monotonicity of $\b_\eps$ with the property in \eqref{pier3}. Namely, it holds that
\begin{align*}
	&\frac 1K \iG (\psi_\eps - \phi_\eps) \b_\eps(\phi_\eps)
	\\ & \quad 
	=  - \frac 1K\iG (\phi_\eps - \psi_\eps) (\b_\eps(\phi_\eps) - \b_\eps(\psi_\eps)) 
	+\frac 1K \iG (\psi_\eps - \phi_\eps) \b_\eps(\psi_\eps) 
	\\ & \quad 
	\leq \frac  {1}K \iG |\psi_\eps- \phi_\eps|\, |\b_{\eps}(\psi_\eps)|
	\\ & \quad 
	\leq \frac  {\kappa_1}K \iG (|\psi_\eps|+|\phi_\eps|)|\b_{\Gamma,\eps}(\psi_\eps)|
	+\frac  {\kappa_2}K \iG (|\ps_\eps|+|\phi_\eps|)
	\\[1ex] & \quad 
	\leq 
	\norma{ \b_{\Gamma,\eps}(\psi_\eps) }^2_{\HG}
	+ C \bigl(\norma{\psi_\eps}^2_{\HG} + \norma{\phi_\eps}^2_V+1 \bigr).
\end{align*}
Hence, with the help of~\eqref{bG:L2}, this shows the corresponding estimate
\begin{align}
	\label{b:L2}
	 \norma{\b_{\eps} (\phi_\eps)}_{\L2 H}
	\leq C.
\end{align}

\step Further estimate in the case $K=0$

Recall that, as $K=0$, it now holds that $\sigma(K)=0$, $\alpha(K)=1$, and $\phi_\eps\vert_\Gamma = \psi_\eps$ a.e.~on $\Sigma$, along with \eqref{pier10} and \eqref{pier11}. Here, in our argumentation, we follow in parts the procedure devised in \cite{Colli2020}.

Multiplying \eqref{pier10} by $1/|\Omega|$ and integrating over $\Omega$, we find that 
\begin{align}
	\label{pier13}
	|\<\mu_\eps>_\Omega | \leq  C \norma{\deln \phi_\eps}_{\HG} + 
		\frac1{|\Omega|} \bigl(\norma{\b_\eps(\phi_\eps)}_{\Lx 1} +
	\norma{\pi(\phi_\eps)}_{\Lx 1} \bigr) .
\end{align}
Similarly, multiplying \eqref{pier11} by $1/|\Gamma|$ and integrating over $\Gamma$, we infer that 
\begin{align}
	\label{pier14}
	|\<\th_\eps>_\Gamma | \leq  C \norma{\deln \phi_\eps}_{\HG} + 
		\frac1{|\Gamma|} \bigl(\norma{\b_{\Gamma, \eps}(\psi_\eps)}_{\LGx 1} +
	\norma{\pi_\Gamma (\psi_\eps)}_{\LGx 1} \bigr) .
\end{align}
Then, combining \eqref{pier13} and \eqref{pier14}, on account of the estimates \eqref{pier9a} and \eqref{pier8b}
along with the Lipschitz continuity of $\pi$ and $ \pi_\Gamma$, we deduce that 
\begin{align}
	\label{pier15}
|\<\mu_\eps>_\Omega | + |\<\th_\eps>_\Gamma | \leq C\bigl( \Phi_\eps + \norma{\deln \phi_\eps }_{\HG} \bigr).
\end{align}
Combining \eqref{pier9a.1} and \eqref{pier9a.2}, we obtain the estimate
\begin{align*}
 \norma{\nabla \mu_\eps}_{{\HH}}
+ \norma{\nabla_\Gamma \theta_\eps}_{{\HH_\Gamma}} \le \Phi_\eps.
\end{align*}
Hence, with the help of the Poincar\'e--Wirtinger inequality in $\Omega$ and Poincar\'e's inequality on $\Gamma$ (see~Lemma~\ref{LEM:POIN}), we arrive at 
\begin{align} 
\label{pier16}
	&\norma{\mu_\eps(t)}_V +
	\norma{\theta_\eps(t)}_{V_\Gamma} \leq 
	 C\bigl( \Phi_\eps(t) + \norma{\deln \phi_\eps (t)}_{\HG} \bigr) 
\end{align}
for almost all $t\in (0,T)$.
Now, we multiply \eqref{pier10} by $\b_\eps(\phi_\eps)$ and integrate by parts. This yields
\begin{align}
\label{pier17}
	&\norma{ \b_\eps(\phi_\eps) }^2_H
	+ \iO \b'_\eps(\phi_\eps)|\nabla \phi_\eps |^2
	\notag\\
	&\quad  = \iO (\mu_\eps - \pi(\phi_\eps))\b_\eps(\phi_\eps)
	 + \iG \deln \phi_\eps \b_\eps(\phi_\eps).
	\notag\\
	&\quad  \leq  \frac12 \iO | \mu_\eps - \pi(\phi_\eps))|^2 + \frac12 \norma{ \b_\eps(\phi_\eps) }^2_H 
	 + \iG \deln \phi_\eps \b_\eps(\phi_\eps). 
\end{align}
Similarly, multiplying \eqref{pier11} by $- \b_{\Gamma,\eps} (\psi_\eps)$, it is straightforward to deduce that
\begin{align}
\label{pier18}
	 & \norma{\b_{\Gamma,\eps} (\psi_\eps)}^2_{\HG}
	 +\iG \b'_{\Gamma,\eps}(\psi_\eps)|\nabla_\Gamma \psi_\eps|^2
	\notag\\ 
	& \quad 
	= 
	 \iG (\theta_\eps 
	 - \pi_\Gamma (\psi_\eps))  \b_{\Gamma,\eps}(\psi_\eps) 
     - \iG \deln \phi_\eps \b_{\Gamma,\eps}(\psi_\eps)
     \notag\\
	&\quad  \leq   \iO |\theta_\eps 
	 - \pi_\Gamma (\psi_\eps)|^2 + \frac14 \norma{\b_{\Gamma,\eps}(\psi_\eps) }^2_{\HG} 
	 - \iG \deln \phi_\eps \b_{\Gamma,\eps}(\psi_\eps) . 
\end{align}
Recalling \eqref{pier3}, we observe that
\begin{align*}
	&  
   \biggl| \iG \deln \phi_\eps \b_\eps(\phi_\eps) -  \iG \deln \phi_\eps \b_{\Gamma,\eps}(\psi_\eps) \biggr|
   \\
   &\quad{}
    \leq    
   \norma{\deln \phi_\eps }_{\HG} \norma{ (\kappa_1 +1) |\b_{\Gamma,\eps}(\psi_\eps)| +\kappa_2 }_{\HG}
  \\
  &\quad{}
    \leq
     \frac14 \norma{\b_{\Gamma,\eps}(\psi_\eps) }^2_{\HG} + C \bigl(\norma{\deln \phi_\eps }_{\HG}^2 + 1\bigr).
\end{align*}
Hence, adding \eqref{pier17} and \eqref{pier18}, and using \eqref{pier9a} as well as \eqref{pier16}, we conclude that 
\begin{align} 
\label{pier19}
	&\norma{\b_\eps(\phi_\eps(t))}_H +
	\norma{\b_{\Gamma,\eps}(\psi_\eps(t))}_{\HG}\leq 
	 C\bigl( \Phi_\eps(t) + \norma{\deln \phi_\eps (t)}_{\HG} \bigr)  
\end{align}
for almost all $t\in (0,T)$.
Now, recalling again \eqref{pier10} and \eqref{pier11}, we observe that $\phi_\eps  $ solves the following bulk-surface elliptic problem: 
\begin{alignat*}{2}
    - \Delta \phi_\eps &= \mu_\eps - \b_\eps(\phi_\eps) -  \pi (\phi_\eps)	
    	&& \quad\text{in $\Omega$}, \\
     - \Delta_\Gamma \psi_\eps  + \deln \phi_\eps &=  \theta_\eps  - \b_{\Gamma,\eps}(\psi_\eps) - \pi_\Gamma (\psi_\eps)
    	&& \quad \text{on $\Gamma$},\\
    \phi_\eps|_\Gamma &= \psi_\eps 	
    	&& \quad \text{on $\Gamma$}
\end{alignat*}
a.e.~in $(0,T)$. 
Due to \eqref{pier9a}, \eqref{pier16}, \eqref{pier19} and the Lipschitz continuity of $\pi$ and $\pi_\Gamma$, it is clear that the right-hand sides in the above system belong to $L^2(\Omega)$ and $L^2(\Gamma)$, respectively.
Hence, applying regularity theory for elliptic problems with bulk-surface coupling (see \cite[Theorem~3.3]{knopf-liu}), we deduce that the estimate
\begin{align*}
    \begin{aligned}
	&\norma{ \phi_\eps }_{H^2(\Omega)}
	+
	\norma{\psi_\eps }_{H^2(\Gamma)}
	 \\
	 &\quad{} \le
	C \Bigl( 
	\norma{\mu_\eps - \b_\eps(\phi_\eps) -  \pi (\phi_\eps)}_H 
	 + 
	 \norma{\theta_\eps  - \b_{\Gamma,\eps}(\psi_\eps) + \psi_\eps- \pi_\gamma (\psi_\eps)}_{\HG}
	\Bigl) 
    \end{aligned}
\end{align*} 
holds a.e.~in $(0,T)$.
Now, in view of \eqref{pier9a}, \eqref{pier16}, \eqref{pier19} we can completely control the above \rhs\ and infer that 
\begin{align} 
\label{pier20}
	&\norma{ \phi_\eps (t) }_{H^2(\Omega)}
	+
	\norma{\psi_\eps(t) }_{H^2(\Gamma)}\leq 
	 C\bigl( \Phi_\eps(t) + \norma{\deln \phi_\eps (t)}_{\HG} \bigr) 
\end{align}
for almost all $t\in (0,T)$. On this basis, at this point we can use the standard trace theorem for the normal derivative concluding that, for some fixed $3/2<s<2$ there is a positive constant $C_s$ such that
$$ \norma{\deln \phi_\eps (t)}_{\HG} \leq C_s \norma{ \phi_\eps (t) }_{H^s(\Omega)}$$
for almost all $t\in (0,T)$. Hence, as $H^2 (\Omega) \subset H^s(\Omega) \subset V$ with compact embeddings, we infer from \eqref{pier20} by means of the Ehrling lemma that
\begin{align} 
\label{pier21}
	&\norma{ \phi_\eps (t) }_{H^2(\Omega)}
	+
	\norma{\psi_\eps(t) }_{H^2(\Gamma)}+ 
	\norma{\deln \phi_\eps (t)}_{\HG}
	\notag \\
	&\quad 
	\leq 
	 C\bigl( \Phi_\eps(t) + C_s \norma{ \phi_\eps (t) }_{H^s(\Omega)} \bigr)
	 + C_s \norma{ \phi_\eps (t) }_{H^s(\Omega)}
	 \notag \\
	&\quad 
	\leq 
	 \delta \norma{ \phi_\eps (t) }_{H^2(\Omega)} + C\, \Phi_\eps(t) + C \delta^{-1} \norma{ \phi_\eps (t) }_{V},
\end{align}
for all $t\in (0,T)$ and any $\delta\in(0,1)$. 
Eventually, from \eqref{pier9a} \eqref{pier21}, it follows that 
\begin{align} 
\label{pier22}
\norma{ \phi_\eps}_{\L2 {H^2(\Omega)}} +
	\norma{\psi_\eps}_{\L2 {H^2(\Gamma)}}+ 
	\norma{\deln \phi_\eps}_{\L2 {\HG}} \leq C
\end{align}
and consequently, recalling \eqref{pier16} and \eqref{pier19}, we also have  
\begin{align} 
\label{pier23}
\norma{\mu_\eps}_{\L2 V} +
	\norma{\theta_\eps}_{\L2{V_\Gamma}} +
	\norma{\b_\eps(\phi_\eps(t))}_{\L2 H} +
	\norma{\b_{\Gamma,\eps}(\psi_\eps)}_{\L2{\HG}}\leq 
	 C. 
\end{align}

\subsection{Passage to the limit and conclusion of the proof}
\label{SEC:YOS:epstozero}

The final step consists in passing to the limit as $\eps$ is sent to zero. 
As the line of argument resembles the one presented in Section \ref{SSSEC:CONV}, we proceed rather quickly just pointing out the main points and differences.

Owing to the above uniform estimates and to standard compactness results, we obtain that
there exist a subsequence of $\eps$ and a seventuple of limits $$(\bv^*, \phi^*, \xi^*, \mu^*, \psi^*, \xi_{\Gamma}^*, \theta^*)$$ such that, as $\eps \to 0$, 
\begin{alignat*}{2}
    \bv_\eps  & \to \bv^* \quad && \text{weakly  in $\L2 {\Vs} $,}
    \\ & {} && \quad \text{strongly in $\C0 {\Hx s }$, and $\aeQ$,}
    \\
    \bv_\eps|_\Gamma  & \to \bv^*|_\Gamma \quad && \text{weakly  in $\L2 {\HHH_\Gamma}$,}
    \\ & {} && \quad \text{strongly in $\C0 { H^s(\Gamma)}$, and $\aeS$,}
    \\
    \phi_\eps
    &\to \phi^*
    &&\text{weakly-$^*$ in $\L\infty V$, weakly in $\H1 {V^*}$}, \notag\\
    &&&\quad\text{strongly in $C^0([0,T];H^s(\Omega))$, and \aeQ},
    \\
    \psi_\eps 
    &\to \psi^*
    &&\text{weakly-$^*$ in $\L\infty \VG$, weakly in $\H1 {\VG^*}$}, \notag\\
    &&&\quad\text{strongly in $C^0([0,T];H^s(\Gamma))$, and \aeS},
    \\
     \beps(\phi_\eps) & \to \xi^*
     \quad &&\text{weakly in $\L2 {{ H}}$,}
     \\
     \beta_{\Gamma,\eps}(\psi_\eps )
     & \to \xi_{\Gamma}^*
     \quad &&\text{weakly in $\L2 {{ H_\Gamma}}$,}
    \\
    \mu_\eps   & \to \mu^* \quad &&\text{weakly in $\L2 V$,}
    \\
    \th_\eps   & \to \th^* \quad && \text{weakly in $\L2 {\VG}$,}
\end{alignat*}
for all $s \in [0,1)$. In the case $K=0$, we further infer from \eqref{pier22} the convergences
\begin{alignat*}{2}
    \phi_\eps 
    &\to \phi^*
    &&\quad\text{weakly in $\L2 {H^2(\Omega)}$},
    \\
    \psi_\eps 
    &\to \psi^*
    &&\quad\text{weakly in $\L2 {H^2(\Gamma)}$}. 
\end{alignat*}
Repeating the arguments employed in Section \ref{SSSEC:CONV}, we can easily show that the above weak and strong convergences suffice to pass to the limit in the variational formulation \eqref{wf:1*}--\eqref{wf:4*} written for $\beta=\beps$ and $\betaG= \beG$. Furthermore, the inclusions
\begin{align*}
    \text{$\xi^* \in \beta(\phi^*)$ $\aeQ$ \quad 
    and
    \quad 
    $\xi_\Gamma^* \in \betaG(\psi^*)$ $\aeS$}
\end{align*}
follow directly from the maximality of the monotone operators $\beta$ and $\beta_\Gamma$, and the facts that
$$ \lim_{\eps\to 0} \int_0^T\!\!\int_\Omega \beta_\eps(\phi_\eps) \phi_\eps = \int_0^T\!\!\int_\Omega \xi^* \phi^*, \quad \lim_{\eps\to 0} \int_0^T\!\!\int_\Gamma \beta_{\Gamma,\eps}(\psi_\eps ) \psi_\eps = \int_0^T\!\!\int_\Gamma \xi_\Gamma^* \psi^*$$
(see, e.g., \cite[Prop.~1.1, p.~42]{Bar}). 
Due to the aforementioned strong convergences of $\phi_\eps$ and $\psi_\eps$, it is straightforward to check that condition (iii) of Definition~\ref{DEF:WS:SING} is fulfilled. Moreover, condition (iv) of Definition~\ref{DEF:WS:SING} can be established by proceeding analogously as in Subection~\ref{SSSEC:CONV}.

Finally, if the domain is of class $C^2$, we need to establish the higher regularity properties of the phase-fields $\phi_*$ and $\psi_*$. In the case $K=0$, this directly follows from the above convergences.
In the case $K>0$, these properties can be proved as in Subsection~\ref{SSSEC:REG} by taking advantage of the regularities $\L2 {{ H}}$ for $\xi$ and $\L2 {{ H_\Gamma}}$ for $\xi_\Gamma$. This concludes the proof of Theorem~\ref{THM:WEAK:SING}.

\section*{Appendix: Some calculus for bulk-surface function spaces}
\addcontentsline{toc}{section}{Appendix: Some calculus for bulk-surface function spaces}
\setcounter{theorem}{0}
\setcounter{equation}{0}
\renewcommand{\thetheorem}{A.\arabic{theorem}}
\renewcommand{\theequation}{A.\arabic{equation}}

\begin{proposition} \label{PROP:A}
    Let $T>0$ and $K\ge 0$ be arbitrary. 
    \begin{enumerate}[label=\textnormal{(\alph*)}, ref = \textnormal{(\alph*)}]
        \item \label{CR1} 
        Let $(u,v) \in \L2{\CV_K}$ and suppose that the weak time derivative satisfies $(\delt u,\delt v) \in \L2{\CV_K^*}$. Then, the continuity property $(u,v)\in C^0([0,T];\CH)$ holds, the mapping 
        $$t\mapsto \norma{(u,v)(t)}_{\CH}^2 = \norma{u(t)}_H^2 + \norma{v(t)}_{\HG}^2 $$ 
        is absolutely continuous in $[0,T]$, and the chain rule formula
        \begin{align}
        \label{EQ:CR1}
            \frac{\mathrm d}{\mathrm dt} \Big[ \norma{u(t)}_H^2 + \norma{v(t)}_{\HG}^2  \Big]
            = 2 \bigang{(\delt u,\delt v)(t)}{(u,v)(t)}_{\CV_K}
        \end{align}
        holds for almost all $t\in [0,T]$.
        \item \label{CR2}  
        Let $(u,v) \in L^2\big(0,T;H^3(\Omega)\times H^3(\Gamma) \big)$
        with $K \partial_n u = v - u$ \aeS,
        and suppose that their weak time derivative satisfies $(\delt u,\delt v) \in \L2{{\CV}^*}$.
        Then, the continuity property $(u,v)\in C^0([0,T];\CV_K)$ holds, the mapping 
        $$t\mapsto \norma{\nabla u(t)}_\HH^2 + \norma{\nabla_\Gamma v(t)}_{\HHG}^2 + \sigma(K) \norma{v(t)-u(t)}^2_{\HG} $$ 
        is absolutely continuous in $[0,T]$, and the chain rule formula
        \begin{align}
        \label{EQ:CR2}
            &\frac{\mathrm d}{\mathrm dt} \Big[ \norma{\nabla u(t)}_\HH^2 + \norma{\nabla_\Gamma v(t)}_{\HHG}^2 + \sigma(K) \norma{v(t)-u(t)}^2_{\HG} \Big]
            \notag\\
            &\quad= 2 \bigang{(\delt u,\delt v)(t)}{(-\Delta u, -\Delta_\Gamma v + \deln u)(t)}_{\CV}
        \end{align}
        holds for almost all $t\in [0,T]$.
    \end{enumerate}
\end{proposition}

\begin{proof}
    \textit{Proof of \ref{CR1}.} Since $\CV_K$ and $\CH$ are separable Hilbert spaces with compact embedding $\CV_K\emb \CH$ and continuous embedding
    $\CH \emb \CV^*_K$,
    the assertion follows directly from the Lions--Magenes lemma (see, e.g., \cite[Chapter~III, Lemma~1.2]{temam1979navier-stokes}).

    \textit{Proof of \ref{CR2}.} 
    We first fix $u$ and $v$ as arbitrary representatives of their respective equivalence class.
    Recall that due to \ref{CR1}, we have $u\in C^0([0,T];H)$ and $v\in C^0([0,T];\HG)$.
    We can thus extend the functions $u$ and $v$ onto $[-T,0]$ by defining 
    $u(t)$ and $v(t)$ by reflection for all $t<0$. 
    
    Let $\rho\in C^\infty_c(\R)$ be a nonnegative function with $\supp \rho \subset (0,1)$ and $\norma{\rho}_{L^1(\R)}=1$. For any $k\in\N$, we set
    \begin{align*}
        \rho_k(s) := k \rho(ks) \quad\text{for all $s\in\R$}.
    \end{align*}
    
    For any Banach space $X$ and any function $f\in L^2(-1,T;X)$, we define
    \begin{align*}
        f_k(t) := (\rho_k * f)(t)=\int_{t-\frac 1k}^t \rho_k(t-s)\, f(s) \, {\mathrm ds}
    \end{align*}
    for all $t\in [0,T]$ and all $k\in\N$. By this construction, we have $f_k\in C^\infty([0,T];X)$ with $f_k\to f$ strongly in $L^2(0,T;X)$ as $k\to\infty$.
    
    For any $k\in\N$, we now choose $X=H^3(\Omega)$ to define $u_k$ and $X=H^3(\Gamma)$ to define $v_k$ as described above.
    By this construction, it holds $\delt u_k = (\delt u)_k$ and $\delt \nabla u_k = \nabla \delt u_k$ {\aeQ} as well as $\delt v_k = (\delt v)_k$ and $\delt \nabla_\Gamma v_k = \nabla_\Gamma \delt v_k$ {\aeS} for all $k\in\N$. Moreover, as $k\to\infty$, we have
    \begin{alignat}{2}
        \label{CONV:A:0}
        u_k &\to u 
        &&\quad\text{strongly in $L^2(0,T;H^3(\Omega))$},
        \\
        \label{CONV:A:1}
        v_k &\to v
        &&\quad\text{strongly in $L^2(0,T;H^3(\Gamma))$},
        \\
        \label{CONV:A:2}
        (u_k,v_k) &\to (u,v)
        &&\quad\text{strongly in $L^2(0,T;\CV_K)$},
        \\
        \label{CONV:A:3}
        (\delt u_k,\delt v_k) &\to (\delt u,\delt v)
        &&\quad\text{strongly in $L^2(0,T;\CV^*)$}.
    \end{alignat}
    
    In the following, the letter $C$ will denote generic positive constants that are independent of $k$ and may change their value from line to line. 
    Now, for any $k\in\N$, we derive the identity
    \begin{align}
    & \non
    \frac{\mathrm d}{\mathrm dt} \Big[ \norma{\nabla u_k}_{\HH}^2 + \norma{\nabla_\Gamma v_k}_{\HHG}^2 
    + \sigma(K) \norma{v_k-u_k}^2_{\HG} \Big]
        \\  &  \quad    \label{EQ:A:2}
    = 2 \bigang{(\delt u_k,\delt v_k)}{(-\Delta u_k, -\Delta_\Gamma v_k + \deln u_k)}_{{\CV}}
    \end{align}
    in $[0,T]$ by differentiating under the integral sign, applying integration by parts, and employing the relation 
    \begin{align*}
        \sigma(K) (v_k - u_k) = 
        \begin{cases}
            0 &\text{if $K=0$},\\
            \deln u_k &\text{if $K>0$},
        \end{cases}
        \qquad\aeS.
    \end{align*}
    Let now $j,k\in\N$ be arbitrary. Proceeding as above, we calculate
    \begin{align}
        &\frac{\mathrm d}{\mathrm dt} \Big[ \norma{\nabla u_j - \nabla u_k}_{\HH}^2 
        + \norma{\nabla_\Gamma v_j - \nabla_\Gamma v_k}_{\HHG}^2
        + \sigma(K) \norma{(v_j-v_k) -(u_j-u_k)}^2_{\HG}
        \Big]
        \notag\\[1ex]
        & \quad = 2 \bigang{\big(\delt (u_j - u_k),\delt (v_j - v_k)\big)}{\big( -\Delta (u_j-u_k), -\Delta_\Gamma (v_j-v_k) + \deln (u_j-u_k) \big)}_{{\CV}}
        \notag\\[1ex]
        & \qquad \le C \Big( \norma{\big(\delt (u_j - u_k),\delt (v_j - v_k) \big)}_{\CV_*}^2 
        + \norma{u_j - u_k}_{H^3(\Omega)}^2 
        + \norma{v_j - v_k}_{H^3(\Gamma)}^2 \Big)
        \label{EQ:A:3}
    \end{align}
    in $[0,T]$.
    Here, we have used the embedding $H^3(\Omega)\emb H^2(\Gamma)$ resulting from the trace theorem, which yields
    $$\norma{\deln (u_j-u_k)}_{\VG} \le \norma{u_j-u_k}_{H^2(\Gamma)} \le C \norma{u_j-u_k}_{H^3(\Omega)}.$$
    Let now $s,t \in [0,T]$ be arbitrary with $s \leq t$. We then integrate inequality \eqref{EQ:A:3} with respect to time from $s$ to $t$. 
    This yields
    \begin{align}
    \label{EST:A:1}
        &\norma{(\nabla u_j - \nabla u_k)(t)}_{\HH}^2 
        + \norma{(\nabla_\Gamma v_j - \nabla_\Gamma v_k)(t)}_{\HHG}^2
        \notag\\[1ex]
        &\quad
        + \sigma(K) \norma{(v_j(t)-v_k(t)) -(u_j(t)-u_k(t))}^2_{\HG}
        \notag\\[1ex]
        &\leq
        \norma{(\nabla u_j - \nabla u_k)(s)}_{\HH}^2 
        + \norma{(\nabla_\Gamma v_j - \nabla_\Gamma v_k)(s)}_{\HHG}^2
        \\[1ex]\notag
        &\quad
        + \sigma(K) \norma{(v_j(s)-v_k(s)) -(u_j(s)-u_k(s))}^2_{\HG}
        \\\notag
        &\quad 
        + C \int_s^t \norma{\big(\delt (u_j - u_k),\delt (v_j - v_k) \big)}_{{\CV^*}}^2 
        + \norma{u_j - u_k}_{H^3(\Omega)}^2 
        + \norma{v_j - v_k}_{H^3(\Gamma)}^2 .
    \end{align}
    Since $(u_k,v_k) \to (u,v)$ strongly in $L^2\big(0,T;(H^3(\Omega)\times H^3(\Gamma))\big)$, we can fix $s\in [0,t]$ such that $(u_k,v_k)(s) \to (u,v)(s)$ strongly in $H^3(\Omega)\times H^3(\Gamma)$.
    Recalling the convergences \eqref{CONV:A:0}--\eqref{CONV:A:3}, we thus infer that the right-hand side in \eqref{EST:A:1} tends to zero as $j,k\to\infty$. Consequently, $(\nabla u_k)_{k\in\N}$ is a Cauchy sequence in $C^0([0,T];\HH)$ and $(\nabla v_k)_{k\in\N}$ is a Cauchy sequence in $C^0([0,T];\HHG)$. We thus conclude 
    \begin{alignat}{2}
        \label{CONV:A:4}
        \nabla u_k &\to \nabla u 
        &&\quad\text{strongly in $C^0([0,T];\HH)$},
        \\
        \label{CONV:A:5}
        \nabla_\Gamma v_k &\to \nabla_\Gamma u 
        &&\quad\text{strongly in $C^0([0,T];\HHG)$}
    \end{alignat}
    as $k\to\infty$. Together with \ref{CR1}, this proves
    \begin{align*}
        (u,v) \in C^0([0,T];\CV_K).
    \end{align*}
    Let now  $s,t\in[0,T]$ be arbitrary. Without loss of generality, we assume $s\le t$. Integrating \eqref{EQ:A:2} with respect to time from $s$ to $t$, we obtain
    \begin{align*}
        &\norma{\nabla u_k(t)}_{\HH}^2 + \norma{\nabla_\Gamma v_k(t)}_{\HHG}^2
        + \sigma(K) \norma{v_k(t)-u_k(t)}^2_{\HG}
        \\
        &\quad
        = \norma{\nabla u_k(s)}_{\HH}^2 + \norma{\nabla_\Gamma v_k(s)}_{\HHG}^2
        + \sigma(K) \norma{v_k(s)-u_k(s)}^2_{\HG}
        \\
        &\qquad
        + 2 \int_s^t \bigang{(\delt u_k,\delt v_k)}{(-\Delta u_k, -\Delta_\Gamma v_k + \deln u_k)}_{{\CV}}.
    \end{align*}
    Invoking the convergences \eqref{CONV:A:0}--\eqref{CONV:A:3}, \eqref{CONV:A:4} and \eqref{CONV:A:5}, we pass to the limit $k\to\infty$ in this identity. This yields
    \begin{align*}
        &\norma{\nabla u(t)}_{\HH}^2 + \norma{\nabla_\Gamma v(t)}_{\HHG}^2
        + \sigma(K) \norma{v(t)-u(t)}^2_{\HG}
        \\
        &\quad
        = \norma{\nabla u(s)}_{\HH}^2 + \norma{\nabla_\Gamma v(s)}_{\HHG}^2
        + \sigma(K) \norma{v(s)-u(s)}^2_{\HG}
        \\
        &\qquad
        + 2 \int_s^t \bigang{\big(\delt u,\delt v\big)}{\big(-\Delta u, -\Delta_\Gamma v + \deln u\big)}_{{\CV}}.
    \end{align*}
    As the integrand of the integral on the right-hand side belongs to $L^1(0,T)$, we conclude that the mapping $t\mapsto \norma{\nabla u(t)}_\HH^2 + \norma{\nabla_\Gamma v(t)}_{\HHG}^2 + \sigma(K) \norma{v(t)-u(t)}^2_{\HG}$ is absolutely continuous in $[0,T]$. It is thus differentiable almost everywhere in $[0,T]$ and its derivative satisfies the formula \eqref{EQ:CR2}. This verifies \ref{CR2} and thus, the proof is complete.  
\end{proof}

%


\section*{Acknowledgement} 
The authors would like to express their sincere gratitude to the anonymous referee for the detailed review and several valuable suggestions and comments which helped us to improve the contents and the quality of the paper.

This research has been performed within the framework of the MIUR-PRIN Grant
2020F3NCPX ``Mathematics for industry 4.0 (Math4I4)''. The present paper also benefits from the
support of the GNAMPA (Gruppo Nazionale per l’Analisi Matematica, la Probabilit\`a e le loro Applicazioni) of INdAM (Istituto Nazionale di Alta Matematica).

Andrea Signori has been supported by ``MUR GRANT Dipartimento di Eccellenza'' 2023-2027 and by the Alexander von Humboldt Foundation.

Moreover, Patrik Knopf was partially supported by the Deutsche Forschungsgemeinschaft (DFG, German Research Foundation): on the one hand by the DFG-project 524694286, and on the other hand by the RTG 2339 ``Interfaces, Complex Structures, and Singular Limits''.
Their support is gratefully acknowledged.

Parts of this work were done while Patrik Knopf was visiting the 
Dipartimento di Matematica ``F. Casorati'' of the Universit\`{a} di Pavia whose hospitality is gratefully appreciated.

Last but not least, we thank Jonas Stange for helping us to proofread the manuscript.

\section*{Data availability}
There is no additional data associated with this manuscript.

\section*{Conflict of interests}
The authors do not have any financial or non-financial interests that are directly or indirectly related to the work submitted for publication.


\footnotesize
\bibliographystyle{abbrv}
\bibliography{CKSS}

\begin{thebibliography}{10}

\bibitem{AbelsHabil}
H.~Abels.
\newblock \textit{Diffuse Interface Models for Two-Phase Flows of Viscous
  Incompressible Fluids}.
\newblock Habilitationsschrift, Max-Planck-Institut f\"ur Mathematik in den
  Naturwissenschaften, Leipzig,
  \url{https://files-www.mis.mpg.de/mpi-typo3/preprints/ln/lecturenote-3607.pdf},
  2007.

\bibitem{Abels2009}
H.~Abels.
\newblock On a diffuse interface model for two-phase flows of viscous,
  incompressible fluids with matched densities.
\newblock {\em Arch. Ration. Mech. Anal.}, 194(2):463--506, 2009.

\bibitem{Abels2012}
H.~Abels.
\newblock Strong well-posedness of a diffuse interface model for a viscous,
  quasi-incompressible two-phase flow.
\newblock {\em SIAM J. Math. Anal.}, 44(1):316--340, 2012.

\bibitem{abels2013existence}
H.~Abels, D.~Depner, and H.~Garcke.
\newblock Existence of weak solutions for a diffuse interface model for
  two-phase flows of incompressible fluids with different densities.
\newblock {\em Journal of Mathematical Fluid Mechanics}, 15(3):453--480, 2013.

\bibitem{abels2013incompressible}
H.~Abels, D.~Depner, and H.~Garcke.
\newblock On an incompressible {N}avier--{S}tokes/{C}ahn--{H}illiard system
  with degenerate mobility.
\newblock {\em Annales de l'Institut Henri Poincar{\'e} C}, 30(6):1175--1190,
  2013.

\bibitem{AbelsGarckeReview}
H.~Abels and H.~Garcke.
\newblock Weak solutions and diffuse interface models for incompressible
  two-phase flows.
\newblock In {\em Handbook of mathematical analysis in mechanics of viscous
  fluids}, pages 1267--1327. Springer, Cham, 2018.

\bibitem{Abels2023}
H.~Abels, H.~Garcke, and A.~Giorgini.
\newblock Global regularity and asymptotic stabilization for the incompressible
  {N}avier-{S}tokes-{C}ahn-{H}illiard model with unmatched densities.
\newblock {\em Math. Ann.}, 389(2):1267--1321, 2024.

\bibitem{AGG}
H.~Abels, H.~Garcke, and G.~Gr\"{u}n.
\newblock Thermodynamically consistent, frame indifferent diffuse interface
  models for incompressible two-phase flows with different densities.
\newblock {\em Math. Models Methods Appl. Sci.}, 22(3):1150013, 40, 2012.

\bibitem{AbelsWeber2021}
H.~Abels and J.~Weber.
\newblock Local well-posedness of a quasi-incompressible two-phase flow.
\newblock {\em J. Evol. Equ.}, 21(3):3477--3502, 2021.

\bibitem{Allaire1991}
G.~Allaire.
\newblock Homogenization of the {N}avier-{S}tokes equations with a slip
  boundary condition.
\newblock {\em Comm. Pure Appl. Math.}, 44(6):605--641, 1991.

\bibitem{Alt}
H.~W. Alt.
\newblock {\em {Linear Functional Analysis - An Application-Oriented
  Introduction}}.
\newblock Springer, London, 2016.

\bibitem{Bar}
V.~Barbu.
\newblock {\em Nonlinear semigroups and differential equations in {B}anach
  spaces}.
\newblock Editura Academiei Republicii Socialiste Rom\^{a}nia, Bucharest;
  Noordhoff International Publishing, Leiden, 1976.
\newblock Translated from the Romanian.

\bibitem{boyer1999}
F.~Boyer.
\newblock Mathematical study of multi-phase flow under shear through order
  parameter formulation.
\newblock {\em Asymptot. Anal.}, 20(2):175--212, 1999.

\bibitem{boyer2002theoretical}
F.~Boyer.
\newblock A theoretical and numerical model for the study of incompressible
  mixture flows.
\newblock {\em Computers \& fluids}, 31(1):41--68, 2002.

\bibitem{boyer_book}
F.~Boyer and P.~Fabrie.
\newblock {\em Mathematical tools for the study of the incompressible
  {N}avier-{S}tokes equations and related models}, volume 183 of {\em Applied
  Mathematical Sciences}.
\newblock Springer, New York, 2013.

\bibitem{brezis}
H.~Brezis.
\newblock {\em Op\'erateurs maximaux monotones et semi-groupes de contractions
  dans les espaces de Hilbert}.
\newblock Elsevier, 1973.

\bibitem{CaCo}
L.~Calatroni and P.~Colli.
\newblock Global solution to the {A}llen--{C}ahn equation with singular
  potentials and dynamic boundary conditions.
\newblock {\em Nonlinear Anal.}, 79:12--27, 2013.

\bibitem{Chen2019}
J.~Chen, S.~Sun, and X.~Wang.
\newblock Homogenization of two-phase fluid flow in porous media via volume
  averaging.
\newblock {\em J. Comput. Appl. Math.}, 353:265--282, 2019.

\bibitem{Colli2015}
P.~Colli and T.~Fukao.
\newblock Equation and dynamic boundary condition of {C}ahn--{H}illiard type
  with singular potentials.
\newblock {\em Nonlinear Anal.}, 127:413--433, 2015.

\bibitem{CF6}
P.~Colli and T.~Fukao.
\newblock Vanishing diffusion in a dynamic boundary condition for the
  {C}ahn--{H}illiard equation.
\newblock {\em NoDEA Nonlinear Differential Equations Appl.}, 27(6):Paper No.
  53, 27, 2020.

\bibitem{Colli2019}
P.~Colli, T.~Fukao, and K.~F. Lam.
\newblock On a coupled bulk-surface {A}llen--{C}ahn system with an affine
  linear transmission condition and its approximation by a {R}obin boundary
  condition.
\newblock {\em Nonlinear Anal.}, 184:116--147, 2019.

\bibitem{Colli2022}
P.~Colli, T.~Fukao, and L.~Scarpa.
\newblock The {C}ahn--{H}illiard equation with forward-backward dynamic
  boundary condition via vanishing viscosity.
\newblock {\em SIAM J. Math. Anal.}, 54(3):3292--3315, 2022.

\bibitem{Colli2022a}
P.~Colli, T.~Fukao, and L.~Scarpa.
\newblock A {C}ahn--{H}illiard system with forward-backward dynamic boundary
  condition and non-smooth potentials.
\newblock {\em J. Evol. Equ.}, 22(4):Paper No. 89, 31, 2022.

\bibitem{Colli2020}
P.~Colli, T.~Fukao, and H.~Wu.
\newblock On a transmission problem for equation and dynamic boundary condition
  of {C}ahn--{H}illiard type with nonsmooth potentials.
\newblock {\em Math. Nachr.}, 293(11):2051--2081, 2020.

\bibitem{CGSS1}
P.~Colli, G.~Gilardi, A.~Signori, and J.~Sprekels.
\newblock {C}ahn--{H}illiard--{B}rinkman model for tumor growth with possibly
  singular potentials.
\newblock {\em Nonlinearity}, 36(8):4470--4500, 2023.

\bibitem{CGS_dom}
P.~Colli, G.~Gilardi, and J.~Sprekels.
\newblock On the {C}ahn--{H}illiard equation with dynamic boundary conditions
  and a dominating boundary potential.
\newblock {\em J. Math. Anal. Appl.}, 419(2):972--994, 2014.

\bibitem{CGS2017}
P.~Colli, G.~Gilardi, and J.~Sprekels.
\newblock Global existence for a nonstandard viscous {C}ahn--{H}illiard system
  with dynamic boundary condition.
\newblock {\em SIAM J. Math. Anal.}, 49(3):1732--1760, 2017.

\bibitem{CGS2018b}
P.~Colli, G.~Gilardi, and J.~Sprekels.
\newblock On a {C}ahn--{H}illiard system with convection and dynamic boundary
  conditions.
\newblock {\em Ann. Mat. Pura Appl. (4)}, 197(5):1445--1475, 2018.

\bibitem{CGS2018a}
P.~Colli, G.~Gilardi, and J.~Sprekels.
\newblock On the longtime behavior of a viscous {C}ahn--{H}illiard system with
  convection and dynamic boundary conditions.
\newblock {\em J. Elliptic Parabol. Equ.}, 4(2):327--347, 2018.

\bibitem{CGS2018c}
P.~Colli, G.~Gilardi, and J.~Sprekels.
\newblock Optimal velocity control of a viscous {C}ahn--{H}illiard system with
  convection and dynamic boundary conditions.
\newblock {\em SIAM J. Control Optim.}, 56(3):1665--1691, 2018.

\bibitem{Dede2018}
L.~Ded\`e, H.~Garcke, and K.~F. Lam.
\newblock A {H}ele-{S}haw-{C}ahn-{H}illiard model for incompressible two-phase
  flows with different densities.
\newblock {\em J. Math. Fluid Mech.}, 20(2):531--567, 2018.

\bibitem{DiBenedetto}
E.~DiBenedetto.
\newblock {\em Real analysis}.
\newblock Birkh\"{a}user Advanced Texts: Basler Lehrb\"{u}cher. [Birkh\"{a}user
  Advanced Texts: Basel Textbooks]. Birkh\"{a}user Boston, Inc., Boston, MA,
  2002.

\bibitem{ding2007diffuse}
H.~Ding, P.~Spelt, and C.~Shu.
\newblock Diffuse interface model for incompressible two-phase flows with large
  density ratios.
\newblock {\em Journal of Computational Physics}, 226(2):2078--2095, 2007.

\bibitem{du2020phase}
Q.~Du and X.~Feng.
\newblock The phase field method for geometric moving interfaces and their
  numerical approximations.
\newblock {\em Handbook of Numerical Analysis}, 21:425--508, 2020.

\bibitem{Ebenbeck2019}
M.~Ebenbeck and H.~Garcke.
\newblock Analysis of a {C}ahn-{H}illiard-{B}rinkman model for tumour growth
  with chemotaxis.
\newblock {\em J. Differential Equations}, 266(9):5998--6036, 2019.

\bibitem{Ern}
A.~Ern and J.-L. Guermond.
\newblock {\em Theory and practice of finite elements}, volume 159 of {\em
  Applied Mathematical Sciences}.
\newblock Springer-Verlag, New York, 2004.

\bibitem{Feireisl2016}
E.~Feireisl, Y.~Namlyeyeva, and v.~Ne\v{c}asov\'{a}.
\newblock Homogenization of the evolutionary {N}avier-{S}tokes system.
\newblock {\em Manuscripta Math.}, 149(1-2):251--274, 2016.

\bibitem{Feng2018}
X.~Feng, J.~Kou, and S.~Sun.
\newblock A novel energy stable numerical scheme for
  {N}avier-{S}tokes-{C}ahn-{H}illiard two-phase flow model with variable
  densities and viscosities.
\newblock In {\em Computational science---{ICCS} 2018. {P}art {III}}, volume
  10862 of {\em Lecture Notes in Comput. Sci.}, pages 113--128. Springer, Cham,
  2018.

\bibitem{freistuhler2017phase}
H.~Freist{\"u}hler and M.~Kotschote.
\newblock Phase-field and {K}orteweg-type models for the time-dependent flow of
  compressible two-phase fluids.
\newblock {\em Archive for Rational Mechanics and Analysis}, 224(1):1--20,
  2017.

\bibitem{GalGrasselli2010}
C.~Gal and M.~Grasselli.
\newblock Asymptotic behavior of a {C}ahn--{H}illiard--{N}avier--{S}tokes
  system in 2{D}.
\newblock {\em Ann. Inst. H. Poincar\'{e} C Anal. Non Lin\'{e}aire},
  27(1):401--436, 2010.

\bibitem{GGM2016}
C.~G. Gal, M.~Grasselli, and A.~Miranville.
\newblock {C}ahn--{H}illiard--{N}avier--{S}tokes systems with moving contact
  lines.
\newblock {\em Calc. Var. Partial Differential Equations}, 55(3):Art. 50, 47,
  2016.

\bibitem{GGW2019}
C.~G. Gal, M.~Grasselli, and H.~Wu.
\newblock Global weak solutions to a diffuse interface model for incompressible
  two-phase flows with moving contact lines and different densities.
\newblock {\em Arch. Ration. Mech. Anal.}, 234(1):1--56, 2019.

\bibitem{Garcke2020}
H.~Garcke and P.~Knopf.
\newblock Weak solutions of the {C}ahn--{H}illiard system with dynamic boundary
  conditions: a gradient flow approach.
\newblock {\em SIAM J. Math. Anal.}, 52(1):340--369, 2020.

\bibitem{Garcke2022}
H.~Garcke, P.~Knopf, and S.~Yayla.
\newblock Long-time dynamics of the {C}ahn--{H}illiard equation with kinetic
  rate dependent dynamic boundary conditions.
\newblock {\em Nonlinear Anal.}, 215:Paper No. 112619, 44, 2022.

\bibitem{giga2017variational}
M.-H. Giga, A.~Kirshtein, and C.~Liu.
\newblock Variational modeling and complex fluids.
\newblock {\em Handbook of mathematical analysis in mechanics of viscous
  fluids}, pages 1--41, 2017.

\bibitem{GMS2009}
G.~Gilardi, A.~Miranville, and G.~Schimperna.
\newblock On the {C}ahn--{H}illiard equation with irregular potentials and
  dynamic boundary conditions.
\newblock {\em Commun. Pure Appl. Anal.}, 8(3):881--912, 2009.

\bibitem{GMS2010}
G.~Gilardi, A.~Miranville, and G.~Schimperna.
\newblock Long time behavior of the {C}ahn--{H}illiard equation with irregular
  potentials and dynamic boundary conditions.
\newblock {\em Chinese Ann. Math. Ser. B}, 31(5):679--712, 2010.

\bibitem{giorgini2021well}
A.~Giorgini.
\newblock Well-posedness of the two-dimensional {A}bels--{G}arcke--{G}r{\"u}n
  model for two-phase flows with unmatched densities.
\newblock {\em Calculus of Variations and Partial Differential Equations},
  60(3):1--40, 2021.

\bibitem{giorgini2022-3D}
A.~Giorgini.
\newblock Existence and stability of strong solutions to the
  {A}bels--{G}arcke--{G}r{\"u}n model in three dimensions.
\newblock {\em Interfaces Free Bound.}, 24:565–608, 2022.

\bibitem{Giorgini2018}
A.~Giorgini, M.~Grasselli, and H.~Wu.
\newblock The {C}ahn-{H}illiard-{H}ele-{S}haw system with singular potential.
\newblock {\em Ann. Inst. H. Poincar\'{e} C Anal. Non Lin\'{e}aire},
  35(4):1079--1118, 2018.

\bibitem{Giorgini-Knopf}
A.~Giorgini and P.~Knopf.
\newblock {Two-Phase Flows with Bulk-Surface Interaction: Thermodynamically
  Consistent Navier--Stokes--Cahn--Hilliard Models with Dynamic Boundary
  Conditions}.
\newblock {\em J. Math. Fluid Mech.}, 25(65):1--44, 2022.

\bibitem{GMT2019}
A.~Giorgini, A.~Miranville, and R.~Temam.
\newblock Uniqueness and regularity for the
  {N}avier--{S}tokes--{C}ahn--{H}illiard system.
\newblock {\em SIAM J. Math. Anal.}, 51(3):2535--2574, 2019.

\bibitem{Giunti2019}
A.~Giunti and R.~M. H\"{o}fer.
\newblock Homogenisation for the {S}tokes equations in randomly perforated
  domains under almost minimal assumptions on the size of the holes.
\newblock {\em Ann. Inst. H. Poincar\'{e} C Anal. Non Lin\'{e}aire},
  36(7):1829--1868, 2019.

\bibitem{GMS}
G.~R. Goldstein, A.~Miranville, and G.~Schimperna.
\newblock A {C}ahn--{H}illiard model in a domain with non-permeable walls.
\newblock {\em Phys. D}, 240(8):754--766, 2011.

\bibitem{GurtinPolignoneVinals}
M.~Gurtin, D.~Polignone, and J.~Vi\~{n}als.
\newblock Two-phase binary fluids and immiscible fluids described by an order
  parameter.
\newblock {\em Math. Models Methods Appl. Sci.}, 6(6):815--831, 1996.

\bibitem{heida2012development}
M.~Heida, J.~M{\'a}lek, and K.~Rajagopal.
\newblock On the development and generalizations of {C}ahn--{H}illiard
  equations within a thermodynamic framework.
\newblock {\em Zeitschrift f{\"u}r angewandte Mathematik und Physik},
  63(1):145--169, 2012.

\bibitem{Hoefer2023}
R.~M. H\"{o}fer.
\newblock Homogenization of the {N}avier-{S}tokes equations in perforated
  domains in the inviscid limit.
\newblock {\em Nonlinearity}, 36(11):6019--6046, 2023.

\bibitem{HH}
P.~C. Hohenberg and B.~I. Halperin.
\newblock Theory of dynamic critical phenomena.
\newblock {\em Rev. Mod. Phys.}, 49:435--479, 1977.

\bibitem{Knopf2020}
P.~Knopf and K.~Lam.
\newblock {Convergence of a Robin boundary approximation for a Cahn--Hilliard
  system with dynamic boundary conditions}.
\newblock {\em Nonlinearity}, 33(8):4191--4235, 2020.

\bibitem{Knopf2021a}
P.~Knopf, K.~Lam, C.~Liu, and S.~Metzger.
\newblock Phase-field dynamics with transfer of materials: the
  {C}ahn--{H}illiard equation with reaction rate dependent dynamic boundary
  conditions.
\newblock {\em ESAIM Math. Model. Numer. Anal.}, 55(1):229--282, 2021.

\bibitem{knopf-liu}
P.~Knopf and C.~Liu.
\newblock On second-order and fourth-order elliptic systems consisting of bulk
  and surface {PDE}s: well-posedness, regularity theory and eigenvalue
  problems.
\newblock {\em Interfaces Free Bound.}, 23(4):507--533, 2021.

\bibitem{Knopf2021}
P.~Knopf and A.~Signori.
\newblock On the nonlocal {C}ahn--{H}illiard equation with nonlocal dynamic
  boundary condition and boundary penalization.
\newblock {\em J. Differential Equations}, 280:236--291, 2021.

\bibitem{Knopf2022}
P.~Knopf and A.~Signori.
\newblock Existence of weak solutions to multiphase {C}ahn-{H}illiard-{D}arcy
  and {C}ahn-{H}illiard-{B}rinkman models for stratified tumor growth with
  chemotaxis and general source terms.
\newblock {\em Comm. Partial Differential Equations}, 47(2):233--278, 2022.

\bibitem{Liu2019}
C.~Liu and H.~Wu.
\newblock {An energetic variational approach for the {C}ahn--{H}illiard
  equation with dynamic boundary condition: model derivation and mathematical
  analysis}.
\newblock {\em Arch. Ration. Mech. Anal.}, 233(1):167--247, 2019.

\bibitem{lowengrub1998quasi}
J.~Lowengrub and L.~Truskinovsky.
\newblock {Q}uasi-incompressible {C}ahn--{H}illiard fluids and topological
  transitions.
\newblock {\em Proceedings of the Royal Society of London. Series A:
  Mathematical, Physical and Engineering Sciences}, 454(1978):2617--2654, 1998.

\bibitem{Masmoudi2002}
N.~Masmoudi.
\newblock Homogenization of the compressible {N}avier-{S}tokes equations in a
  porous medium.
\newblock {\em ESAIM Control Optim. Calc. Var.}, 8:885--906, 2002.
\newblock A tribute to J. L. Lions.

\bibitem{Miranville2020}
A.~Miranville and H.~Wu.
\newblock Long-time behavior of the {C}ahn--{H}illiard equation with dynamic
  boundary condition.
\newblock {\em J. Elliptic Parabol. Equ.}, 6(1):283--309, 2020.

\bibitem{MZ}
A.~Miranville and S.~Zelik.
\newblock Robust exponential attractors for {C}ahn--{H}illiard type equations
  with singular potentials.
\newblock {\em Math. Methods Appl. Sci.}, 27(5):545--582, 2004.

\bibitem{Necas}
J.~Ne\v{c}as and I.~Hlav\'{a}\v{c}ek.
\newblock {\em Mathematical theory of elastic and elasto-plastic bodies: an
  introduction}, volume~3 of {\em Studies in Applied Mechanics}.
\newblock Elsevier Scientific Publishing Co., Amsterdam-New York, 1980.

\bibitem{pruss2016moving}
J.~Pr{\"u}ss and G.~Simonett.
\newblock {\em Moving interfaces and quasilinear parabolic evolution
  equations}, volume 105.
\newblock Springer, 2016.

\bibitem{Qian-Wang-Sheng}
T.~Qian, X.-P. Wang, and P.~Sheng.
\newblock A variational approach to moving contact line hydrodynamics.
\newblock {\em J. Fluid Mech.}, 564:333--360, 2006.

\bibitem{Rohde}
C.~Rohde and L.~von Wolff.
\newblock Homogenization of nonlocal {N}avier-{S}tokes-{K}orteweg equations for
  compressible liquid-vapor flow in porous media.
\newblock {\em SIAM J. Math. Anal.}, 52(6):6155--6179, 2020.

\bibitem{Schreyer}
L.~Schreyer and Z.~Hilliard.
\newblock {Derivation of generalized Cahn--Hilliard equation for two-phase flow
  in porous media using hybrid mixture theory}.
\newblock {\em Advances in Water Resources}, 149:103839, 2021.

\bibitem{shen2013mass}
J.~Shen, X.~Yang, and Q.~Wang.
\newblock Mass and volume conservation in phase field models for binary fluids.
\newblock {\em Communications in Computational Physics}, 13(4):1045--1065,
  2013.

\bibitem{shokrpour2018diffuse}
M.~Shokrpour~Roudbari, G.~{\c{S}}im{\c{s}}ek, E.~H. van Brummelen, and
  K.~van~der Zee.
\newblock Diffuse-interface two-phase flow models with different densities: A
  new quasi-incompressible form and a linear energy-stable method.
\newblock {\em Mathematical Models and Methods in Applied Sciences},
  28(04):733--770, 2018.

\bibitem{temam1979navier-stokes}
R.~Temam.
\newblock {\em Navier-{S}tokes equations}, volume~2 of {\em Studies in
  Mathematics and its Applications}.
\newblock North-Holland Publishing Co., Amsterdam-New York, revised edition,
  1979.
\newblock Theory and numerical analysis, With an appendix by F. Thomasset.

\bibitem{Triebel}
H.~Triebel.
\newblock {\em Interpolation theory, function spaces, differential operators},
  volume~18 of {\em North-Holland Mathematical Library}.
\newblock North-Holland Publishing Co., Amsterdam-New York, 1978.

\bibitem{Triebel2}
H.~Triebel.
\newblock {\em Theory of function spaces. {II}}, volume~84 of {\em Monographs
  in Mathematics}.
\newblock Birkh\"{a}user Verlag, Basel, 1992.

\bibitem{Wu2022}
H.~Wu.
\newblock A review on the {C}ahn--{H}illiard equation: classical results and
  recent advances in dynamic boundary conditions.
\newblock {\em Electron. Res. Arch.}, 30(8):2788--2832, 2022.

\bibitem{Zhang2002}
D.~Zhang.
\newblock 6 -- {T}wo-phase flow.
\newblock In D.~Zhang, editor, {\em Stochastic Methods for Flow in Porous
  Media}, pages 262--296. Academic Press, San Diego, 2002.

\end{thebibliography}

\end{document}